%% file: ex_article.tex
\documentclass[review,hidelinks,onefignum,onetabnum]{siamart220329}

\input{ex_shared}

\ifpdf
\hypersetup{
  pdftitle={Dual-Valued Functions of Dual Matrices with Applications in Causal Emergence},
  pdfauthor={Tong Wei, Weiyang Ding, and Yimin Wei}
}
\fi

\begin{document}

\maketitle

\begin{abstract}
Dual continuation, an innovative insight into extending the real-valued functions of real matrices to the dual-valued functions of dual matrices with a foundation of the G\^ateaux derivative, is proposed.
Theoretically, the general forms of dual-valued vector and matrix norms, remaining properties in the real field, are provided.
In particular, we focus on the dual-valued vector $p$-norm $(1\!\leq\! p\!\leq\!\infty)$ and the unitarily invariant dual-valued Ky Fan $p$-$k$-norm $(1\!\leq\! p\!\leq\!\infty)$.
The equivalence between the dual-valued Ky Fan $p$-$k$-norm and the dual-valued vector $p$-norm of the first $k$ singular values of the dual matrix is then demonstrated.
Practically, we define the dual transitional probability matrix (DTPM), as well as its dual-valued effective information (${\rm{EI_d}}$).
Additionally, we elucidate the correlation between the ${\rm{EI_d}}$, the dual-valued Schatten $p$-norm, and the dynamical reversibility of a DTPM.
Through numerical experiments on a dumbbell Markov chain, our findings indicate that the value of $k$, corresponding to the maximum value of the infinitesimal part of the dual-valued Ky Fan $p$-$k$-norm by adjusting $p$ in the interval $[1,2)$, characterizes the optimal classification number of the system for the occurrence of the causal emergence.
\end{abstract}



\section{Introduction}

The dual number, initially proposed by Clifford \cite{clifford1871preliminary} in 1873, was then built upon by Study \cite{study1903geometrie} in 1903 to represent the dual angle.
In light of the powerful role that dual algebra plays in the kinematic analysis of spatial mechanisms, an increasing number of researchers have concentrated on the distinctive attributes of dual algebra, with a particular focus on the dual-valued functions.

The preceding work generalized the differentiable functions in the real field and holomorphic functions in the complex field to dual-real-valued and dual-complex-valued functions, respectively.
Specifically, Kramer (1930) \cite{kramer1930polygenic} defined the polygenic functions of dual real numbers.
Gu and Luh (1987) \cite{gu1987dual} employed the Taylor series expansion to derive the general form for functions of dual numbers.
Pennestr{\`i} and Stefanelli (2007) \cite{pennestri2007linear} outlined Gaussian, polar, and exponential representations of dual real numbers, as well as the trigonometric functions for dual angles.
Furthermore, Messelmi (2015) \cite{messelmi2015dual} developed the holomorphic dual-complex-valued functions and presented the exponential, logarithmic, trigonometric, and hyperbolic functions of dual complex numbers.
G{\"u}ng{\"o}r and Tetik (2019) \cite{gungor2019moivre} advanced the De-Moivre and Euler formulae for dual-complex numbers.

Recently, research into the norms of dual vectors and dual matrices has yielded some promising results.
Qi \textit{et al.} (2022) \cite{qi2022low,total_order_qi} defined the magnitude of dual complex numbers, and formulated the 1-norm, 2-norm, and $\infty$-norm of dual complex vectors, as well as the Frobenius norm of dual complex matrices.
Additionally, Miao and Huang (2023) \cite{miao2023norms} introduced the concept of the $p$-norm of a dual complex vector with $p$ being an arbitrary positive integer, along with the operator $p$-norm of a dual complex matrix induced by this $p$-norm.
Besides, the expressions of the operator 1-norm, 2-norm, and $\infty$-norm for dual complex matrices were derived.

Despite the graceful approach \cite{gu1987dual} for constructing the dual continuation of real-valued functions, it was limited to differentiable real-valued functions of real numbers.
Hence, we aim to investigate how to extend a common real-valued function to a dual-valued one in a ``continuous'' manner, while maintaining the properties established on the real space.
An ingenious idea is to employ the G\^ateaux derivative, proposed by G\^ateaux \cite{Gateaux_derivative} in 1913, as the infinitesimal part.
Thus, the expressions of dual-valued vector and matrix norms are derived and their rationality is demonstrated.

The highlight of the applications of the dual-valued functions lies in the interconnection between the dual-valued Ky Fan $p$-$k$-norm and causal emergence (CE) \cite{hoel2017map,hoel2013quantifying} in dynamical systems.
For the conventional methodologies utilized to quantify the extent of CE, it is necessary to traverse all the coarse-graining methods \cite{yang2024finding,yuan2024emergence} or to identify a clear cut-off in the spectrum of singular values of the transitional probability matrix (TPM) \cite{zhang2024dynamical}.
Nevertheless, our method addresses these two issues. 
Particularly, the infinitesimal part of the dual-valued Ky Fan $p$-$k$-norm plays an essential role in identifying the optimal classification number of a system that occurs CE.

The remainder of this paper is structured as follows. 
\Cref{sec2: prelimi} begins by laying out
some preliminary properties of dual matrices and useful notations.
\Cref{sec: dual continuation} proposes the definition and characteristics of the dual continuation of a real-valued function, derived from the G\^ateaux derivative.
\Cref{sec3: dual vector norms} focuses on one of the core contributions to the general form of dual-valued vector norms and illustrates concrete expressions of dual-valued vector $p$-norms $(1 \!\leq\! p \!\leq\! \infty)$.
The other principal development, the general form of dual-valued matrix norms, is developed in \cref{sec4: dual matrix norms}.
Besides, the specific expressions of unitarily invariant dual-valued matrix norms, dual-valued operator norms, and other dual-valued functions are discussed.
In \cref{sec5: causal emergence}, the dual transitional probability matrix (DTPM) is defined, as well as its dual-valued effective information (${\rm{EI_d}}$) and the concept of dynamical reversibility.
Theoretical consistency between the dynamical reversibility of the DTPM, the maximum point of ${\rm{EI_d}}$, and the maximum point of the dual-valued Schatten $p$-norm $(1\!\leq\! p<2)$ enables the capture of the optimal classification number occurring CE in the classical dumbbell Markov chain. 
In conclusion, \cref{sec7:conclusions} presents an overview of our findings and outlines potential avenues for future research.

\section{Preliminaries}\label{sec2: prelimi}
This section presents the fundamental properties of dual matrices and introduces relevant notations.

Let $\mathbb{DR}$ denote the set of dual (real) numbers. 
A dual number $p$ is denoted by $p=p_{\rm{s}}+p_{\rm{i}}\epsilon$, where $p_{\rm{s}},p_{\rm{i}}\in\mathbb{R}$ and $\epsilon$ is the dual infinitesimal unit satisfying $\epsilon\neq0, 0\epsilon = \epsilon0 = 0, 1\epsilon = \epsilon1 = \epsilon, \epsilon^2 = 0$.
We refer respectively to $p_{\rm{s}}$ as the standard part and $p_{\rm{i}}$ as the infinitesimal part of $p$.
Similarly, $\mathbb{DR}^n$ and $\mathbb{DR}^{m\times n}$ separately denote the set of dual (real) vectors of size $n$-by-1 and dual (real) matrices of size $m$-by-$n$.
In addition, the total order of dual numbers is introduced as follows.

\begin{definition}
[\textbf{Total Order of Dual Numbers} \cite{total_order_qi}]
\label{def: total order}
Given dual numbers $p = p_{\rm{s}} + p_{\rm{i}}\epsilon$ and $q = q_{\rm{s}} + q_{\rm{i}}\epsilon$. 
We say $p < q$, if either $p_{\rm{s}} < q_{\rm{s}}$, or $p_{\rm{s}} = q_{\rm{s}}$ and $p_{\rm{i}} < q_{\rm{i}}$.    
\end{definition}

Given a dual matrix $\bm{A}=\bm{A}_{\rm{s}}+\bm{A}_{\rm{i}}\epsilon\in\mathbb{DR}^{n\times n}$.
Then $\bm{A}$ is nonsingular if $\bm{AB}=\bm{BA}=\bm{I}_n$ for some $\bm{B}\in\mathbb{DR}^{n\times n}$.
Moreover, such $\bm{B}$ is unique and denoted by $\bm{A}^{-1}$ satisfying $\bm{A}^{-1}=\bm{A}_{\rm{s}}^{-1}-\bm{A}_{\rm{s}}^{-1}\bm{A}_{\rm{i}}\bm{A}_{\rm{s}}^{-1}\epsilon$.
In addition, $\bm{A}$ is orthogonal if $\bm{AA}^\top=\bm{A}^\top\bm{A}=\bm{I}_n$, which implies $\bm{A}_{\rm{s}}\bm{A}_{\rm{s}}^\top=\bm{A}_{\rm{s}}^\top\bm{A}_{\rm{s}}=\bm{I}_n$, $\bm{A}_{\rm{s}}\bm{A}_{\rm{i}}^\top+\bm{A}_{\rm{i}}\bm{A}_{\rm{s}}^\top=\bm{O}_n$ and $\bm{A}_{\rm{s}}^\top\bm{A}_{\rm{i}}+\bm{A}_{\rm{i}}^\top\bm{A}_{\rm{s}}=\bm{O}_n$.
Hence, $\bm{A}_{\rm{s}}$ is orthogonal and both $\bm{A}_{\rm{s}}^\top\bm{A}_{\rm{i}}$ and $\bm{A}_{\rm{s}}\bm{A}_{\rm{i}}^\top$ are skew-symmetric.
Besides, a dual matrix $\bm{C}=\bm{C}_{\rm{s}}+\bm{C}_{\rm{i}}\epsilon\in\mathbb{DR}^{m\times n}$ $(m\geq n)$ is called to have orthogonal columns if $\bm{C}^\top\bm{C}=\bm{I}_n$.
That is, $\bm{C}_{\rm{s}}^\top\bm{C}_{\rm{s}}=\bm{I}_n$ and $\bm{C}_{\rm{s}}^\top\bm{C}_{\rm{i}}$ is skew-symmetric. 

For a square matrix $\bm{X}$, ${\rm{sym}}(\bm{X})$ stands for $\frac{1}{2}(\bm{X}+\bm{X}^\top)$, ${\rm{Diag}}(\bm{X})$ for a matrix with only the diagonal elements of $\bm{X}$ and ${\rm{diag}}(\bm{X})$ for a vector whose elements are the diagonal elements of $\bm{X}$.
Besides, ${\rm{diag}}(\bm{X}_1,\cdots,\bm{X}_n)$ represents the block diagonal matrix whose diagonal blocks are square $\bm{X}_1,\cdots,\bm{X}_n$.
Moreover, 0, $\bm{0}$, $\bm{O}$, and $\bm{I}$ denote the number zero, an all-zero vector, an all-zero matrix, and an identity matrix of the appropriate size, respectively.
Operator $\odot$ represents the Hadamard product (element-wise product).
Additionally, the sign function of the real number $x$ is defined as ${\rm{sign}}(x)=1$ if $x>0$, ${\rm{sign}}(x)=-1$ if $x<0$ and ${\rm{sign}}(x)=0$ otherwise.

Note that throughout the paper, the use of lowercase letters, bold lowercase letters, and bold capital letters serves to distinguish between scalars, vectors, and matrices in the real field or dual algebra, respectively. 
Symbols with subscripts ``${\rm{s}}$'' and ``${\rm{i}}$'' are separately regarded as standard parts and infinitesimal parts in dual algebra.

\section{Dual Continuation}\label{sec: dual continuation}
Concerning dual-valued functions of dual numbers, Gu and Luh (1987) \cite{gu1987dual} proposed that the dual continuation $\Tilde{\phi}:\mathbb{DR}\rightarrow\mathbb{DR}$ of a given real-valued function $\phi:\mathbb{R}\rightarrow\mathbb{R}$ is expressed as 
\begin{align}
\Tilde{\phi}(a+b\epsilon)=\phi(a)+b\phi^{'}(a)\epsilon\label{eq:f(a+b ep)=f(a)+f'(a)b ep},   
\end{align}
where $a,b\in\mathbb{R}$ and $\phi^{'}$ represents the first-order derivative of $\phi$.

Nevertheless, the formula \cref{eq:f(a+b ep)=f(a)+f'(a)b ep} applies only to differentiable real-valued functions of real numbers.
This motivates us to investigate an innovative methodology for the dual continuation of common real-valued functions involving real vectors and matrices.
Specifically, our focus is on extending a real-valued function $\phi:\mathcal{X}\rightarrow\mathbb{R}$ on the real space $\mathcal{X}$ (representing $\mathbb{R}$, $\mathbb{R}^n$ or $\mathbb{R}^{m\times n}$) to a dual-valued one $\Tilde{\phi}:\mathbb{D}\mathcal{X}\rightarrow\mathbb{DR}$ in a ``continuous'' manner such that
\begin{align*}
\lim_{b\rightarrow0}\Tilde{\phi}(a+b\epsilon)=\phi(a),
\end{align*}
where $a,b\in\mathcal{X}$, namely $a+b\epsilon\in\mathbb{D}\mathcal{X}$.
To achieve this, the G\^ateaux derivative, as defined below, is employed as the infinitesimal part of the dual continuation.

\begin{definition}[\textbf{G\^ateaux Derivative} \cite{Gateaux_derivative}]
\label{direction_gradient}
Let $X$ and $Y$ be locally convex topological spaces, $U\subset X$ be an open subset, and $\phi:X\rightarrow Y$. 
If the limit 
\begin{align}
{\mathcal{D}}_{u}\phi(x) = \lim_{t\downarrow 0}\frac{\phi(x+tu)-\phi(x)}{t}\label{eq:gateaux}
\end{align}
exists, then ${\mathcal{D}}_{u}\phi(x)$ is called the G\^ateaux derivative of $\phi$ at $x\in U$ along $u\in X$.
\end{definition}

In practical terms, convex functions are more commonly employed and possess more advantageous properties for the G\^ateaux derivative.

\begin{lemma}[\textbf{Max Formula} {\cite[Theorem 3.1.8]{Convex_analysis}}]
\label{lem_dire_gra=max}
If the function $\phi : \mathrm{E} \rightarrow (-\infty,+\infty]$ on the Euclidean space $\mathrm{E}$ is convex, then any point $x$ in core (\text{dom} $\phi$) and any direction $u$ in $\mathrm{E}$ satisfy
\begin{align}
{\mathcal{D}}_{u}\phi(x) =\max\limits_{z\in\partial\phi(x)}\langle z,u\rangle.\label{directional}
\end{align}
\end{lemma}

Based on the G\^ateaux derivative, the general form of the dual continuation of real-valued functions on the real space is defined as follows.

\begin{definition}[\textbf{Dual Continuation}]\label{def: dual extension}
Given a function $\phi:\mathcal{X}\rightarrow \mathbb{R}$ on the real space $\mathcal{X}$ (representing $\mathbb{R}$, $\mathbb{R}^{n}$ or $\mathbb{R}^{m\times n}$). 
Define $\tilde{\phi}:\mathbb{D}\mathcal{X}\rightarrow \mathbb{DR}$ satisfying 
\begin{align}
\tilde{\phi}(a + b\epsilon) = \phi(a) + {\mathcal{D}}_{b}\phi(a)\epsilon\label{eq: extension}
\end{align}
for $a,b\in\mathcal{X}$, namely $a+b\epsilon\in\mathbb{D}\mathcal{X}$, as the dual continuation of $\phi$.
\end{definition}

To provide further insight into this definition, the following proposition elucidates the basic characteristics of the dual continuation.

\begin{proposition}\label{pro: properties of dual continuation}
Given real-valued functions $\phi_1$ and $\phi_2$ on the real space $\mathcal{X}$ (representing $\mathbb{R}$, $\mathbb{R}^{n}$ or $\mathbb{R}^{m\times n}$).
Let $\tilde{\phi}_1$ and $\tilde{\phi}_2$ be corresponding dual continuations.
Suppose the G\^ateaux derivatives ${\mathcal{D}}_{b}\phi_1(a)$ and ${\mathcal{D}}_{b}\phi_2(a)$ for $a,b\in\mathcal{X}$ both exist.
Then
\begin{enumerate}
\item[${\rm{(i)}}$]
dual continuations of $\phi_1\pm\phi_2$ satisfy $\widetilde{\phi_1\pm \phi_2}(a+b\epsilon)=\tilde{\phi}_1(a+b\epsilon)\pm\tilde{\phi}_2(a+b\epsilon)$;
\item[${\rm{(ii)}}$] 
dual continuation of $\phi_1\cdot\phi_2$ satisfies $\widetilde{\phi_1\cdot\phi_2}(a+b\epsilon)=\tilde{\phi}_1(a+b\epsilon)\cdot\tilde{\phi}_2(a+b\epsilon)$.
\end{enumerate}
\end{proposition}
\begin{proof}
Based on \Cref{def: dual extension}, dual continuations of $\phi_1$ and $\phi_2$ are expressed as $\tilde{\phi}_1(a+b\epsilon)=\phi_1(a)+\mathcal{D}_b\phi_1(a)\epsilon$ and $\tilde{\phi}_2(a+b\epsilon)=\phi_2(a)+\mathcal{D}_b\phi_2(a)\epsilon$.

Denote $f_1(a)=\phi_1(a)+\phi_2(a)$.
Then the dual continuation of $f_1$ takes the form $\tilde{f}_1(a+b\epsilon)=f_1(a)+\mathcal{D}_bf_1(a)\epsilon$.
Besides, it yields that
\begin{align*}
\mathcal{D}_bf_1(a)&=\lim_{t\downarrow0}\frac{f_1(a+tb)-f_1(a)}{t} =\lim_{t\downarrow0}\frac{\big[\phi_1(a+tb)+\phi_2(a+tb)\big]-\big[\phi_1(a)+\phi_2(a)\big]}{t}\\
&\overset{(\mathrm{I})}{=}\lim_{t\downarrow0}\frac{\phi_1(a+tb)-\phi_1(a)}{t}+\lim_{t\downarrow0}\frac{\phi_2(a+tb)-\phi_2(a)}{t}
=\mathcal{D}_b\phi_1(a)+\mathcal{D}_b\phi_2(a),
\end{align*}
where ($\mathrm{I}$) results from the fact that limits $\mathcal{D}_b\phi_1(a)$ and $\mathcal{D}_b\phi_2(a)$ both exist.
Thus, 
\begin{align*}
\tilde{\phi}_1(a+b\epsilon)+\tilde{\phi}_2(a+b\epsilon)&=\phi_1(a)+\phi_2(a)+\big[\mathcal{D}_b\phi_1(a)+\mathcal{D}_b\phi_2(a)\big]\epsilon\\
&=f_1(a)+\mathcal{D}_bf_1(a)\epsilon= \tilde{f}_1(a+b\epsilon).
\end{align*}
Similarly, we infer that $\widetilde{\phi_1-\phi_2}(a+b\epsilon)=\tilde{\phi}_1(a+b\epsilon)-\tilde{\phi}_2(a+b\epsilon)$.

Denote $f_2(a)=\phi_1(a)\cdot\phi_2(a)$.
We obtain that
\begin{align*}
\mathcal{D}_bf_2(a)&=\lim_{t\downarrow0}\frac{f_2(a+tb)-f_2(a)}{t} =\lim_{t\downarrow0}\frac{\phi_1(a+tb)\cdot\phi_2(a+tb)-\phi_1(a)\cdot\phi_2(a)}{t}\\
&=\lim_{t\downarrow0}\frac{\big[\phi_1(a+tb)-\phi_1(a)\big]\cdot\phi_2(a+tb)+\big[\phi_2(a+tb)-\phi_2(a)\big]\cdot\phi_1(a)}{t}\\
&\overset{(\mathrm{II})}{=}\lim_{t\downarrow0}\frac{\phi_1(a+tb)-\phi_1(a)}{t}\cdot\lim_{t\downarrow0}\phi_2(a+tb)+\phi_1(a)\cdot\lim_{t\downarrow0}\frac{\phi_2(a+tb)-\phi_2(a)}{t}\\
&=\mathcal{D}_b\phi_1(a)\cdot\phi_2(a)+\phi_1(a)\cdot\mathcal{D}_b\phi_2(a),
\end{align*}
where the validity of ($\mathrm{II}$) is due to the existence of the limits $\mathcal{D}_b\phi_1(a)$, $\mathcal{D}_b\phi_2(a)$ and $\lim_{t\downarrow0}\phi_2(a+tb)$.
Accordingly, it derives that
\begin{align*}
\tilde{\phi}_1(a+b\epsilon)\cdot\tilde{\phi}_2(a+b\epsilon)&=\phi_1(a)\cdot\phi_2(a)+\big[\phi_1(a)\cdot\mathcal{D}_b\phi_2(a)\epsilon+\mathcal{D}_b\phi_1(a)\cdot\phi_2(a)\big]\epsilon\\
&=f_2(a)+\mathcal{D}_bf_2(a)\epsilon=\tilde{f}_2(a+b\epsilon).
\end{align*}

Consequently, we complete the proof.
\end{proof}

Subsequently, \Cref{direction_gradient,def: dual extension}, and \Cref{lem_dire_gra=max} are drawn upon to develop concrete dual continuations of real-valued functions for real numbers.
It can be readily demonstrated that the dual-valued absolute function and the dual-valued power functions are expressed as
\begin{subequations}
\begin{align}
&|a+b\epsilon| = |a| + {\mathcal{D}}_b|a|\epsilon = \left\{\begin{array}{ll}
|a|+{\rm{sign}}(a)b\epsilon, & a\neq0,\\
|b|\epsilon, & a=0.
\end{array}\right.
\label{eq: absolute}\\
&(a+b\epsilon)^p =\left\{\begin{array}{ll}
b\epsilon, & p=1,a=0,b\in\mathbb{R}, \\
0, & p>1,a=0,b\in\mathbb{R},\\
a^p+pa^{p-1}b\epsilon, & p\geq1,a>0,b\in\mathbb{R}.
\end{array}\right. \; \label{eq: (a+b eps)p}\\
&(a+b\epsilon)^\frac{1}{p} =\begin{array}{ll}
a^\frac{1}{p}+\frac{1}{p}a^{\frac{1}{p}-1}b\epsilon, & \,p\geq1, a>0,b\in\mathbb{R}.
\end{array}\label{eq: (a+b eps)(1/p)}
\end{align}
\end{subequations}
Indeed, \cref{eq: absolute} is consistent with the conclusion proposed by Qi, Ling, and Yan in \cite{total_order_qi}.
Besides, \cref{eq: (a+b eps)p,eq: (a+b eps)(1/p)} extend the results developed by Miao and Huang in \cite{miao2023norms}.

\section{Dual-Valued Vector Norms}\label{sec3: dual vector norms}
The role of a norm on a vector space is to act as a distance measure.
A similar argument can be made for the definition of dual-valued vector norms. 
With this in mind, a key objective of this section is to establish the general form of dual-valued vector norms.
Furthermore, the basic properties of dual-valued vector norms are discussed, accompanied by some concrete examples.

\subsection{Definitions}
Initially, we propose the definition of a dual-valued vector norm in dual algebra, with the size relations adopting the total order of dual numbers, as detailed in \Cref{def: total order}.

\begin{definition}\label{def: dual vector norms}
A dual-valued vector norm on $\mathbb{DR}^n$ is a function $f:\mathbb{DR}^n\rightarrow\mathbb{DR}$ that satisfies the following properties under the total order of dual numbers
\begin{enumerate}
\item[${\rm{(i)}}$] $f(\bm{x})\geq0$ for any $\bm{x}\in\mathbb{DR}^n$ and $f(\bm{x})=0$ if and only if $\bm{x}=\bm{0}$;
\item[${\rm{(ii)}}$] $f(c\bm{x})=|c|f(\bm{x})$ for any $c\in\mathbb{DR}$ and $\bm{x}\in\mathbb{DR}^n$;
\item[${\rm{(iii)}}$] $f(\bm{x}+\bm{y})\leq f(\bm{x})+f(\bm{y})$ for any $\bm{x},\bm{y}\in\mathbb{DR}^n$.
\end{enumerate}
\end{definition}

A function of this nature is designated by a double bar notation $f(\bm{x})=\Vert\bm{x}\Vert$, same as the vector norms. 
Besides, the use of subscripts on the double bar serves to distinguish different dual-valued vector norms.

As defined in \Cref{def: dual extension}, the dual continuation of a vector norm entails the addition of a term involving G\^ateaux derivative as the infinitesimal part, expressed as $\Vert\bm{x}\Vert=\Vert\bm{x}_{\rm{s}}\Vert+{\mathcal{D}}_{\bm{x}_{\rm{i}}}\Vert\bm{x}_{\rm{s}}\Vert\epsilon$ for a given vector norm $\Vert\cdot\Vert$ and a dual vector $\bm{x}=\bm{x}_{\rm{s}}+\bm{x}_{\rm{i}}\epsilon$.
The following theorem claims that this dual continuation is a dual-valued vector norm.

\begin{theorem}\label{theorem:vector norm}
Given a vector norm $\Vert\cdot\Vert:\mathbb{R}^{n}\rightarrow\mathbb{R}$.
Then its dual continuation $\Vert\cdot\Vert:\mathbb{DR}^{n}\rightarrow\mathbb{DR}$ satisfying
\begin{align}
\Vert\bm{x}\Vert = \left\{\begin{array}{ll}
\Vert\bm{x}_{\rm{s}}\Vert + \max\limits_{\bm{y}\in\partial\Vert\bm{x}_{\rm{s}}\Vert}\langle\bm{y},\bm{x}_{\rm{i}}\rangle\epsilon\;,  & \bm{x}_{\rm{s}}\neq\bm{0}, \\
\Vert\bm{x}_{\rm{i}}\Vert\epsilon \;,& \bm{x}_{\rm{s}}=\bm{0},
\end{array}\right.\label{vector norm expansion}
\end{align}
for $\bm{x}=\bm{x}_{\rm{s}}+\bm{x}_{\rm{i}}\epsilon\in\mathbb{DR}^n$ is a dual-valued vector norm.
\end{theorem}

\begin{proof}
It follows from \Cref{def: dual extension} that the dual continuation of the vector norm $\Vert\cdot\Vert$ has the form $\Vert\bm{x}\Vert=\Vert\bm{x}_{\rm{s}}\Vert+{\mathcal{D}}_{\bm{x}_{\rm{i}}}\Vert\bm{x}_{\rm{s}}\Vert\epsilon$.
For the case that $\bm{x}_{\rm{s}}=\bm{0}$, we obtain ${\mathcal{D}}_{\bm{x}_{\rm{i}}}\Vert\bm{0}\Vert=\lim_{t\downarrow0}\frac{\Vert\bm{0}+t\bm{x}_{\rm{i}}\Vert-\Vert\bm{0}\Vert}{t}=\Vert\bm{x}_{\rm{i}}\Vert$.
Otherwise, the subdifferential $\partial\Vert\bm{x}_{\rm{s}}\Vert$ is valid, since the vector norm is convex.
Thus, ${\mathcal{D}}_{\bm{x}_{\rm{i}}}\Vert\bm{x}_{\rm{s}}\Vert=\max_{\bm{y}\in\partial\Vert\bm{x}_{\rm{s}}\Vert}\langle\bm{y},\bm{x}_{\rm{i}}\rangle$ for non-zero $\bm{x}_{\rm{s}}$, based on \Cref{lem_dire_gra=max}.
Accordingly, the dual continuation of the vector norm has the expression shown in \cref{vector norm expansion}.
The following objective is to verify that \cref{vector norm expansion} is a dual-valued vector norm under the total order of dual numbers.

\textbf{(i) Non-negativity:} $\Vert\bm{x}\Vert\geq0$ and $\Vert\bm{x}\Vert=0\Leftrightarrow\bm{x}=\bm{0}$.
    
If $\bm{x}_{\rm{s}}\neq\bm{0}$, then $\Vert\bm{x}_{\rm{s}}\Vert>0$. 
Thus, $\Vert\bm{x}\Vert>0$, based on \Cref{def: total order}. 
Else if $\bm{x}_{\rm{s}}=\bm{0}$ but $\bm{x}_{\rm{i}}\neq\bm{0}$, then $\Vert\bm{x}_{\rm{i}}\Vert>0$. 
Thus, $\Vert\bm{x}\Vert = \Vert\bm{x}_{\rm{i}}\Vert\epsilon>0$.
Otherwise, if $\bm{x}_{\rm{s}}=\bm{x}_{\rm{i}}=\bm{0}$, then $\Vert\bm{x}\Vert=\Vert\bm{x}_{\rm{i}}\Vert\epsilon=0$.
Hence, $\Vert\bm{x}\Vert\geq0$ for any $\bm{x}$ and $\bm{x}=\bm{0}$ results in $\Vert\bm{x}\Vert=0$.

On the other hand, if $\Vert\bm{x}\Vert=0$, then $\Vert\bm{x}_{\rm{s}}\Vert=0$ and ${\mathcal{D}}_{\bm{x}_{\rm{i}}}\Vert\bm{x}_{\rm{s}}\Vert=0$, which implies that $\bm{x}_{\rm{s}}=\bm{0}$.
Besides, ${\mathcal{D}}_{\bm{x}_{\rm{i}}}\Vert\bm{0}\Vert = \Vert\bm{x}_{\rm{i}}\Vert=0$ leads to $\bm{x}_{\rm{i}}=\bm{0}$.
Hence, $\bm{x}=\bm{0}$.
    
\textbf{(ii) Homogeneity:}
$\Vert c\bm{x}\Vert=|c|\cdot\Vert\bm{x}\Vert$.

Set $c=a+b\epsilon\in\mathbb{DR}$.
If $a=0$, then $\Vert c\bm{x}\Vert = \Vert b\bm{x}_{\rm{s}}\epsilon\Vert = \Vert b\bm{x}_{\rm{s}}\Vert\epsilon = |b|\cdot\Vert\bm{x}_{\rm{s}}\Vert\epsilon$. 
Besides, $|c|\cdot\Vert\bm{x}\Vert = (|b|\epsilon)(\Vert\bm{x}_{\rm{s}}\Vert + {\mathcal{D}}_{\bm{x}_{\rm{i}}}\Vert\bm{x}_{\rm{s}}\Vert\epsilon) = |b|\cdot\Vert\bm{x}_{\rm{s}}\Vert\epsilon$, as illustrated in \cref{eq: absolute}.
If $\bm{x}_{\rm{s}}=\bm{0}$, then $\Vert c\bm{x}\Vert = \Vert a\bm{x}_{\rm{i}}\epsilon\Vert = \Vert a\bm{x}_{\rm{i}}\Vert\epsilon = |a|\cdot\Vert\bm{x}_{\rm{i}}\Vert\epsilon$. 
In addition, $|a+b\epsilon|\cdot\Vert\bm{x}_{\rm{s}}+\bm{x}_{\rm{i}}\epsilon\Vert = (|a|+{\mathcal{D}}_b|a|\epsilon)(\Vert \bm{x}_{\rm{i}}\Vert\epsilon) = |a|\cdot\Vert \bm{x}_{\rm{i}}\Vert\epsilon$.
Moreover, if $a\neq0$ and $\bm{x}_{\rm{s}}\neq\bm{0}$, then 
\begin{align*}
{\mathcal{D}}_{a\bm{x}_{\rm{i}}+b\bm{x}_{\rm{s}}}\Vert a\bm{x}_{\rm{s}}\Vert&= \lim\limits_{t\downarrow0}\frac{\Vert a\bm{x}_{\rm{s}} + t(a\bm{x}_{\rm{i}}+b\bm{x}_{\rm{s}})\Vert-\Vert a\bm{x}_{\rm{s}}\Vert}{t}
=|a|\cdot{\mathcal{D}}_{\bm{x}_{\rm{i}}+\frac{b}{a}\bm{x}_{\rm{s}}}\Vert\bm{x}_{\rm{s}}\Vert \\
&\overset{{\rm{(I)}}}{=} |a|\max\limits_{\bm{y}\in\partial\Vert\bm{x}_{\rm{s}}\Vert}\langle\bm{y},\bm{x}_{\rm{i}}\rangle+|a|\frac{b}{a}\Vert\bm{x}_{\rm{s}}\Vert
= |a|\cdot{\mathcal{D}}_{\bm{x}_{\rm{i}}}\Vert\bm{x}_{\rm{s}}\Vert+{\rm{sign}}(a)b\cdot\Vert\bm{x}_{\rm{s}}\Vert,
\end{align*}
where ${\rm{(I)}}$ results from the fact that $\langle\bm{y},\bm{x}_{\rm{s}}\rangle = \Vert\bm{x}_{\rm{s}}\Vert$ for any $\bm{y}\in\partial\Vert\bm{x}_{\rm{s}}\Vert$.

Thus, $\Vert c\bm{x}\Vert = \Vert a\bm{x}_{\rm{s}}\Vert + {\mathcal{D}}_{a\bm{x}_{\rm{i}}+b\bm{x}_{\rm{s}}}\Vert a\bm{x}_{\rm{s}}\Vert\epsilon  = |a|\cdot\Vert\bm{x}_{\rm{s}}\Vert + \big(|a|\cdot{\mathcal{D}}_{\bm{x}_{\rm{i}}}\Vert\bm{x}_{\rm{s}}\Vert+{\rm{sign}}(a)b\cdot\Vert\bm{x}_{\rm{s}}\Vert\big)\epsilon$. 
Moreover, $|c|\cdot\Vert\bm{x}\Vert = \big(|a|+{\rm{sign}}(a)b\epsilon\big)\big(\Vert \bm{x}_{\rm{s}}\Vert + {\mathcal{D}}_{\bm{x}_{\rm{i}}}\Vert\bm{x}_{\rm{s}}\Vert\epsilon\big) = |a|\cdot\Vert \bm{x}_{\rm{s}}\Vert+\big( |a|\cdot{\mathcal{D}}_{\bm{x}_{\rm{i}}}\Vert\bm{x}_{\rm{s}}\Vert+{\rm{sign}}(a)b\cdot\Vert\bm{x}_{\rm{s}}\Vert\big)\epsilon$.  
In conclusion, $\Vert c\bm{x}\Vert=|c|\cdot\Vert\bm{x}\Vert$.
    
\textbf{(iii) Triangle inequality:} 
$\Vert\bm{x}+\bm{y}\Vert\leq\Vert\bm{x}\Vert+\Vert\bm{y}\Vert$.

Set $\bm{y}=\bm{y}_{\rm{s}}+\bm{y}_{\rm{i}}\epsilon\in\mathbb{DR}^{n}$.
Note that $\Vert\bm{x}+\bm{y}\Vert = \Vert\bm{x}_{\rm{s}}+\bm{y}_{\rm{s}}\Vert + {\mathcal{D}}_{\bm{x}_{\rm{i}}+\bm{y}_{\rm{i}}}\Vert\bm{x}_{\rm{s}}+\bm{y}_{\rm{s}}\Vert\epsilon$ and $\Vert\bm{x}\Vert+\Vert\bm{y}\Vert = \Vert\bm{x}_{\rm{s}}\Vert+\Vert\bm{y}_{\rm{s}}\Vert+({\mathcal{D}}_{\bm{x}_{\rm{i}}}\Vert\bm{x}_{\rm{s}}\Vert+{\mathcal{D}}_{\bm{y}_{\rm{i}}}\Vert\bm{y}_{\rm{s}}\Vert)\epsilon$.
If $\Vert\bm{x}_{\rm{s}}+\bm{y}_{\rm{s}}\Vert<\Vert\bm{x}_{\rm{s}}\Vert+\Vert\bm{y}_{\rm{s}}\Vert$, then $\Vert\bm{x}+\bm{y}\Vert<\Vert\bm{x}\Vert+\Vert\bm{y}\Vert$.
Otherwise, if $\Vert\bm{x}_{\rm{s}}+\bm{y}_{\rm{s}}\Vert=\Vert\bm{x}_{\rm{s}}\Vert+\Vert\bm{y}_{\rm{s}}\Vert$, then 
\begin{align*}
&{\mathcal{D}}_{\bm{x}_{\rm{i}}+\bm{y}_{\rm{i}}}\Vert\bm{x}_{\rm{s}}+\bm{y}_{\rm{s}}\Vert = \lim\limits_{t\downarrow0}\frac{\Vert\bm{x}_{\rm{s}}+\bm{y}_{\rm{s}}+t(\bm{x}_{\rm{i}}+\bm{y}_{\rm{i}})\Vert-\Vert\bm{x}_{\rm{s}}+\bm{y}_{\rm{s}}\Vert}{t}\\
\leq& \lim\limits_{t\downarrow0}\frac{\Vert\bm{x}_{\rm{s}}+t\bm{x}_{\rm{i}}\Vert-\Vert\bm{x}_{\rm{s}}\Vert+\Vert\bm{y}_{\rm{s}}+t\bm{y}_{\rm{i}}\Vert-\Vert\bm{y}_{\rm{s}}\Vert}{t}
\overset{{\rm{(II)}}}{=} {\mathcal{D}}_{\bm{x}_{\rm{i}}}\Vert\bm{x}_{\rm{s}}\Vert+{\mathcal{D}}_{\bm{y}_{\rm{i}}}\Vert\bm{y}_{\rm{s}}\Vert,
\end{align*}
where (II) follows from the existence of both limits.
Hence, $\Vert\bm{x}+\bm{y}\Vert\leq\Vert\bm{x}\Vert+\Vert\bm{y}\Vert$.

Consequently, the expression \cref{vector norm expansion} is a dual-valued vector norm.
\end{proof}

\Cref{theorem:vector norm} has provided a comprehensive overview of the general form of dual-valued vector norms.
This allows us to focus on the relationship between a vector norm and its dual continuation, similar to the triangle inequality.

\begin{lemma}\label{lem: ||x||<=||xs||+||xi||eps}
Given a vector norm $\Vert\cdot\Vert:\mathbb{R}^n\rightarrow\mathbb{R}$ and its dual continuation $\Vert\cdot\Vert:\mathbb{DR}^n\rightarrow\mathbb{DR}$.
Then for the dual vector $\bm{x}=\bm{x}_{\rm{s}}+\bm{x}_{\rm{i}}\epsilon\in\mathbb{DR}^n$,
\begin{align}
\Vert\bm{x}\Vert \leq\Vert\bm{x}_{\rm{s}}\Vert+\Vert\bm{x}_{\rm{i}}\Vert\epsilon.
\end{align}
\end{lemma}
\begin{proof}
Suppose that $\bm{x}_{\rm{s}}=\bm{0}$.
Then
$\Vert\bm{x}\Vert =\Vert\bm{x}_{\rm{i}}\Vert\epsilon$ based on \Cref{theorem:vector norm}.
Otherwise, $\Vert\bm{y}\Vert_d\leq1$ for any $\bm{y}\in\partial\Vert\bm{x}_{\rm{s}}\Vert$, where $\Vert\cdot\Vert_d$ represents the dual norm of the vector norm $\Vert\cdot\Vert$.
Besides, H$\ddot{o}$lder's inequality implies that $\langle\bm{y},\bm{x}_{\rm{i}}\rangle\leq\Vert\bm{y}\Vert_d\Vert\bm{x}_{\rm{i}}\Vert\leq\Vert\bm{x}_{\rm{i}}\Vert$.
Thus, 
\begin{align*}
\Vert\bm{x}\Vert = \Vert\bm{x}_{\rm{s}}\Vert+\max\limits_{\bm{y}\in\partial\Vert\bm{x}_{\rm{s}}\Vert}\langle\bm{y},\bm{x}_{\rm{i}}\rangle\epsilon\leq\Vert\bm{x}_{\rm{s}}\Vert+\Vert\bm{x}_{\rm{i}}\Vert\epsilon.
\end{align*}
\end{proof}

\subsection{Dual-Valued Vector \texorpdfstring{$p$}{p}-Norm}
A useful class of vector norms is the vector $p$-norm defined by $\Vert\bm{y}\Vert_p=(|y_1|^p+\cdots+|y_n|^p)^{\frac{1}{p}}$ for $1\leq p<\infty$ and $\Vert\bm{y}\Vert_\infty=\max_{1\leq i\leq n}|y_i|$, where $\bm{y}=[y_1,\cdots,y_n]^\top\in\mathbb{R}^n$.
This section derives the dual-valued vector $p$-norm $(1\leq p\leq\infty)$ and discusses the consistency between the dual-valued vector $p$-norm and the result proposed by Miao and Huang in \cite{miao2023norms}.

\begin{proposition}\label{pro: dual vector p norm}
The dual-valued vector $p$-norm $(1\leq p\leq \infty)$ $\Vert\cdot\Vert_p:\mathbb{DR}^{n}\rightarrow\mathbb{DR}$ of the dual vector $\bm{x}=\bm{x}_{\rm{s}}+\bm{x}_{\rm{i}}\epsilon\in\mathbb{DR}^n$ extending from the vector $p$-norm has the form 
\begin{subequations}
\begin{align}
\Vert\bm{x}\Vert_1 &=\Vert\bm{x}_{\rm{s}}\Vert_1+ \Big(\langle {\rm{sign}}(\bm{x}_{\rm{s}}),\bm{x}_{\rm{i}}\rangle+\sum_{k\in{\rm{supp}}^c(\bm{x}_{\rm{s}})}\big|x_{\rm{i}}^k\big|\Big)\epsilon,\label{dual vector 1-norm}\\
\Vert\bm{x}\Vert_p &=\left\{\begin{array}{ll}
\Vert\bm{x}_{\rm{s}}\Vert_p+ \frac{\langle |\bm{x}_{\rm{s}}|^{p-2}\odot\bm{x}_{\rm{s}},\bm{x}_{\rm{i}}\rangle}{\Vert\bm{x}_{\rm{s}}\Vert_p^{p-1}}\epsilon,  &  \bm{x}_{\rm{s}}\neq\bm{0},\\
\Vert\bm{x}_{\rm{i}}\Vert_p\epsilon, & \bm{x}_{\rm{s}}=\bm{0},
\end{array}\label{eq: dual vector p-norm}\right. (1<p<\infty),\\
\Vert\bm{x}\Vert_\infty &=\left\{\begin{array}{ll}
\Vert\bm{x}_{\rm{s}}\Vert_\infty+ \max\limits_{k\in\mathcal{I}}\big({\rm{sign}}(x_{\rm{s}}^k)x_{\rm{i}}^k\big)\epsilon,  & \bm{x}_{\rm{s}}\neq\bm{0}, \\
\Vert\bm{x}_{\rm{i}}\Vert_\infty\epsilon, &  \bm{x}_{\rm{s}}=\bm{0},
\end{array}\right.
\label{dual vector infty-norm}
\end{align}
\end{subequations}
where $x_{\rm{s}}^k$ and $x_{\rm{i}}^k$ stand for the $k$-th element of $\bm{x}_{\rm{s}}$ and $\bm{x}_{\rm{i}}$.
Besides, ${\rm{supp}}^c(\bm{x}_{\rm{s}})$ represents the set of indices where $\bm{x}_{\rm{s}}$ is zero and $\mathcal{I} = \{k:|x_{\rm{s}}^k|=\Vert\bm{x}_{\rm{s}}\Vert_\infty\}$.
\end{proposition}

To demonstrate the expression of dual-valued vector $p$-norm, it is necessary to introduce the subdifferential of the vector $p$-norm initially.

\begin{lemma}\label{lem: sub of vector p norm}
The subdifferential of the vector $p$-norm at a point $\bm{0}\neq\bm{z}\in\mathbb{R}^n$ is
\begin{subequations}
\begin{align}
\partial\Vert\bm{z}\Vert_1 &= \{{\rm{sign}}(\bm{z})+\bm{w}:\bm{w}\in\mathbb{R}^n,{\rm{supp}}(\bm{w})\subseteq{\rm{supp}}^c(\bm{z}),\Vert\bm{w}\Vert_\infty\leq 1\},\label{eq: sub of vector 1}\\
\partial\Vert\bm{z}\Vert_p &= \left\{\frac{|\bm{z}|^{p-2}\odot\bm{z}}{\Vert\bm{z}\Vert_p^{p-1}}\right\}, \;(1<p<\infty),\label{eq: sub of vector p}\\
\partial\Vert\bm{z}\Vert_\infty &= \bigg\{\sum_{k\in\mathcal{I}}\beta_k\hat{\bm{e}}_k:\mathcal{I}=\{k:|z_k|=\Vert\bm{z}\Vert_\infty\}, \sum_{k\in\mathcal{I}}\beta_k=1, \beta_k\geq0\bigg\},\label{eq: sub of vector infty}
\end{align}
\end{subequations}
where ${\rm{supp}}(\bm{w})$ and ${\rm{supp}}^c(\bm{z})$ in \cref{eq: sub of vector 1} denote the sets of indices where $\bm{w}$ is non-zero and $\bm{z}$ is zero, and $|\bm{z}|\!:=\!\big[|z_1|,\cdots,|z_n|\big]^\top$ in \cref{eq: sub of vector p}.
Besides, in \cref{eq: sub of vector infty}, $z_k$ represents the $k$-th element of $\bm{z}$ and $\hat{\bm{e}}_k\!:=\!{\rm{sign}}(z_k)\bm{e}_k$ denotes a vector whose $k$-th element is ${\rm{sign}}(z_k)$ and the others are zero.
\end{lemma}
\begin{proof}
It should be noted that the subdifferential of the vector $p$-norm for $p=1$ and $p=\infty$, shown in \cref{eq: sub of vector 1,eq: sub of vector infty}, is provided in \cite{Nest2018convex}.

Consider the case $1<p<\infty$.
Given that $\partial\Vert\bm{z}\Vert_p = \{\bm{y}|\bm{y}\in\mathbb{R}^{n},\langle\bm{y},\bm{z}\rangle = \Vert\bm{z}\Vert_p,\Vert\bm{y}\Vert_q\leq1\}$, where $\frac{1}{p}+\frac{1}{q}=1$, namely $p+q=pq$. 
Suppose that the $k$-th element of $\bm{z}$ satisfies $z_k\neq0$, we derive that 
\begin{align*}
\frac{\partial\Vert\bm{z}\Vert_p}{\partial z_k} = \frac{1}{p}\Vert\bm{z}\Vert_p^{1-p}\cdot p|z_k|^{p-1} \cdot {\rm{sign}}(z_k) = \frac{|z_k|^{p-2}z_k}{\Vert\bm{z}\Vert_p^{p-1}}.
\end{align*}
Thus, $\frac{|\bm{z}|^{p-2}\odot\bm{z}}{\Vert\bm{z}\Vert_p^{p-1}}\in\partial\Vert\bm{z}\Vert_p$.
Suppose that $\bm{y}\in\partial\Vert\bm{z}\Vert_p$.
It follows that $\bm{y}=\frac{|\bm{z}|^{p-2}\odot\bm{z}}{\Vert\bm{z}\Vert_p^{p-1}}+\bm{w}$ satisfying $\langle\bm{z},\bm{w}\rangle=0$, $\big\langle\frac{|\bm{z}|^{p-2}\odot\bm{z}}{\Vert\bm{z}\Vert_p^{p-1}},\bm{w}\big\rangle=0$ and $\Vert\bm{y}\Vert_q\leq1$.
Hence, it yields that
\begin{align*}
\Vert\bm{y}\Vert_q^q =\left\Vert\frac{|\bm{z}|^{p-2}\odot\bm{z}}{\Vert\bm{z}\Vert_p^{p-1}}\right\Vert_q^q + \Vert\bm{w}\Vert_q^q = 1 + \Vert\bm{w}\Vert_q^q.
\end{align*}
Then we obtain from $\Vert\bm{y}\Vert_q\leq1$ that $\Vert\bm{w}\Vert_q=0$.
That is, $\bm{w}=\bm{0}$.
Consequently, $\partial\Vert\bm{z}\Vert_p$ is a singleton for non-zero $\bm{z}$, and the result shown in \cref{eq: sub of vector p} follows.

Here, we complete the proof.
\end{proof}

Furthermore, the expression of the dual-valued vector $p$-norm $(1\leq p\leq\infty)$ can be derived based on \Cref{theorem:vector norm} and \Cref{lem: sub of vector p norm}.

\begin{proof}[Proof of \Cref{pro: dual vector p norm}]
\textbf{Case 1:} $p=1$.
It can be derived from \cref{eq: sub of vector 1} that
\begin{align*}
&{\mathcal{D}}_{\bm{x}_{\rm{i}}}\Vert\bm{x}_{\rm{s}}\Vert_1 = \max_{\bm{y}\in\partial\Vert\bm{x}_{\rm{s}}\Vert_1}\langle\bm{y},\bm{x}_{\rm{i}}\rangle = \max_{\substack{{\rm{supp}}(\bm{w})\subseteq{\rm{supp}}^c(\bm{x}_{\rm{s}})\\ \Vert\bm{w}\Vert_\infty\leq1}}\langle{\rm{sign}}(\bm{x}_{\rm{s}})+\bm{w},\bm{x}_{\rm{i}}\rangle\\
=& \langle{\rm{sign}}(\bm{x}_{\rm{s}}),\bm{x}_{\rm{i}}\rangle+\max_{\Vert\bm{w}\Vert_\infty=1}\sum_{k\in{\rm{supp}}^c(\bm{x}_{\rm{s}})}w_k x_{\rm{i}}^k
= \langle{\rm{sign}}(\bm{x}_{\rm{s}}),\bm{x}_{\rm{i}}\rangle+\sum_{k\in{\rm{supp}}^c(\bm{x}_{\rm{s}})}\big|x_{\rm{i}}^k\big|,
\end{align*}
where $w_k$ and $x_{\rm{i}}^k$ represent the $k$-th element of $\bm{w}$ and $\bm{x}_{\rm{i}}$, respectively.
Hence, combining \Cref{theorem:vector norm}, the dual-valued vector 1-norm has the form in \cref{dual vector 1-norm}.

\textbf{Case 2:} $1<p<\infty$.
The formula of the dual-valued vector $p$-norm $(1<p<\infty)$ shown in \cref{eq: dual vector p-norm} is captured directly by \Cref{theorem:vector norm} and \Cref{lem: sub of vector p norm} \cref{eq: sub of vector p}. 

\textbf{Case 3:} $p=\infty$.
We obtain from \cref{eq: sub of vector infty} that
\begin{align*}
{\mathcal{D}}_{\bm{x}_{\rm{i}}}\Vert\bm{x}_{\rm{s}}\Vert_\infty &= \max_{\bm{y}\in\partial\Vert\bm{x}_{\rm{s}}\Vert_\infty}\langle\bm{y},\bm{x}_{\rm{i}}\rangle = \max_{\sum_{k\in\mathcal{I}}\beta_k=1,\beta_k\geq0}\Big\langle\sum_{k\in\mathcal{I}}\beta_k\hat{\bm{e}}_k,\bm{x}_{\rm{i}}\Big\rangle\\
&= \max_{\sum_{k\in\mathcal{I}}\beta_k=1,\beta_k\geq0}\sum_{k\in\mathcal{I}}\beta_k{\rm{sign}}\big(x_{\rm{s}}^k\big)x_{\rm{i}}^k= \max_{k\in\mathcal{I}}\big({\rm{sign}}(x_{\rm{s}}^k)x_{\rm{i}}^k\big).
\end{align*}
Thus, the dual-valued vector $\infty$-norm is expressed as \cref{dual vector infty-norm} based on \Cref{theorem:vector norm}.

Consequently, we complete the proof.
\end{proof}

Additionally, the following theorem claims that our proposed dual-valued vector $p$-norm ($1\leq p\leq\infty$) is accordant with the element-wise method, not only limited to positive integer $p$ put forth in \cite{miao2023norms}.

\begin{theorem}
Let $\bm{x}=\bm{x}_{\rm{s}}+\bm{x}_{\rm{i}}\epsilon\in\mathbb{DR}^n$ $(\bm{x}_{\rm{s}}\neq\bm{0})$ be a dual vector and $x_k=x^k_{\rm{s}}+x^k_{\rm{i}}\epsilon$ denote the $k$-th element of $\bm{x}$.
Then the dual-valued vector $p$-norm satisfies
\begin{enumerate}
\item[${\rm{(i)}}$] $\Vert\bm{x}\Vert_1=|x_1|+|x_2|+\cdots+|x_n|$;
\item[${\rm{(i)}}$] $\Vert\bm{x}\Vert_p=\left(\sum_{k=1}^n |x_k|^p\right)^{\frac{1}{p}}$, for $1<p<\infty$;
\item[${\rm{(ii)}}$] $\Vert\bm{x}\Vert_\infty=\max_{1\leq k\leq n}|x_k|$.
\end{enumerate}
\end{theorem}

\begin{proof}
${\rm{(i)}}$ It yields from \cref{eq: absolute,dual vector 1-norm} that 
\begin{align*} 
\sum_{k=1}^n|x_k|
&= \sum\limits_{k=1}^n \big|x_{\rm{s}}^k\big| + \sum\limits_{k\in{\rm{supp}}(\bm{x}_{\rm{s}})}{\rm{sign}}\big(x_{\rm{s}}^k\big)x_{\rm{i}}^k\epsilon +\sum\limits_{k\in{\rm{supp}}^c(\bm{x}_{\rm{s}})}\big|x_{\rm{i}}^k\big|\epsilon\\
&=\Vert\bm{x}_{\rm{s}}\Vert_1+ \Big(\langle{\rm{sign}}(\bm{x}_{\rm{s}}),\bm{x}_{\rm{i}}\rangle+\sum_{k\in{\rm{supp}}^c(\bm{x}_{\rm{s}})}\big|x_{\rm{i}}^k\big|\Big)\epsilon= \Vert\bm{x}\Vert_1.
\end{align*}

${\rm{(ii)}}$ It infers from \cref{eq: (a+b eps)p,eq: (a+b eps)(1/p),eq: dual vector p-norm} that
\begin{align*}
\left(\sum_{k=1}^n|x_k|^p\right)^{\frac{1}{p}} 
&=\left(\sum_{k=1}^n \big|x_{\rm{s}}^k\big|^p +  p\sum_{k=1}^n \big|x_{\rm{s}}^k\big|^{p-1}{\rm{sign}}\big(x_{\rm{s}}^k\big)x_{\rm{i}}^k\epsilon\right)^{\frac{1}{p}}\\
&=\left(\sum_{k=1}^n \big|x_{\rm{s}}^k\big|^p\right)^{\frac{1}{p}} + \frac{1}{p} \left(\sum_{k=1}^n \big|x_{\rm{s}}^k\big|^p\right)^{\frac{1}{p}-1}\cdot p\sum_{k=1}^n \big|x_{\rm{s}}^k\big|^{p-2}x_{\rm{s}}^kx_{\rm{i}}^k\epsilon\\
&= \Vert\bm{x}_{\rm{s}}\Vert_p + \frac{\langle|\bm{x}_{\rm{s}}|^{p-2}\odot\bm{x}_{\rm{s}},\bm{x}_{\rm{i}}\rangle}{\Vert\bm{x}_{\rm{s}}\Vert_p^{p-1}}\epsilon= \Vert\bm{x}\Vert_p.
\end{align*}

${\rm{(iii)}}$ We obtain from \cref{dual vector infty-norm,eq: absolute} that \begin{align*}
\max_{1\leq k\leq n}|x_k|
&= \max_{1\leq k\leq n}\Big(\big|x_{\rm{s}}^k\big| + {\rm{sign}}\big(x_{\rm{s}}^k\big)x_{\rm{i}}^k\epsilon\Big)\\
&\overset{({\rm{I}})}{=} \Vert\bm{x}_{\rm{s}}\Vert_\infty + \max\limits_{k\in\mathcal{{I}}}\big({\rm{sign}}(x_{\rm{s}}^k)x_{\rm{i}}^k\big)\epsilon= \Vert\bm{x}\Vert_\infty,
\end{align*}
where (I) results from the total order of dual numbers, which means that maximizing a dual number is equivalent to first maximizing its standard part and then its infinitesimal part.
Besides, $\mathcal{I}=\{k:|x_{\rm{s}}^k|=\Vert\bm{x}_{\rm{s}}\Vert_\infty\}$.

As a result, the dual-valued vector $p$-norm ($1\leq p\leq \infty$) is equivalent to the element-wise calculation.
\end{proof}

\section{Dual-Valued Matrix Norms}\label{sec4: dual matrix norms}

Matrix norms can be employed to measure the distance on the space of matrices.
Similarly, it is urgent to define dual-valued matrix norms to quantify the distance between two dual matrices.
In this section, we propose the general form of the dual-valued matrix norms and discuss its several fundamental properties.
Moreover, some dual-valued unitarily invariant norms, dual-valued operator norms, and other dual-valued functions are provided.

\subsection{Definitions}
Initially, we propose the definition of a dual-valued matrix norm in dual algebra, under the total order of dual numbers in \Cref{def: total order}.
\begin{definition} 
The function $f:\mathbb{DR}^{m\times n}\rightarrow\mathbb{DR}$ is a dual-valued matrix norm if the following three properties hold under the total order of dual numbers
\begin{enumerate}
\item[${\rm{(i)}}$] $f(\bm{A})\geq0$ for any $\bm{A}\in\mathbb{DR}^{m\times n}$ and $f(\bm{A})=0$ if and only if $\bm{A}=\bm{O}$;
\item[${\rm{(ii)}}$] $f(a\bm{A})=|a|f(\bm{A})$ for any $a\in\mathbb{DR}$ and $\bm{A}\in\mathbb{DR}^{m\times n}$;
\item[${\rm{(iii)}}$] $f(\bm{A}+\bm{B})\leq f(\bm{A})+f(\bm{B})$ for any $\bm{A},\bm{B}\in\mathbb{DR}^{m\times n}$.
\end{enumerate}
\end{definition}
As with matrix norms, we still use a double bar notation with subscripts to designate dual matrix norms, i.e., $\Vert\bm{A}\Vert=f(\bm{A})$.

\Cref{def: dual extension} provides a method to extend a real-valued function to a dual-valued one by adding a term concerning the G\^ateaux derivative as the infinitesimal part.
Specifically, the dual continuation of a matrix norm $\Vert\cdot\Vert$ is $\Vert\bm{A}\Vert=\Vert\bm{A}_{\rm{s}}\Vert+{\mathcal{D}}_{\bm{A}_{\rm{i}}}\Vert\bm{A}_{\rm{s}}\Vert\epsilon$, for the dual matrix $\bm{A}=\bm{A}_{\rm{s}}+\bm{A}_{\rm{i}}\epsilon$.
Given that matrix norms are all convex, this implies that ${\mathcal{D}}_{\bm{A}_{\rm{i}}}\Vert\bm{A}_{\rm{s}}\Vert=\max_{\bm{G}\in\partial\Vert\bm{A}_{\rm{s}}\Vert}\langle\bm{G},\bm{A}_{\rm{i}}\rangle$ for non-zero $\bm{A}_{\rm{s}}$, as outlined in \Cref{lem_dire_gra=max}.
Besides, ${\mathcal{D}}_{\bm{A}_{\rm{i}}}\Vert\bm{O}\Vert=\lim_{t\downarrow0}\frac{\Vert\bm{O}+t\bm{A}_{\rm{i}}\Vert-\Vert\bm{O}\Vert}{t}=\Vert\bm{A}_{\rm{i}}\Vert$.
Furthermore, the following theorem elucidates that this dual continuation is a dual-valued matrix norm.

\begin{theorem}\label{the: extension dual matrix norm}
Given a matrix norm $\Vert\cdot\Vert:\mathbb{R}^{m\times n}\rightarrow\mathbb{R}$. 
Then its dual continuation $\Vert\cdot\Vert:\mathbb{DR}^{m\times n}\rightarrow\mathbb{DR}$ satisfying
\begin{align}
\Vert\bm{A}\Vert = \left\{\begin{array}{ll}
\Vert\bm{A}_{\rm{s}}\Vert + \max\limits_{\bm{G}\in\partial\Vert\bm{A}_{\rm{s}}\Vert}\langle\bm{G},\bm{A}_{\rm{i}}\rangle\epsilon,  & \bm{A}_{\rm{s}}\neq\bm{O}, \\
\Vert\bm{A}_{\rm{i}}\Vert\epsilon,& \bm{A}_{\rm{s}}=\bm{O},
\end{array}\right.\label{matrix norm expansion}
\end{align}
for $\bm{A}=\bm{A}_{\rm{s}}+\bm{A}_{\rm{i}}\epsilon\in\mathbb{DR}^{m\times n}$ is a dual-valued matrix norm.
\end{theorem}

\begin{proof}
As with proof of the dual-valued vector norm in \Cref{theorem:vector norm}, we induce that \cref{matrix norm expansion} is a dual-valued matrix norm in dual algebra. 
\end{proof}

\subsection{Properties}
In light of the general form of the dual-valued matrix norm, an investigation into its properties is now undertaken.
Our first goal is to validate that the extended dual-valued matrix norm remains the mutually consistent property.

\begin{theorem}\label{the: mutually consistent}
Let $\Vert\cdot\Vert_{m\times n}$ be a matrix norm on $m$-by-$n$ matrices. 
If matrix norms $\Vert\cdot\Vert_{m\times n}$, $\Vert\cdot\Vert_{n\times p}$ and $\Vert\cdot\Vert_{m\times p}$ are mutually consistent, i.e., $\Vert\bm{XY}\Vert_{m\times p}\leq\Vert\bm{X}\Vert_{m\times n}\Vert\bm{Y}\Vert_{n\times p}$, then their corresponding dual continuations are mutually consistent.
\end{theorem}

\begin{proof}
Given dual matrices $\bm{A}=\bm{A}_{\rm{s}}+\bm{A}_{\rm{i}}\epsilon\in\mathbb{DR}^{m\times n}$ and $\bm{B}=\bm{B}_{\rm{s}}+\bm{B}_{\rm{i}}\epsilon\in\mathbb{DR}^{n\times p}$.
\Cref{the: extension dual matrix norm} yields that $\Vert\bm{AB}\Vert_{m\times p} =\Vert\bm{A}_{\rm{s}}\bm{B}_{\rm{s}}\Vert_{m\times p} + {\mathcal{D}}_{\bm{A}_{\rm{s}}\bm{B}_{\rm{i}}+\bm{A}_{\rm{i}}\bm{B}_{\rm{s}}}\Vert\bm{A}_{\rm{s}}\bm{B}_{\rm{s}}\Vert_{m\times p}\epsilon$ and 
\begin{align*}
&\Vert\bm{A}\Vert_{m\times n}\Vert\bm{B}\Vert_{n\times p}
=\big(\Vert\bm{A}_{\rm{s}}\Vert_{m\times n}+ {\mathcal{D}}_{\bm{A}_{\rm{i}}}\Vert\bm{A}_{\rm{s}}\Vert_{m\times n}\epsilon\big)\big(\Vert\bm{B}_{\rm{s}}\Vert_{n\times p}+ {\mathcal{D}}_{\bm{B}_{\rm{i}}}\Vert\bm{B}_{\rm{s}}\Vert_{n\times p}\epsilon\big)\\
=& \Vert\bm{A}_{\rm{s}}\Vert_{m\times n}\Vert\bm{B}_{\rm{s}}\Vert_{n\times p} + \big(\Vert\bm{A}_{\rm{s}}\Vert_{m\times n}\cdot{\mathcal{D}}_{\bm{B}_{\rm{i}}}\Vert\bm{B}_{\rm{s}}\Vert_{n\times p}+{\mathcal{D}}_{\bm{A}_{\rm{i}}}\Vert\bm{A}_{\rm{s}}\Vert_{m\times n}\cdot\Vert\bm{B}_{\rm{s}}\Vert_{n\times p}\big)\epsilon.
\end{align*}

If $\Vert\bm{A}_{\rm{s}}\bm{B}_{\rm{s}}\Vert_{m\times p}<\Vert\bm{A}_{\rm{s}}\Vert_{m\times n}\Vert\bm{B}_{\rm{s}}\Vert_{n\times p}$, then $\Vert\bm{A}\bm{B}\Vert_{m\times p}<\Vert\bm{A}\Vert_{m\times n}\Vert\bm{B}\Vert_{n\times p}$. 
Otherwise, if $\Vert\bm{A}_{\rm{s}}\bm{B}_{\rm{s}}\Vert_{m\times p}=\Vert\bm{A}_{\rm{s}}\Vert_{m\times n}\Vert\bm{B}_{\rm{s}}\Vert_{n\times p}$, we obtain
\begin{align*}
&{\mathcal{D}}_{\bm{A}_{\rm{s}}\bm{B}_{\rm{i}}+\bm{A}_{\rm{i}}\bm{B}_{\rm{s}}}\Vert\bm{A}_{\rm{s}}\bm{B}_{\rm{s}}\Vert_{m\times p}
=\lim\limits_{t\downarrow0}\frac{\Vert\bm{A}_{\rm{s}}\bm{B}_{\rm{s}}+t(\bm{A}_{\rm{s}}\bm{B}_{\rm{i}}+\bm{A}_{\rm{i}}\bm{B}_{\rm{s}})\Vert_{m\times p}-\Vert\bm{A}_{\rm{s}}\bm{B}_{\rm{s}}\Vert_{m\times p}}{t}\\
=&\lim\limits_{t\downarrow0}\frac{\Vert\bm{A}_{\rm{s}}(\frac{1}{2}\bm{B}_{\rm{s}}+t\bm{B}_{\rm{i}})+(\frac{1}{2}\bm{A}_{\rm{s}}+t\bm{A}_{\rm{i}})\bm{B}_{\rm{s}})\Vert_{m\times p}-\Vert\bm{A}_{\rm{s}}\Vert_{m\times n}\Vert\bm{B}_{\rm{s}}\Vert_{n\times p}}{t}\\
\overset{{\rm{(I)}}}{\leq}& \Vert\bm{A}_{\rm{s}}\Vert_{m\!\times\! n}\lim\limits_{t\downarrow0}\frac{\Vert\bm{B}_{\rm{s}}+2t\bm{B}_{\rm{i}}\Vert_{n\!\times\! p}\!-\!\Vert\bm{B}_{\rm{s}}\Vert_{n\!\times\! p}}{2t}+\lim\limits_{t\downarrow0}\frac{\Vert\bm{A}_{\rm{s}}+2t\bm{A}_{\rm{i}}\Vert_{m\!\times\! n}\!-\!\Vert\bm{A}_{\rm{s}}\Vert_{m\!\times\! n}}{2t}\Vert\bm{B}_{\rm{s}}\Vert_{n\!\times\! p}\\
=&\Vert\bm{A}_{\rm{s}}\Vert_{m\times n}\cdot{\mathcal{D}}_{\bm{B}_{\rm{i}}}\Vert\bm{B}_{\rm{s}}\Vert_{n\times p}+{\mathcal{D}}_{\bm{A}_{\rm{i}}}\Vert\bm{A}_{\rm{s}}\Vert_{m\times n}\cdot\Vert\bm{B}_{\rm{s}}\Vert_{n\times p},
\end{align*}
where (I) applies the triangle inequality of the matrix norm $\Vert\cdot\Vert_{m\times p}$ and the mutual consistency of given matrix norms.
Consequently, $\Vert\bm{A}\bm{B}\Vert_{m\times p} 
\leq \Vert\bm{A}\Vert_{m\times n}\Vert\bm{B}\Vert_{n\times p}$.
\end{proof}

Subsequently, we intend to demonstrate the equivalence between the dual continuation of a matrix operator norm and the operator norm induced by two dual-valued vector norms.

\begin{theorem}\label{the: dual matrix operator alpha beta norm}
The dual continuation of the matrix operator norm $\Vert\cdot\Vert_{\alpha,\beta}:\mathbb{R}^{m\times n}\rightarrow\mathbb{R}$, namely $\Vert\cdot\Vert_{\alpha,\beta}:\mathbb{DR}^{m\times n}\rightarrow\mathbb{DR}$, is equivalent to the operator norm induced by two dual-valued vector norms $\Vert\cdot\Vert_\alpha:\mathbb{DR}^m\rightarrow\mathbb{DR}$ and $\Vert\cdot\Vert_\beta:\mathbb{DR}^n\rightarrow\mathbb{DR}$.
Specifically, for a dual matrix $\bm{A}=\bm{A}_{\rm{s}}+\bm{A}_{\rm{i}}\epsilon\in\mathbb{DR}^{m\times n}$,
\begin{align}
\Vert\bm{A}\Vert_{\alpha,\beta} = \max\limits_{\substack{\bm{x}\in\mathbb{DR}^{n}\\\bm{x}\neq\bm{0}}}\frac{\Vert\bm{Ax}\Vert_\alpha}{\Vert\bm{x}\Vert_\beta}.\label{eq: dual matrix operator alpha,beta norm}
\end{align}
\end{theorem}
To prove this, we initially introduce the G\^ateaux derivative of the matrix operator norm, provided by Watson in 1992.
\begin{lemma}\cite[Theorem 3]{Watson1992CharacterizationOT}\label{lem: subdiff of operator norm}
Let $\Vert\cdot\Vert_{\alpha}$ and $\Vert\cdot\Vert_{\beta}$ be vector norms on $\mathbb{R}^m$ and $\mathbb{R}^n$, respectively.
Given the operator norm on $m$-by-$n$ matrices defined by $\Vert\bm{A}\Vert_{\alpha,\beta}=\max_{\Vert\bm{x}\Vert_{\beta}=1}\Vert\bm{Ax}\Vert_{\alpha}$.
Then the G\^ateaux derivative of $\Vert\cdot\Vert_{\alpha,\beta}$ at $\bm{A}$ along $\bm{B}\in\mathbb{R}^{m\times n}$ is
\begin{align}
{\mathcal{D}}_{\bm{B}}\Vert\bm{A}\Vert_{\alpha,\beta}=\max\limits_{(\bm{v},\bm{w})\in\Phi(\bm{A})}\bm{w}^\top\bm{B}\bm{v},
\end{align}
where $\Phi(\bm{A})\!=\!\{\bm{v}\in\mathbb{R}^n,\bm{w}\in\mathbb{R}^m\!:\Vert\bm{v}\Vert_{\beta}\!=\!1,\;\bm{Av}\!=\!\Vert\bm{A}\Vert_{\alpha,\beta}\bm{u},\;\Vert\bm{u}\Vert_{\alpha}\!=\!1,\;\bm{w}\in\partial\Vert\bm{u}\Vert_{\alpha}\}$.
\end{lemma}

\begin{proof}[Proof of \Cref{the: dual matrix operator alpha beta norm}]
Since norms $\Vert\cdot\Vert_{\alpha,\beta}$, $\Vert\cdot\Vert_{\alpha}$ and $\Vert\cdot\Vert_{\beta}$ on $\mathbb{R}^{m\times n}$, $\mathbb{R}^m$ and $\mathbb{R}^n$ are mutually consistent, we deduce from \Cref{the: mutually consistent} that their dual continuations are still mutually consistent.
That is, for the given dual matrix $\bm{A}=\bm{A}_{\rm{s}}+\bm{A}_{\rm{i}}\epsilon\in\mathbb{DR}^{m\times n}$, $\Vert\bm{Ax}\Vert_{\alpha}\leq\Vert\bm{A}\Vert_{\alpha,\beta}\Vert\bm{x}\Vert_{\beta}$ holds for any $\bm{x}=\bm{x}_{\rm{s}}+\bm{x}_{\rm{i}}\epsilon\in\mathbb{DR}^n$.
Besides, if $\bm{x}\neq\bm{0}$, then $\Vert\bm{x}\Vert_{\beta}>0$ based on the non-negativity of dual-valued vector norms. 
Thus, 
\begin{align*}
\Vert\bm{A}\Vert_{\alpha,\beta}\geq\max\limits_{\substack{\bm{x}\in\mathbb{DR}^{n}\\\bm{x}\neq\bm{0}}}\frac{\Vert\bm{Ax}\Vert_{\alpha}}{\Vert\bm{x}\Vert_\beta}.
\end{align*}

Next, we proceed to demonstrate that the equality sign of the inequality mentioned above can be reached. 
For the case that $\bm{A}_{\rm{s}}=\bm{O}$, let $\bm{x}_{\rm{i}}=\bm{0}$ and $\bm{x}_{\rm{s}}$ satisfy $\Vert\bm{x}_{\rm{s}}\Vert_{\beta}=1$ and $\Vert\bm{A}_{\rm{i}}\bm{x}_{\rm{s}}\Vert_{\alpha}=\Vert\bm{A}_{\rm{i}}\Vert_{\alpha,\beta}$.
Then it is obvious that 
\begin{align*}
\frac{\Vert\bm{Ax}\Vert_{\alpha}}{\Vert\bm{x}\Vert_{\beta}}=\frac{\Vert\bm{A}_{\rm{i}}\bm{x}_{\rm{s}}\Vert_{\alpha}\epsilon}{\Vert\bm{x}_{\rm{s}}\Vert_{\beta}}=\Vert\bm{A}_{\rm{i}}\Vert_{\alpha,\beta}\epsilon=\Vert\bm{A}\Vert_{\alpha,\beta}.    
\end{align*}

Consider the other case that $\bm{A}_{\rm{s}}\neq\bm{O}$.
Let $\bm{x}_{\rm{i}}=\bm{0}$ and $\bm{x}_{\rm{s}}$ satisfy $\Vert\bm{x}_{\rm{s}}\Vert_\beta=1$ and $\Vert\bm{A}_{\rm{s}}\bm{x}_{\rm{s}}\Vert_{\alpha}=\Vert\bm{A}_{\rm{s}}\Vert_{\alpha,\beta}:=\gamma$ $(\gamma\neq0)$.
Denote $\bm{A}_{\rm{s}}\bm{x}_{\rm{s}}:=\gamma\bm{u}$ with $\Vert\bm{u}\Vert_{\alpha}=1$.
Thus,
\begin{align*}
\frac{\Vert\bm{Ax}\Vert_{\alpha}}{\Vert\bm{x}\Vert_\beta}&=\frac{\Vert\bm{A}_{\rm{s}}\bm{x}_{\rm{s}}\Vert_{\alpha}\!+\!{\mathcal{D}}_{\bm{A}_{\rm{i}}\bm{x}_{\rm{s}}}\Vert\bm{A}_{\rm{s}}\bm{x}_{\rm{s}}\Vert_{\alpha}\epsilon}{\Vert\bm{x}_{\rm{s}}\Vert_\beta}=\Vert\bm{A}_{\rm{s}}\Vert_{\alpha,\beta}+\lim_{t\downarrow0}\frac{\Vert\bm{A}_{\rm{s}}\bm{x}_{\rm{s}}\!+\!t\bm{A}_{\rm{i}}\bm{x}_{\rm{s}}\Vert_{\alpha}\!-\!\Vert\bm{A}_{\rm{s}}\bm{x}_{\rm{s}}\Vert_{\alpha}}{t}\epsilon\\
&=\Vert\bm{A}_{\rm{s}}\Vert_{\alpha,\beta}+\gamma\lim_{t\downarrow0}\frac{\Vert\bm{u}+t\frac{1}{\gamma}\bm{A}_{\rm{i}}\bm{x}_{\rm{s}}\Vert_{\alpha} -\Vert\bm{u}\Vert_{\alpha}}{t}\epsilon= \Vert\bm{A}_{\rm{s}}\Vert_{\alpha,\beta}+\gamma{\mathcal{D}}_{\frac{1}{\gamma}\bm{A}_{\rm{i}}\bm{x}_{\rm{s}}}\Vert\bm{u}\Vert_{\alpha}\epsilon\\
&=\Vert\bm{A}_{\rm{s}}\Vert_{\alpha,\beta}+\max\limits_{\bm{w}\in\partial\Vert\bm{u}\Vert_{\alpha}}\big\langle\bm{w},\bm{A}_{\rm{i}}\bm{x}_{\rm{s}}\big\rangle\epsilon\overset{\rm{(I)}}{=}\Vert\bm{A}_{\rm{s}}\Vert_{\alpha,\beta}+{\mathcal{D}}_{\bm{A}_{\rm{i}}}\Vert\bm{A}_{\rm{s}}\Vert_{\alpha,\beta}\epsilon=\Vert\bm{A}\Vert_{\alpha,\beta},
\end{align*}
where (I) is established by \Cref{lem: subdiff of operator norm}.
Therefore, the result follows.
\end{proof}

Moreover, we aim to verify that the dual continuation of a unitarily invariant matrix norm is still unitarily invariant in dual algebra.
In particular, given a dual-valued matrix norm $\Vert\cdot\Vert:\mathbb{DR}^{m\times n}\rightarrow\mathbb{DR}$.
If $\Vert\bm{PXQ}\Vert=\Vert\bm{X}\Vert$ for any $\bm{X}\in\mathbb{DR}^{m\times n}$, as well as any orthogonal $\bm{P}\in\mathbb{DR}^{m\times m}$ and $\bm{Q}\in\mathbb{DR}^{n\times n}$, then the dual-valued matrix norm is said to be unitarily invariant in dual algebra.

\begin{theorem}\label{the: unitarily invariant dual matrix norm}
The dual continuation of a unitarily invariant matrix norm is still unitarily invariant in dual algebra.
\end{theorem}

An essential lemma is introduced to prove this.

\begin{lemma}\label{lemma: unitarily matrix norm} \cite[Theorem 1]{Watson1992CharacterizationOT}
Let $\interleave\cdot\interleave:\mathbb{R}^{m\times n}\rightarrow\mathbb{R}$ be a unitarily invariant norm, and $\phi:\mathbb{R}^p\rightarrow\mathbb{R}$ $(p=\min\{m,n\})$ be the corresponding symmetric gauge function. 
Then the G\^ateaux derivative of the norm at $\bm{X}\in\mathbb{R}^{m\times n}$ along $\bm{Y}\in\mathbb{R}^{m\times n}$ is
\begin{align}
\mathcal{D}_{\bm{Y}}\interleave\bm{X}\interleave= \max_{\bm{d}\in\partial\phi(\bm{\sigma}(\bm{X}))}\sum_{k=1}^p d_k\bm{u}_k^\top\bm{Y}\bm{v}_k,
\end{align}
where $\bm{X} = \bm{U\Sigma V}^\top$ is a full SVD satisfying $\bm{U} = [\bm{u}_1,\cdots,\bm{u}_m]$ and $\bm{V} = [\bm{v}_1,\cdots,\bm{v}_n]$.
Besides, $\bm{\sigma}(\bm{X}) = [\sigma_1,\cdots,\sigma_p]^\top$ consists of the singular values of $\bm{X}$.
\end{lemma}

\begin{proof}[Proof of \Cref{the: unitarily invariant dual matrix norm}]
Given orthogonal dual matrices $\bm{P}=\bm{P}_{\rm{s}}+\bm{P}_{\rm{i}}\epsilon\in\mathbb{DR}^{m\times m}$ and $\bm{Q}=\bm{Q}_{\rm{s}}+\bm{Q}_{\rm{i}}\epsilon\in\mathbb{DR}^{n\times n}$.
Then $\bm{P}_{\rm{s}}$ and $\bm{Q}_{\rm{s}}$ are orthogonal matrices, and both $\bm{P}_{\rm{s}}^\top\bm{P}_{\rm{i}}$ and $\bm{Q}_{\rm{s}}\bm{Q}_{\rm{i}}^\top$ are skew-symmetric. 
Let $\interleave\cdot\interleave:\mathbb{R}^{m\times n}\rightarrow\mathbb{R}$ be a unitarily invariant matrix norm and $\interleave\cdot\interleave:\mathbb{DR}^{m\times n}\rightarrow\mathbb{DR}$ be its dual continuation, as outlined in \Cref{the: extension dual matrix norm}.
Given $\bm{X}=\bm{X}_{\rm{s}}+\bm{X}_{\rm{i}}\epsilon\in\mathbb{DR}^{m\times n}$.

On the one hand, if $\bm{X}_{\rm{s}}=\bm{O}$, then $\interleave\bm{X}\interleave = \interleave\bm{X}_{\rm{i}}\epsilon\interleave=\interleave\bm{X}_{\rm{i}}\interleave\epsilon$ and $\interleave\bm{PXQ}\interleave = \interleave\bm{P}_{\rm{s}}\bm{X}_{\rm{i}}\bm{Q}_{\rm{s}}\epsilon\interleave = \interleave\bm{P}_{\rm{s}}\bm{X}_{\rm{i}}\bm{Q}_{\rm{s}}\interleave\epsilon=\interleave\bm{X}_{\rm{i}}\interleave\epsilon$. Hence, $\interleave\bm{PXQ}\interleave=\interleave\bm{X}\interleave$.

Conversely, consider the case that $\bm{X}_{\rm{s}}\neq\bm{O}$. 
Let $\bm{X}_{\rm{s}} = \bm{U\Sigma V}^\top$ be a full SVD with $\bm{U} = [\bm{u}_1,\cdots,\bm{u}_m]$ and $\bm{V} = [\bm{v}_1,\cdots,\bm{v}_n]$.
Besides, $\bm{\sigma}(\bm{X}_{\rm{s}}) = [\sigma_1,\cdots,\sigma_p]^\top$ $(p=\min\{m,n\})$ represents the singular values of $\bm{X}_{\rm{s}}$. 
Then it yields from \Cref{lemma: unitarily matrix norm} that
\begin{align*}
&{\mathcal{D}}_{\bm{P}_{\rm{i}}\bm{X}_{\rm{s}}\bm{Q}_{\rm{s}}+\bm{P}_{\rm{s}}\bm{X}_{\rm{i}}\bm{Q}_{\rm{s}}+\bm{P}_{\rm{s}}\bm{X}_{\rm{s}}\bm{Q}_{\rm{i}}}\interleave\bm{P}_{\rm{s}}\bm{X}_{\rm{s}}\bm{Q}_{\rm{s}}\interleave={\mathcal{D}}_{\bm{P}_{\rm{s}}^\top\bm{P}_{\rm{i}}\bm{X}_{\rm{s}}+\bm{X}_{\rm{i}}+\bm{X}_{\rm{s}}\bm{Q}_{\rm{i}}\bm{Q}_{\rm{s}}^\top}\interleave\bm{X}_{\rm{s}}\interleave\\
=& \max\limits_{\bm{d}\in\partial\phi(\bm{\sigma}(\bm{X}_{\rm{s}}))}\sum_{k=1}^pd_k\bm{u}_k^\top\big(\bm{P}_{\rm{s}}^\top\bm{P}_{\rm{i}}\bm{X}_{\rm{s}}+\bm{X}_{\rm{i}}+\bm{X}_{\rm{s}}\bm{Q}_{\rm{i}}\bm{Q}_{\rm{s}}^\top\big)\bm{v}_k\\
=& \max\limits_{\bm{d}\in\partial\phi(\bm{\sigma}(\bm{X}_{\rm{s}}))}\sum_{k=1}^p\Big(d_k\sigma_k\bm{u}_k^\top\bm{P}_{\rm{s}}^\top\bm{P}_{\rm{i}}\bm{u}_k + d_k\bm{u}_k^\top\bm{X}_{\rm{i}}\bm{v}_k+d_k\sigma_k\bm{v}_k^\top\bm{Q}_{\rm{i}}\bm{Q}_{\rm{s}}^\top\bm{v}_k\Big)\\
\overset{\rm{(I)}}{=}& \max\limits_{\bm{d}\in\partial\phi(\bm{\sigma}(\bm{X}_{\rm{s}}))}\sum_{k=1}^pd_k\bm{u}_k^\top\bm{X}_{\rm{i}}\bm{v}_k={\mathcal{D}}_{\bm{X}_{\rm{i}}}\interleave\bm{X}_{\rm{s}}\interleave,
\end{align*}
where (I) results from skew-symmetric matrices $\bm{P}_{\rm{s}}^\top\bm{P}_{\rm{i}}$ and  $\bm{Q}_{\rm{i}}\bm{Q}_{\rm{s}}^\top$.
In conclusion, we deduce that $\interleave\bm{PXQ}\interleave = \interleave\bm{P}_{\rm{s}}\bm{X}_{\rm{s}}\bm{Q}_{\rm{s}}\interleave+{\mathcal{D}}_{\bm{P}_{\rm{i}}\bm{X}_{\rm{s}}\bm{Q}_{\rm{s}}+\bm{P}_{\rm{s}}\bm{X}_{\rm{i}}\bm{Q}_{\rm{s}}+\bm{P}_{\rm{s}}\bm{X}_{\rm{s}}\bm{Q}_{\rm{i}}}\interleave\bm{P}_{\rm{s}}\bm{X}_{\rm{s}}\bm{Q}_{\rm{s}}\interleave\epsilon = \interleave\bm{X}_{\rm{s}}\interleave+{\mathcal{D}}_{\bm{X}_{\rm{i}}}\interleave\bm{X}_{\rm{s}}\interleave\epsilon = \interleave\bm{X}\interleave$.
\end{proof}

\subsection{Unitarily Invariant Dual-Valued Matrix Norms}

Symmetric gauge functions lead to several examples of unitarily invariant matrix norms. 
Ky Fan $p$-$k$-norm is one of the most significant unitarily invariant norms.
Given a matrix $\bm{X}\in\mathbb{R}^{m\times n}$ $(m\!\geq\!n)$ and its singular values $\sigma_1\geq\cdots\geq\sigma_n$.
Then the Ky Fan $p$-$k$-norm $(1\!\leq\! k\!\leq\! n)$ is defined as $\Vert\bm{X}\Vert_{(k,p)} \!=\! \big(\sum_{i=1}^k\sigma_i^p\big)^{\frac{1}{p}}$ for $1\leq p<\infty$ and $\Vert\bm{X}\Vert_{(k,\infty)} \!=\! \sigma_1$.

The first particular case of the Ky Fan $p$-$k$-norm with $k=n$ is known as the Schatten $p$-norm, expressed as $\Vert\bm{X}\Vert_{S_p}\! =\! \left(\sum_{i=1}^n\sigma_i^p\right)^{\frac{1}{p}}$ for $1\!\leq\! p\!<\!\infty$ and $\Vert\bm{X}\Vert_{S_{\infty}} \!=\! \sigma_1$.
The second one with $p\!=\!1$, namely Ky Fan $k$-norm, is denoted as $\Vert\bm{X}\Vert_{(k)}\!=\!\sum_{i=1}^k\!\sigma_i$.

It is worth mentioning that the Ky Fan $k$-norm becomes the matrix spectral norm for $k=1$ and the nuclear norm for $k=n$.
Besides, when $p=1,2,\infty$, the Schatten $p$-norm becomes the nuclear norm, the Frobenius norm and the spectral norm, respectively.
Specifically, $\Vert\bm{X}\Vert_{S_1}=\Vert\bm{X}\Vert_{(n)}=\Vert\bm{X}\Vert_*$, $\Vert\bm{X}\Vert_{S_\infty}=\Vert\bm{X}\Vert_{(1)}=\Vert\bm{X}\Vert_2$, $\Vert\bm{X}\Vert_{S_2}=\Vert\bm{X}\Vert_F$, $\Vert\bm{X}\Vert_{(n,p)}=\Vert\bm{X}\Vert_{S_p}$, and $\Vert\bm{X}\Vert_{(k,1)}=\Vert\bm{X}\Vert_{(k)}$.

Accordingly, the following part establishes the dual-valued Ky Fan $p$-$k$-norm, which is unitarily invariant based on \Cref{the: unitarily invariant dual matrix norm}.
Furthermore, it is demonstrated that the dual-valued Ky Fan $p$-$k$-norm corresponds to the dual-valued vector $p$-norm of the first $k$ singular values of the dual matrix.

The following propositions separately present the dual-valued Ky Fan $p$-$k$-norms with $1\!<\!p\!<\!\infty$, $p\!=\!1$ and $p\!=\!\infty$, denoted as $\Vert\cdot\Vert_{(k,p)}$, $\Vert\cdot\Vert_{(k)}$ and $\Vert\cdot\Vert_2$, where the last two are referred to as dual-valued Ky Fan $k$-norm and dual-valued spectral norm.

\begin{proposition}
\label{pro: dual ky fan p-k-norm}
Given a dual matrix $\bm{A} = \bm{A}_{\rm{s}}+\bm{A}_{\rm{i}}\epsilon\in\mathbb{DR}^{m\times n}$ $(m\geq n)$.
Then the dual-valued Ky Fan $p$-$k$-norm $(1< p< \infty)$ $(1\leq k\leq n)$ $\Vert\cdot\Vert_{(k,p)}:\mathbb{DR}^{m\times n}\rightarrow\mathbb{DR}$ of $\bm{A}$ extending from the Ky Fan $p$-$k$-norm $\Vert\cdot\Vert_{(k,p)}:\mathbb{R}^{m\times n}\rightarrow\mathbb{R}$ is expressed as
\begin{align}
\Vert\bm{A}\Vert_{(k,p)}\!=\!\left\{\!\!\!\begin{array}{ll}
\Vert\bm{A}_{\rm{s}}\Vert_{(k,p)}\!+\!\frac{1}{\Vert\bm{A}_{\rm{s}}\Vert_{(k,p)}^{p-1}}\Big(\big\langle\bm{U}_1\bm{\Sigma}_1^{p-1}\bm{V}_1^\top,\bm{A}_{\rm{i}}\big\rangle\!+\!\sigma_k^{p-1}\sum_{\ell=1}^t\lambda_\ell(\bm{M})\Big)\epsilon, &\!\! \bm{A}_{\rm{s}}\!\neq\!\bm{O},\\
\Vert\bm{A}_{\rm{i}}\Vert_{(k,p)}\epsilon, &\!\! \bm{A}_{\rm{s}}\!=\!\bm{O},
\end{array}\right.\label{eq: ky fan k-p-norm}
\end{align}
where the non-zero $\bm{A}_{\rm{s}}$ is such that the singular values satisfy 
\begin{align}
\sigma_1\geq\cdots\geq\sigma_{k-t}>\sigma_{k-t+1}=\cdots=\sigma_k=\cdots=\sigma_{k+s}>\sigma_{k+s+1}\geq\cdots\geq\sigma_{n}.\label{eq: singular values of As for k-p-norm}
\end{align}
Let $\bm{A}_{\rm{s}}=\bm{U\Sigma V}^\top$ be a full SVD with $\bm{U}_1=\bm{U}(:,1:k-t)$, $\bm{U}_2=\bm{U}(:,k-t+1:k+s)$, $\bm{V}_1=\bm{V}(:,1:k-t)$, $\bm{V}_2=\bm{V}(:,k-t+1:k+s)$ and $\bm{\Sigma}_1=\bm{\Sigma}(1:k-t,1:k-t)$.
Moreover, $\bm{M}={\rm{sym}}(\bm{U}_2^\top\bm{A}_{\rm{i}}\bm{V}_2)$ and $\lambda_1(\bm{M})\geq\cdots\geq\lambda_{t+s}(\bm{M})$ are eigenvalues of $\bm{M}$.
\end{proposition}

\begin{proposition}
\label{pro: dual ky fan}
Given a dual matrix $\bm{A} = \bm{A}_{\rm{s}}+\bm{A}_{\rm{i}}\epsilon\in\mathbb{DR}^{m\times n}$ $(m\geq n)$.
Then the dual-valued Ky Fan $k$-norm $(1\leq k\leq n)$ $\Vert\cdot\Vert_{(k)}:\mathbb{DR}^{m\times n}\rightarrow\mathbb{DR}$ extending from the Ky Fan $k$-norm $\Vert\cdot\Vert_{(k)}:\mathbb{R}^{m\times n}\rightarrow\mathbb{R}$ is expressed as
\begin{align}
\Vert\bm{A}\Vert_{(k)}=\left\{\begin{array}{ll}
\Vert\bm{A}_{\rm{s}}\Vert_{(k)}\!+\!\left(\!\langle\bm{U}_1\bm{V}_1^\top,\bm{A}_{\rm{i}}\rangle+\sum_{\ell=1}^t\lambda_\ell(\bm{M})\!\right)\epsilon, & \bm{A}_{\rm{s}}\neq\bm{O},\sigma_k\neq0, \\
\Vert\bm{A}_{\rm{s}}\Vert_{(k)}\!+\!\left(\!\langle\bm{U}_1\bm{V}_1^\top,\bm{A}_{\rm{i}}\rangle+\sum_{\ell=1}^t\sigma_\ell(\bm{N})\!\right)\epsilon, & \bm{A}_{\rm{s}}\neq\bm{O},\sigma_k=0,\\
\Vert\bm{A}_{\rm{i}}\Vert_{(k)}\epsilon, & \bm{A}_{\rm{s}}=\bm{O},
\end{array}\right.\label{eq: ky fan k-norm}
\end{align}
where the non-zero $\bm{A}_{\rm{s}}$ has the same properties as in \Cref{pro: dual ky fan p-k-norm}, $\widetilde{\bm{U}}_1\!=\!\bm{U}(:,k-t+1\!:\!m)$, $\bm{N}\!=\!\widetilde{\bm{U}}_1^\top\bm{A}_{\rm{i}}\bm{V}_2$, and $\sigma_1(\bm{N})\!\geq\!\cdots\!\geq\!\sigma_{t+s}(\bm{N})$ are singular values of $\bm{N}$.
\end{proposition}

\begin{proposition}\label{pro: dual matrix spectral norm}
Given a dual matrix $\bm{A} = \bm{A}_{\rm{s}}+\bm{A}_{\rm{i}}\epsilon\in\mathbb{DR}^{m\times n}$ $(m\geq n)$.
Then the dual-valued spectral norm $\Vert\cdot\Vert_2:\mathbb{DR}^{m\times n}\rightarrow \mathbb{DR}$ of $\bm{A}$ extending from the spectral norm $\Vert\cdot\Vert_2:\mathbb{R}^{m\times n}\rightarrow \mathbb{R}$ is expressed as
\begin{align}
\Vert\bm{A}_{\rm{s}}+\bm{A}_{\rm{i}}\epsilon\Vert_2 = \left\{\begin{array}{ll}
\Vert \bm{A}_{\rm{s}} \Vert_2 + \lambda_{\rm{max}}(\bm{R})\epsilon,  & \bm{A}_{\rm{s}}\neq \bm{O}, \\
\Vert \bm{A}_{\rm{i}} \Vert_2\epsilon, & \bm{A}_{\rm{s}}=\bm{O},
\end{array}\right.\label{dual spectral norm}
\end{align}
where $\bm{R}={\rm{sym}}(\bm{U}_{r_1}^\top\bm{A}_{\rm{i}}\bm{V}_{r_1})$ and $\lambda_{\rm{max}}(\bm{R})$ represents the maximum eigenvalue of $\bm{R}$. 
Besides, $\bm{U}_{r_1}\in\mathbb{R}^{m\times r_1}$ and $\bm{V}_{r_1}\in\mathbb{R}^{n\times r_1}$ are the left and right singular vectors of $\bm{A}_{\rm{s}}$ corresponding to the largest singular value with multiplicity $r_1$, respectively. 
\end{proposition}

To obtain the proofs of the three propositions mentioned above, it is necessary to develop the subdifferential of the Ky Fan $p$-$k$-norm $(1\!\leq\! p\!\leq\!\infty)$.
Note that the Ky Fan $p$-$k$-norm is a unitarily invariant norm and its corresponding symmetric gauge function is the $p$-norm of the first $k$ singular values.
Hence, we first introduce the subdifferential of a unitarily invariant norm.

\begin{lemma}\label{lem:partial of unitarily invariant matrix norm}\cite[Theorem 2]{Watson1992CharacterizationOT}\cite[Corollary 3.3, Theorem 8.3]{zhang2015singular}
Given a unitarily invariant norm $\interleave\cdot\interleave$ on $\mathbb{R}^{m\times n}$, and its corresponding symmetric gauge function $\phi$.
Let $\bm{X}\in\mathbb{R}^{m\times n}$ have a full SVD $\bm{X}=\bm{U}\bm{\Sigma}\bm{V}^\top$ with ${\rm{Rank}}(\bm{X})=r$ and $\bm{\sigma} = {\rm{diag}}(\bm{\Sigma})$.
Suppose that there are $\rho$ distinct values among the nonzero singular values with respective multiplicities $r_i$ (satisfying $\sum_{i=1}^\rho r_i=r$).
Then
\begin{align*}
\partial\interleave\bm{X}\interleave &={\rm{conv}}\big\{\bm{UDV}^\top:
\bm{d}\in\partial\phi(\bm{\sigma}),\bm{D} = {\rm{diag}}(\bm{d})\big\}\\
&= \left\{\sum_{\ell}\beta_{\ell}\bm{U}\bm{P}^\ell\bm{D}^\ell(\bm{Q}^\ell)^\top\bm{V}^\top:\substack{\beta_\ell\geq0,\sum_{\ell}\beta_\ell=1,\bm{d}_\ell\in\partial\phi(\bm{\sigma}),\bm{D}^\ell={\rm{diag}}(\bm{d}_\ell)\\
\bm{P}^\ell = {\rm{diag}}(\bm{P}_1,\cdots,\bm{P}_\rho,\bm{P}_0),\bm{Q}^\ell = {\rm{diag}}(\bm{P}_1,\cdots,\bm{P}_\rho,\bm{Q}_0)}\right\},
\end{align*}
where ${\rm{conv}}\{\cdot\}$ signifies the convex hull of a set. 
Besides, the symbols in the second equality satisfy $\bm{P}_i\!\in\!\mathbb{R}^{r_i\times r_i}$ $(i\!=\!1,\cdots,\rho)$, $\bm{P}_0\!\in\!\mathbb{R}^{(m-r)\times(m-r)}$ and $\bm{Q}_0\!\in\!\mathbb{R}^{(n-r)\times(n-r)}$ are any orthogonal matrices.
Obviously, $\bm{P}^\ell\bm{\Sigma}(\bm{Q}^\ell)^\top=\bm{\Sigma}$ and $(\bm{P}^\ell)^\top\bm{\Sigma}\bm{Q}^\ell=\bm{\Sigma}$ hold.
\end{lemma}

The objective here is to characterize the subdifferential of $\phi(\bm{\sigma})=(\sum_{i=1}^k\sigma_i^p)^{\frac{1}{p}}$ for $1<p<\infty$, which draws inspiration from the subdifferential of the vector 2-norm of the first $k$ singular values, as proposed in reference \cite[Lemma 2.8]{Doan2016finding}.

\begin{lemma}
\label{lem: subdiff of phi(sigma) 1<p<infty}
Let non-zero $\bm{\sigma}=[\sigma_1,\cdots,\sigma_n]^\top\in\mathbb{R}^n$ satisfy the formula \cref{eq: singular values of As for k-p-norm}.
Suppose that $\phi(\bm{\sigma})=(\sum_{i=1}^k\sigma_i^p)^{\frac{1}{p}}$ with $1\leq k\leq n$ and $1<p<\infty$.
Then $\bm{v}\in\partial\phi(\bm{\sigma})$ if and only if $\bm{v}$ satisfies the following conditions
\begin{enumerate}
\item[${\rm{(i)}}$] $v_i = \frac{\sigma_i^{p-1}}{\phi^{p-1}(\bm{\sigma})}$ for $i=1,\cdots,k-t$;
\item[${\rm{(ii)}}$] $v_i = \alpha_i\frac{\sigma_k^{p-1}}{\phi^{p-1}(\bm{\sigma})}$ for $i=k-t+1,\cdots,k+s$ with $0\leq\alpha_i\leq1$ and $\sum\limits_{i=k-t+1}^{k+s}\alpha_i=t$;
\item[${\rm{(iii)}}$] $v_i = 0$ for $i=k+s+1,\cdots,n$.
\end{enumerate}
\end{lemma}

\begin{proof}
Set $[n]:=\{1,2,\cdots,n\}$.
Denote the collection of all subsets with exactly $k$ elements of $[n]$ as $\mathcal{N}_k$.
Then we obtain
\begin{align*}
\phi(\bm{\sigma})=\max_{N\in\mathcal{N}_k}f_N(\bm{\sigma}),
\end{align*}
where $f_N(\bm{\sigma}):=(\sum_{i\in N}\sigma_i^p)^{\frac{1}{p}}$ for any $N\in\mathcal{N}_k$.
Since $f_N(\bm{\sigma})$ is closed and convex, we derive from the reference \cite[Lemma 3.1.13]{Nest2018convex} that the subdifferential of $\phi(\bm{\sigma})$ is the convex hull of the
subdifferentials of $f_N(\bm{\sigma})$ satisfying $f_N(\bm{\sigma})=\phi(\bm{\sigma})$, i.e.,
\begin{align*}
\partial\phi(\bm{\sigma})={\rm{conv}}\{\partial f_N(\bm{\sigma}):N\in\mathcal{N}_k,f_N(\bm{\sigma})=\phi(\bm{\sigma})\}.
\end{align*}
In consideration of the structure of $\bm{\sigma}$ in \cref{eq: singular values of As for k-p-norm}, it can be inferred that the first $k-t$ elements of the set $N$ for which $f_N(\bm{\sigma})=\phi(\bm{\sigma})$ must be $[k-t]$, while the remaining $t$ elements are selected from the set $\Omega:=\{k-t+1,\cdots,k+s\}$ of size $t+s$.

Furthermore, the subdifferential of the vector $p$-norm, as derived in \Cref{lem: sub of vector p norm} \cref{eq: sub of vector p}, allows us to conclude that
$\frac{\partial f_N(\bm{\sigma})}{\partial\sigma_i}$ equals $\frac{\sigma_i^{p-1}}{\big(f_N(\bm{\sigma})\big)^{p-1}}$ if $i\in N$ and 0 otherwise.
Thus, given $\bm{v}\in\partial\phi(\bm{\sigma})$, results (i) and (iii) in \Cref{lem: subdiff of phi(sigma) 1<p<infty} are derived.

Ultimately, an enumeration of $t$ elements from the set $\Omega$ to construct the set $N$ such that $f_N(\bm{\sigma})=\phi(\bm{\sigma})$ enables the characterization of $v_i$ for $i\in\Omega$ as $v_i=\alpha_i\frac{\sigma_i^{p-1}}{\big(f_N(\bm{\sigma})\big)^{p-1}}$, where $0\leq\alpha_i\leq1$ and $\sum_{i=k-t+1}^{k+s}\alpha_i=t$.

Consequently, we complete the proof.
\end{proof}

We then propose the subdifferential of Ky Fan $p$-$k$-norm for $1<p<\infty$ based on \Cref{lem:partial of unitarily invariant matrix norm,lem: subdiff of phi(sigma) 1<p<infty}.

\begin{lemma}
\label{lem: subdiff of ky fan p-k}
Let $\bm{X}\!\in\!\mathbb{R}^{m\times n}$ $(m\!\geq\! n)$ be a non-zero matrix with singular values $\bm{\sigma}(\bm{X})\!=\![\sigma_1\!,\!\cdots\!,\!\sigma_n]^\top$ satisfying the formula \cref{eq: singular values of As for k-p-norm}.
Let $\bm{X}\!=\!\bm{U\Sigma V}^\top$ be a full SVD, where $\bm{U}_1\!=\!\bm{U}(:,1\!:\!k-t)$, $\bm{U}_2\!=\!\bm{U}(:,k-t+1\!:\!k+s)$, $\bm{V}_1\!=\!\bm{V}(:,1\!:\!k-t)$, $\bm{V}_2\!=\!\bm{V}(:,k\!-\!t\!+\!1\!:\!k\!+\!s)$ and $\bm{\Sigma}_1\!=\!\bm{\Sigma}(1\!:\!k-t,1\!:\!k-t)$.
Then $\bm{G}\in\partial\Vert\bm{X}\Vert_{(k,p)}$ if there exists a symmetric positive semidefinite $\bm{T}\!\in\!\mathbb{R}^{(t+s)\times(t+s)}$ satisfying $\Vert\bm{T}\Vert_2\!\leq\!1$ and $\Vert\bm{T}\Vert_*\!=\!t$ such that 
\begin{align}
\bm{G}=\frac{1}{\Vert\bm{X}\Vert_{(k,p)}^{p-1}}\big(\bm{U}_1\bm{\Sigma}_1^{p-1}\bm{V}_1^\top+\sigma_k^{p-1}\bm{U}_2\bm{T}\bm{V}_2^\top\big).\label{eq: subdiff of ky fan k-p 1<p<infty}
\end{align}
\end{lemma}

\begin{proof}
Given that $\phi(\bm{\sigma})=(\sum_{i=1}^k\sigma_i^p)^{\frac{1}{p}}=\Vert\bm{X}\Vert_{(k,p)}$.
According to \Cref{lem:partial of unitarily invariant matrix norm,lem: subdiff of phi(sigma) 1<p<infty}, if $\bm{G}\in\partial\Vert\bm{X}\Vert_{(k,p)}$ and $\sigma_k\neq0$, then it yields that 
\begin{small}
\begin{align*}
\bm{G} &=\sum_{\ell}\beta_{\ell}\bm{U}^{\ell}{\rm{diag}}\left(\frac{\sigma_1^{p-1}}{\Vert\bm{X}\Vert_{(k,p)}^{p-1}},\cdots,\frac{\sigma_{k-t}^{p-1}}{\Vert\bm{X}\Vert_{(k,p)}^{p-1}},\frac{\alpha_{k-t+1}\sigma_{k}^{p-1}}{\Vert\bm{X}\Vert_{(k,p)}^{p-1}},\cdots,\frac{\alpha_{k+s}\sigma_{k}^{p-1}}{\Vert\bm{X}\Vert_{(k,p)}^{p-1}},0,\cdots,0\right)
(\bm{V}^{\ell})^\top\\
&=\frac{1}{\Vert\bm{X}\Vert_{(k,p)}^{p-1}} \sum_{\ell}\!\beta_{\ell}\bm{U}_1\bm{P}_{11}^{\ell}\bm{\Sigma}_1^{p-1}(\!\bm{P}_{11}^{\ell}\!)\!^\top\bm{V}_1\!^\top\! + \!\beta_{\ell}\sigma_k^{p-1}\bm{U}_2\bm{P}_{22}^{\ell}\Big[{\rm{diag}}(\alpha_{k-t+1}\!,\!\cdots\!,\!\alpha_{k+s})\Big](\!\bm{P}_{22}^{\ell}\!)\!^\top\bm{V}_2\!^\top\\
&=\frac{1}{\Vert\bm{X}\Vert_{(k,p)}^{p-1}}\left(\bm{U}_1\bm{\Sigma}_1^{p-1}\bm{V}_1^\top+\sigma_k^{p-1}\bm{U}_2\bm{T}\bm{V}_2^\top\right),
\end{align*}
\end{small}
where $\bm{P}_{11}^\ell \!=\! \bm{P}^\ell(1\!:\!k-t,1\!:\!k-t)\!=\!\bm{Q}^\ell(1\!:\!k-t,1\!:\!k-t)$, $\bm{P}_{22}^\ell\! =\! \bm{P}^\ell(k-t+1\!:\!k+s,k-t+1\!:\!k+s)\!=\!\bm{Q}^\ell(k-t+1\!:\!k+s,k-t+1\!:\!k+s)$ and $\bm{P}_{11}^{\ell}\bm{\Sigma}_1^{p-1}(\bm{P}_{11}^{\ell})^\top=\bm{\Sigma}_1^{p-1}$, as outlined in \Cref{lem:partial of unitarily invariant matrix norm}.
Besides, denote $\bm{T}\!=\!\bm{P}_{22}^{\ell}\big[{\rm{diag}}(\alpha_{k-t+1},\cdots,\alpha_{k+s})\big](\!\bm{P}_{22}^{\ell}\!)\!^\top$.
Since $0\!\leq\!\alpha_i\!\leq\!1$ for $i=k-t+1,\cdots,k+s$ and $\sum_{i=k-t+1}^{k+s}\alpha_i=t$ based on \Cref{lem: subdiff of phi(sigma) 1<p<infty} ${\rm{(ii)}}$, we derive that $\bm{T}$ is symmetric positive semidefinite and these $\alpha_i$ are eigenvalues and also singular values.
Thus $\Vert\bm{T}\Vert_2=\max\{\alpha_{k-t+1},\cdots,\alpha_{k+s}\}\leq1$ and $\Vert\bm{T}\Vert_*=\sum_{i=k-t+1}^{k+s}\alpha_{i}=t$.

For the other case that $\sigma_k=0$.
Then $\partial\phi(\bm{\sigma}) = \big\{\big[\frac{\sigma_1^{p-1}}{\Vert\bm{X}\Vert_{(k,p)}^{p-1}},\cdots,\frac{\sigma_{k-t}^{p-1}}{\Vert\bm{X}\Vert_{(k,p)}^{p-1}},0,\cdots,0\big]\big\}$ is a singleton.
Similar to the proof above, we obtain $\partial\Vert\bm{X}\Vert_{(k,p)} = \big\{\frac{\bm{U}_1\bm{\Sigma}_1^{p-1}\bm{V}_1^\top}{\Vert\bm{X}\Vert_{(k,p)}^{p-1}}\big\}$.

In conclusion, if $\bm{G}\in\partial\Vert\bm{X}\Vert_{(k,p)}$, then $\bm{G}$ is expressed as the formula \cref{eq: subdiff of ky fan k-p 1<p<infty} for a special symmetric positive semidefinite $\bm{T}$ with $\Vert\bm{T}\Vert_2\leq1$ and $\Vert\bm{T}\Vert_*=t$.
\end{proof}

We can now demonstrate \Cref{pro: dual ky fan p-k-norm} based on \Cref{the: extension dual matrix norm} and \Cref{lem: subdiff of ky fan p-k}.

\begin{proof}[Proof of \Cref{pro: dual ky fan p-k-norm}]
For the case that $\bm{A}_{\rm{s}}\neq\bm{O}$, \Cref{lem: subdiff of ky fan p-k} implies that
\begin{align*}
\max\limits_{\bm{G}\in\partial\Vert\bm{A}_{\rm{s}}\Vert_{(k,p)}}\langle\bm{G},\bm{A}_{\rm{i}}\rangle
=\frac{1}{\Vert\bm{A}_{\rm{s}}\Vert_{(k,p)}^{p-1}}\big\langle\bm{U}_1\bm{\Sigma}_1^{p-1}\bm{V}_1^\top,\bm{A}_{\rm{i}}\big\rangle
+ \frac{\sigma_k^{p-1}}{\Vert\bm{A}_{\rm{s}}\Vert_{(k,p)}^{p-1}}\max\limits_{\substack{\bm{T}\succcurlyeq0\\\Vert\bm{T}\Vert_2\leq1\\\Vert\bm{T}\Vert_*=t}}\big\langle\bm{U}_2\bm{T}\bm{V}_2^\top,\bm{A}_{\rm{i}}\big\rangle,
\end{align*}
where $\bm{T}\succcurlyeq0$ represents $\bm{T}$ is symmetric positive semidefinite.
Moreover, it yields that
\begin{align*}
&\max\limits_{\substack{\bm{T}\succcurlyeq0,\Vert\bm{T}\Vert_2\leq1,\Vert\bm{T}\Vert_*=t}}\big\langle\bm{U}_2\bm{T}\bm{V}_2^\top,\bm{A}_{\rm{i}}\big\rangle=\max\limits_{\substack{\bm{T}\succcurlyeq0,\Vert\bm{T}\Vert_2\leq1,\Vert\bm{T}\Vert_*=t}}\big\langle\bm{T},{\rm{sym}}(\bm{U}_2^\top\bm{A}_{\rm{i}}\bm{V}_2)\big\rangle\\
=& \max\limits_{\substack{\bm{P}\in\mathcal{O}(t+s)\\0\leq\theta_i\leq1,\sum_{i=1}^{t+s}\theta_i=t}}\big\langle{\rm{diag}}(\theta_{1},\cdots,\theta_{t+s}),\bm{P}^\top\bm{M}\bm{P}\big\rangle=\max\limits_{\bm{P}\in\mathcal{O}(t+s)}\sum_{i=1}^t\big[{\rm{diag}}(\bm{P}^\top\bm{MP})\big]_{i}^{\downarrow},
\end{align*}
where $\bm{M}:={\rm{sym}}(\bm{U}_2^\top\bm{A}_{\rm{i}}\bm{V}_2)$ and $\mathcal{O}(t+s)$ represents the set of $(t+s)$-by-$(t+s)$ orthogonal matrices.
Besides, $\bm{T}=\bm{P}{\rm{diag}}(\theta_{1},\cdots,\theta_{t+s})\bm{P}^\top$ is the eigenvalue decomposition and also the SVD of $\bm{T}$.
Here, the nonincreasing rearrangement vector of the diagonal entries of $\bm{P}^\top\bm{MP}$ is denoted as $\big[{\rm{diag}}(\bm{P}^\top\bm{MP})\big]^\downarrow$, whose list of entries is the same as that of ${\rm{diag}}(\bm{P}^\top\bm{MP})$ (including multiplicities).
In addition, $\big[{\rm{diag}}(\bm{P}^\top\bm{MP})\big]_{i}^{\downarrow}$ denotes the $i$-th element of $\big[{\rm{diag}}(\bm{P}^\top\bm{MP})\big]^{\downarrow}$.

In the light of \cite[Theorem 4.3.45]{horn2012matrix}, it should be noted that 
\begin{align*}
\sum_{i=1}^t\lambda_i(\bm{P}^\top\bm{MP})\geq\sum_{i=1}^t\big[{\rm{diag}}(\bm{P}^\top\bm{MP})\big]_{i}^{\downarrow},
\end{align*}
where $\lambda_1(\bm{P}\!^\top\bm{MP})\!\geq\!\cdots\!\geq\!\lambda_{t+s}(\bm{P}\!^\top\bm{MP})$ are eigenvalues of $\bm{P}\!^\top\bm{MP}$.
Since $\bm{P}$ is orthogonal, $\bm{P}\!^\top\bm{MP}$'s eigenvalues are also $\bm{M}$'s.
Let $\lambda_1(\!\bm{M}\!)\!\geq\!\cdots\!\geq\!\lambda_{t+s}(\!\bm{M}\!)$ be eigenvalues of $\bm{M}$.
Thus, $\sum_{i=1}^t\!\big[{\rm{diag}}(\bm{P}\!^\top\bm{MP})\big]_{i}^{\downarrow}\!\leq\!\sum_{i=1}^t\!\lambda_i(\bm{M})$ for any orthogonal $\bm{P}$.
Moreover, if $\bm{P}$ satisfies $\bm{P}\!^\top\bm{MP}\!=\!{\rm{diag}}\big(\lambda_1(\bm{M}),\!\cdots\!,\lambda_{t+s}(\bm{M})\big)$ is the eigenvalue decomposition, then the inequality mentioned above is in fact an equality.
That is,
\begin{align*}
\max\limits_{\bm{P}\in\mathcal{O}(t+s)}\sum_{i=1}^t\big[{\rm{diag}}(\bm{P}^\top\bm{MP})\big]_{i}^{\downarrow} = \sum_{i=1}^t\lambda_i(\bm{M}).
\end{align*}

Hence, combining the above conclusions, we derive that
\begin{align*}
\max\limits_{\bm{G}\in\partial\Vert\bm{A}_{\rm{s}}\Vert_{(k,p)}}\langle\bm{G},\bm{A}_{\rm{i}}\rangle=\frac{1}{\Vert\bm{A}_{\rm{s}}\Vert_{(k,p)}^{p-1}}\big\langle\bm{U}_1\bm{\Sigma}_1^{p-1}\bm{V}_1^\top,\bm{A}_{\rm{i}}\big\rangle + \frac{\sigma_k^{p-1}}{\Vert\bm{A}_{\rm{s}}\Vert_{(k,p)}^{p-1}} \sum_{i=1}^t\lambda_i(\bm{M}).
\end{align*}
Consequently, the result of \Cref{pro: dual ky fan p-k-norm} follows using \Cref{the: extension dual matrix norm}.
\end{proof}

Subsequently, we introduce the G\^ateaux derivative of the Ky Fan $k$-norm proposed by Watson in 1993 to derive the dual-valued Ky Fan $k$-norm.

\begin{lemma}\cite[Theorem 5]{watson1993matrix}\label{lem: subdiff of ky fan}
Let $\bm{X}\in\mathbb{R}^{m\times n}$ $(m\geq n)$ be a non-zero matrix with singular values $\bm{\sigma}(\bm{X})=[\sigma_1,\cdots,\sigma_n]^\top$ satisfying the formula \cref{eq: singular values of As for k-p-norm}.
Let $\bm{X}=\bm{U\Sigma V}^\top$ be a full SVD with $\bm{U}_1=\bm{U}(:,1:k-t)$, $\bm{U}_2=\bm{U}(:,k-t+1:k+s)$, $\widetilde{\bm{U}}=\bm{U}(:,k-t+1:m)$, $\bm{V}_1=\bm{V}(:,1:k-t)$ and $\bm{V}_2=\bm{V}(:,k-t+1:k+s)$. 
Then the G\^ateaux derivative of the Ky Fan $k$-norm $\Vert\cdot\Vert_{(k)}$ at $\bm{X}$ along the direction $\bm{Y}\in\mathbb{R}^{m\times n}$ fulfills
\begin{align}
{\mathcal{D}}_{\bm{Y}}\Vert\bm{X}\Vert_{(k)} = \left\{\begin{array}{ll}
\langle\bm{U}_1\bm{V}_1^\top,\bm{Y}\rangle+\sum_{\ell=1}^t\lambda_\ell(\bm{M}), & \sigma_k\neq0, \\
\langle\bm{U}_1\bm{V}_1^\top,\bm{Y}\rangle+\sum_{\ell=1}^t\sigma_\ell(\bm{N}), & \sigma_k=0,
\end{array}\right.
\end{align}
where $\bm{M}={\rm{sym}}(\bm{U}_2^\top\bm{Y}\bm{V}_2)$ and $\bm{N}=\widetilde{\bm{U}}_1^\top\bm{Y}\bm{V}_2$.
Besides, $\lambda_1(\bm{M})\geq\cdots\geq\lambda_{t+s}(\bm{M})$ are eigenvalues of $\bm{M}$ and $\sigma_1(\bm{N})\geq\cdots\geq\sigma_{t+s}(\bm{N})$ are singular values of $\bm{N}$.
\end{lemma}

\begin{proof}[Proof of \Cref{pro: dual ky fan}]
The result is based on \Cref{the: extension dual matrix norm} and \Cref{lem: subdiff of ky fan}.
\end{proof}

Finally, the subdifferential of the spectral norm is described and the proof of the dual-valued spectral norm, based on this and \Cref{the: extension dual matrix norm}, is then provided.

\begin{lemma}\label{lem: spectral subdiff}\cite{Watson1992CharacterizationOT}
The subdifferential of the matrix spectral norm $\Vert\cdot\Vert_2$ at the point $\bm{X}\in\mathbb{R}^{m\times n}$ is
\begin{align}
\partial\Vert\bm{X}\Vert_2 &=\left\{\bm{U}_1\bm{H}\bm{V}_1^\top\big\vert  \bm{H}\in\mathbb{R}^{r_1\times r_1}, \bm{H}\succcurlyeq 0, tr(\bm{H})=1\right\}.\label{subdiff spectral}
\end{align}
Let $\bm{X}=\bm{U\Sigma V}^\top$ be a full SVD and the multiplicity of $\sigma_1$ (the largest singular value of $\bm{X}$) be $r_1$, with $\bm{U}_1$ and $\bm{V}_1$ consisting of the first $r_1$ columns of $\bm{U}$ and $\bm{V}$. 
\end{lemma}

\begin{proof}[Proof of \Cref{pro: dual matrix spectral norm}]
It is worth mentioning that any symmetric positive semidefinite matrix $\bm{H}\in\mathbb{R}^{r_1\times r_1}$ can be equivalently written as the product $\bm{P}\bm{P}^\top$, where $\bm{P}\in\mathbb{R}^{r_1\times k}$ for some $k$.
Then we derive form \Cref{lem: spectral subdiff} that
\begin{align*}
\max\limits_{\bm{G}\in\partial\Vert\bm{A}_{\rm{s}}\Vert_2}\langle\bm{G},\bm{A}_{\rm{i}}\rangle&=
\max\limits_{\substack{\bm{H}\in\mathbb{R}^{r_1\times r_1},\\\bm{H}\succcurlyeq0,\;{\rm{tr}}(\bm{H})=1}}\big\langle\bm{U}_1\bm{H}\bm{V}_1^\top,\bm{A}_{\rm{i}}\big\rangle= \max\limits_{\substack{\bm{P}\in\mathbb{R}^{r_1\times k},\;k\leq r_1,\\{\rm{tr}}(\bm{P}\bm{P}^\top)=1}}\big\langle\bm{U}_1\bm{PP}^\top\bm{V}_1^\top,\bm{A}_{\rm{i}}\big\rangle\\
&= \max\limits_{\substack{\bm{O}\neq\bm{P}\in\mathbb{R}^{r_1\times k},\\k\leq r_1}}\frac{{\rm{tr}}\big(\bm{P}^\top(\bm{U}_1^\top\bm{A}_{\rm{i}}\bm{V}_1)\bm{P}\big)}{{\rm{tr}}(\bm{P}^\top\bm{P})}.
\end{align*}
Denote $\bm{R}:={\rm{sym}}\big(\bm{U}_1^\top\bm{A}_{\rm{i}}\bm{V}_1\big)$.
Then ${\rm{tr}}(\bm{P}^\top\bm{R}\bm{P})={\rm{tr}}\big(\bm{P}^\top(\bm{U}_1^\top\bm{A}_{\rm{i}}\bm{V}_1)\bm{P}\big)$.
Since $\bm{R}\in\mathbb{R}^{r_1\times r_1}$ is symmetric, there exists orthogonal $\bm{Q}\in\mathbb{R}^{r_1\times r_1}$ and diagonal $\bm{\Lambda}\in\mathbb{R}^{r_1\times r_1}$ such that $\bm{R} = \bm{Q\Lambda Q}^\top$ with $\bm{\Lambda} = {\rm{diag}}(\lambda_1,\cdots,\lambda_{r_1})$ satisfying $\lambda_1\geq\cdots\geq\lambda_{r_1}$.
Thus, 
\begin{align*}
\max\limits_{\substack{\bm{O}\neq\bm{P}\in\mathbb{R}^{r_1\times k}\\k\leq r_1}}\frac{{\rm{tr}}(\bm{P}^\top\bm{R}\bm{P})}{{\rm{tr}}(\bm{P}^\top\bm{P})}
&\xlongequal{\bm{W}=\bm{Q}^\top\bm{P}}\max\limits_{\substack{\bm{O}\neq\bm{W}\in\mathbb{R}^{r_1\times k}\\k\leq r_1}}\frac{{\rm{tr}}(\bm{W}^\top\bm{\Lambda}\bm{W})}{{\rm{tr}}(\bm{W}^\top\bm{W})}\\
&\xlongequal{\bm{W}=(w_{ij})} \max\limits_{w_{ij}\in\mathbb{R},k\leq r_1}\frac{\sum_{i=1}^{r_1}\sum_{j=1}^{k}\lambda_iw_{ij}^2}{\sum_{i=1}^{r_1}\sum_{j=1}^{k}w_{ij}^2}
\leq\lambda_1.
\end{align*}
Actually, when taking $\bm{W}$ with its other entries equal to zero except $w_{11}\!=\!1$, we have ${\rm{tr}}(\bm{W}^\top\!\bm{\Lambda}\bm{W})\!=\!\lambda_1$ and ${\rm{tr}}(\bm{W}^\top\!\bm{W})\!=\!1$. 
Thus, the final inequality sign can be replaced with the equal sign.
That is, $\max_{\bm{G}\in\partial\Vert\bm{A}_{\rm{s}}\Vert_2}\langle\bm{G},\bm{A}_{\rm{i}}\rangle=\lambda_{\rm{max}}(\bm{R})$ for nonzero $\bm{A}_{\rm{s}}$.

Following \Cref{the: extension dual matrix norm}, the dual-valued spectral norm can thus be established.
\end{proof}

Furthermore, we can deduce other well-known dual-valued matrix norms from the dual-valued Ky Fan $p$-$k$-norm.
These include (i) the dual-valued Schatten $p$-norm $(1<p<\infty)$ by setting $k=n$, (ii) the dual-valued nuclear norm by setting $k=n$ and $p=1$, and (iii) the dual-valued Frobenius norm by setting $k=n$ and $p=2$, as outlined in the following three corollaries.

\begin{corollary}
\label{cor: dual schatten norm}
Given a dual matrix $\bm{A} = \bm{A}_{\rm{s}}+\bm{A}_{\rm{i}}\epsilon\in\mathbb{DR}^{m\times n}$ $(m\geq n)$.
Then the dual-valued Schatten $p$-norm $(1<p<\infty)$ $\Vert\cdot\Vert_{S_p}:\mathbb{DR}^{m\times n}\rightarrow \mathbb{DR}$ of $\bm{A}$ extending from the Schatten $p$-norm $\Vert\cdot\Vert_{S_p}:\mathbb{R}^{m\times n}\rightarrow \mathbb{R}$ is expressed as
\begin{align}
\Vert\bm{A}\Vert_{S_p} =   \left\{\begin{array}{ll}
\Vert \bm{A}_{\rm{s}} \Vert_{S_p}+\frac{1}{\Vert\bm{A}_{\rm{s}}\Vert_{S_p}^{p-1}}\langle \bm{U}_r\bm{\Sigma}_r^{p-1}\bm{V}_r^\top,\bm{A}_{\rm{i}}\rangle\epsilon, & \bm{A}_{\rm{s}}\neq\bm{O},\\
\Vert\bm{A}_{\rm{i}}\Vert_{S_p}\epsilon, & \bm{A}_{\rm{s}}=\bm{O},
\end{array}\right.\label{dual Schatten norm}
\end{align}
where $\bm{A}_{\rm{s}}=\bm{U}_r\bm{\Sigma}_r\bm{V}_r^\top$ is a compact SVD of $\bm{A}_{\rm{s}}$ with $r$ non-zero singular values. 
\end{corollary}

\begin{corollary}\label{cor: dual nuclear norm}
Given a dual matrix $\bm{A} = \bm{A}_{\rm{s}}+\bm{A}_{\rm{i}}\epsilon\in\mathbb{DR}^{m\times n}$ $(m\geq n)$.
Then the dual-valued nuclear norm $\Vert\cdot\Vert_*:\mathbb{DR}^{m\times n}\rightarrow \mathbb{DR}$ of $\bm{A}$ extending from the nuclear norm $\Vert\cdot\Vert_*:\mathbb{R}^{m\times n}\rightarrow \mathbb{R}$ is expressed as
\begin{align}
\Vert\bm{A}\Vert_* = \left\{\begin{array}{ll}
\Vert \bm{A}_{\rm{s}} \Vert_* + \Big(\big\langle\bm{U}_r\bm{V}_r^\top,\bm{A}_{\rm{i}}\big\rangle + \big\Vert\widetilde{\bm{U}}_r^\top\bm{A}_{\rm{i}}\widetilde{\bm{V}}_r\big\Vert_*\Big)\epsilon, &  \bm{A}_{\rm{s}}\neq\bm{O},\\
\Vert \bm{A}_{\rm{i}} \Vert_*\epsilon, &  \bm{A}_{\rm{s}}=\bm{O},
\end{array}\right.\label{dual matrix nuclear norm}
\end{align}
where $\bm{A}_{\rm{s}} = \bm{U\Sigma V}^\top$ is a full SVD and ${\rm{Rank}}(\bm{A}_{\rm{s}})=r$.
In addition, $\bm{U}$ and $\bm{V}$ are partitioned as $\bm{U} =[\bm{U}_r,\widetilde{\bm{U}}_r]$ and $\bm{V} = [\bm{V}_r,\widetilde{\bm{V}}_r]$, with $\bm{U}_r$ and $\bm{V}_r$ both having $r$ columns.
\end{corollary}

\begin{corollary}\label{cor: dual Frobenius norm}
Given a dual matrix $\bm{A} = \bm{A}_{\rm{s}}+\bm{A}_{\rm{i}}\epsilon\in\mathbb{DR}^{m\times n}$ $(m\geq n)$.
Then the dual-valued Frobenius norm $\Vert\cdot\Vert_F:\mathbb{DR}^{m\times n}\rightarrow \mathbb{DR}$ of $\bm{A}$ extending from the Frobenius norm $\Vert\cdot\Vert_F:\mathbb{R}^{m\times n}\rightarrow \mathbb{R}$ is expressed as
\begin{align}
\Vert\bm{A}\Vert_F =   \left\{\begin{array}{ll}
\Vert \bm{A}_{\rm{s}} \Vert_F+\frac{\langle \bm{A}_{\rm{s}},\bm{A}_{\rm{i}}\rangle}{\Vert\bm{A}_{\rm{s}}\Vert_F}\epsilon, & \bm{A}_{\rm{s}}\neq\bm{O},\\
\Vert\bm{A}_{\rm{i}}\Vert_F\epsilon, & \bm{A}_{\rm{s}}=\bm{O}.
\end{array}\right.\label{eq: forbenius norm}
\end{align}
\end{corollary}

Indeed, the subdifferentials of the Schatten $p$-norm, the nuclear norm, and the Frobenius norm can also be employed to infer their dual continuations based on \Cref{the: extension dual matrix norm}, which yield expressions that are the same as those in the above three corollaries.
The detailed proofs are omitted.

Moreover, the dual-valued Frobenius norm in \cref{eq: forbenius norm} is consistent with the result in \cite{qi2022low} and G\^ateaux derivative of the nuclear norm is developed in \cite[Theorem 2.1]{kevckic2004orthogonality}.

So far, several dual-valued matrix norms have been proposed. 
The following remark elucidates the equivalence relationship between these norms.

\begin{remark}
Given a dual matrix $\bm{A} = \bm{A}_{\rm{s}}+\bm{A}_{\rm{i}}\epsilon\in\mathbb{DR}^{m\times n}$ $(m\geq n)$.
Then 
\begin{enumerate}
\item[${\rm{(i)}}$] $\Vert\bm{A}\Vert_{(1,p)} = \Vert\bm{A}\Vert_{(k,\infty)} = \Vert\bm{A}\Vert_{(1)} = \Vert\bm{A}\Vert_{S_\infty} = \Vert\bm{A}\Vert_2$;
\item[${\rm{(ii)}}$] $\Vert\bm{A}\Vert_{(n,1)} = \Vert\bm{A}\Vert_{(n)} = \Vert\bm{A}\Vert_{S_1}=\Vert\bm{A}\Vert_*$;
\item[${\rm{(iii)}}$] $\Vert\bm{A}\Vert_{(n,2)} = \Vert\bm{A}\Vert_{S_2}=\Vert\bm{A}\Vert_F$.
\end{enumerate}
\end{remark}

Inspired by that the Ky Fan $p$-$k$-norm of a matrix equals the vector $p$-norm of its first $k$ singular values, we intend to investigate the same property in dual algebra.
The following theorem corroborates this assertion by drawing upon our prior work on the compact dual singular value decomposition (CDSVD) of dual matrices in \cite{wei2024singular}.

\begin{theorem}
If a dual matrix has a CDSVD, then its dual-valued Ky Fan $p$-$k$-norm $(\!1\!\leq\! p\!\leq\!\infty\!)$ equals the dual-valued vector $p$-norm of its first $k$ singular values.
\end{theorem}

\begin{proof}
Suppose that the non-zero dual matrix $\bm{A}=\bm{A}_{\rm{s}}+\bm{A}_{\rm{i}}\epsilon\in\mathbb{DR}^{m\times n}$ $(m\geq n)$ with ${\rm{Rank}}(\bm{A}_{\rm{s}})=r$ has a CDSVD $\bm{A}=\bm{U\Sigma V}^\top$, where $\bm{U}=\bm{U}_{\rm{s}}+\bm{U}_{\rm{i}}\epsilon\in\mathbb{DR}^{m\times r}$, $\bm{V}=\bm{V}_{\rm{s}}+\bm{V}_{\rm{i}}\epsilon\in\mathbb{DR}^{n\times r}$ have orthogonal columns and $\bm{\Sigma} = \bm{\Sigma}_{\rm{s}}+\bm{\Sigma}_{\rm{i}}\epsilon\in\mathbb{DR}^{r\times r}$ is diagonal positive.
Besides, denote 
\begin{align*}
\bm{\Sigma}_{\rm{s}}&= {\rm{diag}}(\Tilde{\sigma}_1\bm{I}_{r_1},\cdots,\Tilde{\sigma}_d\bm{I}_{r_d}) = {\rm{diag}}(\sigma^1_{\rm{s}},\cdots,\sigma_{\rm{s}}
^r),\\
\bm{\Sigma}_{\rm{i}}&= {\rm{diag}}(\sigma_{1,1},\cdots,\sigma_{1,r_1},\cdots,\sigma_{d,1},\cdots,\sigma_{d,r_d})={\rm{diag}}(\sigma^1_{\rm{i}},\cdots,\sigma_{\rm{i}}
^r),
\end{align*}
where $\Tilde{\sigma}_1>\cdots>\Tilde{\sigma}_d>0$, $\sum_{\ell=1}^d r_\ell=r$, and $\sigma_{\ell,1}\geq\cdots\geq\sigma_{\ell,r_\ell}$ for $\ell=1,\cdots,d$.
Let $\bm{\sigma}^k=\bm{\sigma}_{\rm{s}}^k+\bm{\sigma}_{\rm{i}}^k\epsilon\in\mathbb{DR}^k$ represent the first $k$ singular values of $\bm{A}$ with $\bm{\sigma}_{\rm{s}}^k=[\sigma^1_{\rm{s}},\cdots,\sigma_{\rm{s}}
^k]^\top$ and $\bm{\sigma}_{\rm{i}}^k=[\sigma^1_{\rm{i}},\cdots,\sigma_{\rm{i}}
^k]^\top$.

Assume that the $k$-th singular value of $\bm{A}_{\rm{s}}$ satisfies $\sigma_{\rm{s}}^k=\Tilde{\sigma}_\theta$ $(1\leq\theta\leq d)$ with $k-\sum_{\ell=1}^{\theta-1}r_\ell=t$ and $\sum_{\ell=1}^{\theta}r_\ell-k=s$.
Hence, the multiplicity of $\sigma_{\rm{s}}^k$ is $t+s$, which is equal to $r_\theta$.
Moreover, $\bm{U}_{\rm{s}}$ and $\bm{V}_{\rm{s}}$ are partitioned as $\bm{U}_{\rm{s}}=[\bm{U}_1,\bm{U}_2,\bm{U}_3]$ and $\bm{V}_{\rm{s}}=[\bm{V}_1,\bm{V}_2,\bm{V}_3]$, where $\bm{U}_1=\bm{U}_{\rm{s}}(:,1:k-t)$, $\bm{U}_2=\bm{U}_{\rm{s}}(:,k-t+1:k+s)$, $\bm{V}_1=\bm{V}_{\rm{s}}(:,1:k-t)$, $\bm{V}_2=\bm{V}_{\rm{s}}(:,k-t+1:k+s)$ and $\bm{\Sigma}_1=\bm{\Sigma}_{\rm{s}}(1:k-t,1:k-t)$.
Denote $\bm{W}={\rm{sym}}(\bm{U}_1^\top\bm{A}_{\rm{i}}\bm{V}_1)$ and $\bm{M}={\rm{sym}}(\bm{U}_2^\top\bm{A}_{\rm{i}}\bm{V}_2)$.
Then the CDSVD \cite[Theorem 3.1]{wei2024singular} implies that
\begin{align*}
{\rm{Diag}}(\bm{W}) &= {\rm{diag}}(\sigma_{1,1},\cdots,\sigma_{\theta-1,r_{\theta-1}})={\rm{diag}}(\sigma_{\rm{i}}^1,\cdots,\sigma_{\rm{i}}^{k-t}),\\
\bm{M} &= {\rm{diag}}(\sigma_{\theta,1},\cdots,\sigma_{\theta,r_{\theta}})={\rm{diag}}(\sigma_{\rm{i}}^{k-t+1},\cdots,\sigma_{\rm{i}}^{k+s}).
\end{align*}
Besides, $\sigma_{\theta,1}\geq\cdots,\geq\sigma_{\theta,r_{\theta}}$ are eigenvalues of $\bm{M}$.

We then demonstrate that $\Vert\bm{A}\Vert_{(k,p)}=\Vert\bm{\sigma}^k\Vert_p$ for $1\leq k\leq r$ and $1\leq p\leq\infty$, in view of $\Vert\bm{A}_{\rm{s}}\Vert_{(k,p)}=\Vert\bm{\sigma}_{\rm{s}}^k\Vert_p$, $\bm{A}_{\rm{s}}\neq\bm{O}$ and $\sigma^k_{\rm{s}}>0$.

\textbf{Case 1:} $p=1$.
Since $\bm{\sigma}_{\rm{s}}^k\!>\!\bm{0}$, it follows that ${\rm{sign}}(\bm{\sigma}_{\rm{s}}^k)\!=\!\mathds{1}_k$, where $\mathds{1}_k$ denotes an all-one vector of size $k$-by-1.
Besides, the set of indices where $\bm{\sigma}_{\rm{s}}^k$ is zero, written as ${\rm{supp}}^c(\bm{\sigma}_{\rm{s}}^k)$, is empty.
Thus, combining the dual-valued vector 1-norm in \Cref{pro: dual vector p norm} \cref{dual vector 1-norm} with the dual-valued Ky Fan $k$-norm in \Cref{pro: dual ky fan}, we obtain
\begin{align*}
\Vert\bm{\sigma}^k\Vert_1 &= \Vert\bm{\sigma}_{\rm{s}}^k\Vert_1 + \langle\mathds{1}_k,\bm{\sigma}_{\rm{i}}^k\rangle\epsilon = \Vert\bm{A}_{\rm{s}}\Vert_{(k)} + \left(\sum_{\ell=1}^{k-t}\sigma_{\rm{i}}^\ell+\sum_{\ell=k-t+1}^{k}\sigma_{\rm{i}}^\ell\right)\epsilon\\
&= \Vert\bm{A}_{\rm{s}}\Vert_{(k)}+\left({\rm{tr}}(\bm{W})+ \sum_{\ell=1}^t\sigma_{\theta,\ell}\right)\epsilon = \Vert\bm{A}\Vert_{(k)}.
\end{align*}

\textbf{Case 2:} $1<p<\infty$.
According to the dual-valued vector $p$-norm in \Cref{pro: dual vector p norm} \cref{eq: dual vector p-norm} and the dual-valued Ky Fan $p$-$k$-norm in \Cref{pro: dual ky fan p-k-norm}, it yields
\begin{align*}
&\Vert\bm{\sigma}^k\Vert_p = \Vert\bm{\sigma}^k_{\rm{s}}\Vert_p + \frac{\big\langle(\bm{\sigma}_{\rm{s}}^k)^{p-1},\bm{\sigma}_{\rm{i}}\big\rangle}{\Vert\bm{\sigma}_{\rm{s}}^k\Vert_p^{p-1}}\epsilon\\
=& \Vert\bm{A}_{\rm{s}}\Vert_{(k,p)} + \frac{1}{\Vert\bm{A}_{\rm{s}}\Vert_{(k,p)}^{p-1}}\left(\sum_{\ell=1}^{k-t}(\sigma_{\rm{s}}^\ell)^{p-1}(\sigma_{\rm{i}}^\ell)+\sum_{\ell=k-t+1}^{k}(\sigma_{\rm{s}}^\ell)^{p-1}(\sigma_{\rm{i}}^\ell)\right)\epsilon\\
=& \Vert\bm{A}_{\rm{s}}\Vert_{(k,p)} + \frac{1}{\Vert\bm{A}_{\rm{s}}\Vert_{(k,p)}^{p-1}}\left(\big\langle\bm{\Sigma}_1^{p-1},{\rm{Diag}}(\bm{W})\big\rangle + (\sigma_{\rm{s}}^k)^{p-1}\sum_{\ell=k-t+1}^{k}\sigma_{\rm{i}}^\ell\right)\epsilon = \Vert\bm{A}\Vert_{(k,p)}.
\end{align*}

\textbf{Case 3:} $p=\infty$.
It should be noted that $\Vert\bm{\sigma}_{\rm{s}}^k\Vert_\infty=\Vert\bm{A}_{\rm{s}}\Vert_2 = \Tilde{\sigma}_1=\sigma_{\rm{s}}^1=\cdots=\sigma_{\rm{s}}^{r_1}$.
Thus, the set of indices $\mathcal{I}=\{\ell:|\sigma_{\rm{s}}^\ell|=\Vert\bm{\sigma}_{\rm{s}}^k\Vert_\infty\}=\{1,\cdots,\min\{r_1,k\}\}$.
Let $\bm{U}_{r_1}$ and $\bm{V}_{r_1}$ represent the first $r_1$ columns of $\bm{U}_{\rm{s}}$ and $\bm{V}_{\rm{s}}$.
Denote $\bm{R}:={\rm{sym}}(\bm{U}_{r_1}^\top\bm{A}_{\rm{i}}\bm{V}_{r_1})$.
In addition, the CDSVD \cite[Theorem 3.1]{wei2024singular} indicates that 
\begin{align*}
\bm{R}={\rm{diag}}(\sigma_{1,1},\cdots,\sigma_{1,r_1}) = {\rm{diag}}(\sigma_{\rm{i}}^{1},\cdots,\sigma_{\rm{i}}^{r_1})={\rm{Diag}}(\bm{U}_{r_1}^\top\bm{A}_{\rm{i}}\bm{V}_{r_1}),
\end{align*}
and $\sigma_{\rm{i}}^{1}\geq\cdots\geq\sigma_{\rm{i}}^{r_1}$.
Then, the dual-valued vector $\infty$-norm in \Cref{pro: dual vector p norm} \cref{dual vector infty-norm} and the dual-valued spectral norm in $\Cref{pro: dual matrix spectral norm}$ derives that
\begin{align*}
\Vert\bm{\sigma}^k\Vert_\infty = \Vert\bm{\sigma}_{\rm{s}}^k\Vert_\infty + \max\limits_{\ell\in\mathcal{I}}\sigma_{\rm{i}}^\ell\epsilon = \Vert\bm{A}_{\rm{s}}\Vert_2 + \lambda_{\max}(\bm{R})\epsilon = \Vert\bm{A}\Vert_2.
\end{align*}

Hence, we conclude that $\Vert\bm{A}\Vert_{(k,p)}=\Vert\bm{\sigma}^k\Vert_p$ for $1\leq k\leq r$ and $1\leq p\leq \infty$.
\end{proof}

\subsection{Dual-Valued Operator Norms}

In the previous section, we outline a class of unitary invariant dual-valued matrix norms. The following section presents another class, referred to as the dual-valued operator norms. Specifically, the dual-valued operator 1-norm and the dual-valued operator $\infty$-norm are proposed.

\begin{proposition}\label{pro: dual operator 1-norm}
Given a dual matrix $\bm{A}=\bm{A}_{\rm{s}}+\bm{A}_{\rm{i}}\epsilon=[\bm{a}_1,\cdots,\bm{a}_n]\in\mathbb{DR}^{m\times n}$.
Then the dual-valued operator 1-norm $\Vert\cdot\Vert_1:\mathbb{DR}^{m\times n}\rightarrow\mathbb{DR}$ of $\bm{A}$ extending from the matrix operator 1-norm $\Vert\cdot\Vert_1:\mathbb{R}^{m\times n}\rightarrow\mathbb{R}$ is expressed as
\begin{align}
\Vert\bm{A}\Vert_1=\left\{\begin{array}{ll}
\max_{1\leq t\leq n}\Vert\bm{a}_t\Vert_1, & \bm{A}_{\rm{s}}\neq\bm{O},\\
\Vert\bm{A}_{\rm{i}}\Vert_1, & \bm{A}_{\rm{s}}=\bm{O}.
\end{array}\right.\label{eq: dual operator 1-norm}
\end{align}
\end{proposition}

\begin{proposition}\label{pro: dual operator infty norm}
Given a dual matrix $\bm{A}\!=\!\bm{A}_{\rm{s}}+\bm{A}_{\rm{i}}\epsilon\!=\![\bm{b}_1^\top,\cdots,\bm{b}_m^\top]^\top\!\in\!\mathbb{DR}^{m\times n}$.
Then the dual-valued operator $\infty$-norm $\Vert\cdot\Vert_\infty:\mathbb{DR}^{m\times n}\rightarrow\mathbb{DR}$ of $\bm{A}$ extending from the matrix operator $\infty$-norm $\Vert\cdot\Vert_\infty:\mathbb{R}^{m\times n}\rightarrow\mathbb{R}$ is expressed as
\begin{align}
\Vert\bm{A}\Vert_\infty=\left\{\begin{array}{ll}
\max_{1\leq t\leq m}\Vert\bm{b}_t\Vert_1, & \bm{A}_{\rm{s}}\neq\bm{O},\\
\Vert\bm{A}_{\rm{i}}\Vert_\infty, & \bm{A}_{\rm{s}}=\bm{O}.
\end{array}\right.
\end{align}
\end{proposition}

To demonstrate these two propositions, it is necessary to discuss the subdifferential of the matrix operator 1-norm and that of the matrix operator $\infty$-norm. 

\begin{lemma}\label{lem: sub of operator 1-norm}
Let $\bm{X}=[\bm{x}_1,\cdots,\bm{x}_n]\in\mathbb{R}^{m\times n}$.
Then $\bm{G}\in\partial\Vert\bm{X}\Vert_1$ if and only if $\bm{G}=[\bm{g}_1,\cdots,\bm{g}_n]\in\mathbb{R}^{m\times n}$ satisfies
\begin{align}\left\{
\begin{array}{ll}
\bm{g}_t=\bm{0}, & t\notin\mathcal{I}, \\
\bm{g}_t=\beta_t\big({\rm{sign}}(\bm{x}_t)+\bm{w}_t\big), & t\in\mathcal{I}, 
\end{array}\right.
\end{align}
where $\mathcal{I}\!=\!\{t\!:\!\Vert\bm{X}\Vert_1\! =\! \Vert\bm{x}_t\Vert_1\}$, $\beta_t\!\geq\! 0$, $\sum\limits_{t\in\mathcal{I}}\beta_t\!=\!1$, $\Vert\bm{w}_t\Vert_\infty\!\leq\!1$ and ${\rm{supp}}(\bm{w}_t)\!\subseteq\!{\rm{supp}}^c(\bm{x}_t)$.
\end{lemma}

\begin{proof}
Since the matrix operator 1-norm is expressed as $\Vert\bm{X}\Vert_1=\max_{1\leq t\leq n}\Vert\bm{x}_t\Vert_1$, the subdifferential of $\Vert\bm{X}\Vert_1$ is derived directly from the subdifferentials of vector 1-norm and vector $\infty$-norm shown in \Cref{lem: sub of vector p norm} \cref{eq: sub of vector 1,eq: sub of vector infty}. 
\end{proof}

\begin{proof}[Proof of \Cref{pro: dual operator 1-norm}]
Based on \Cref{lem: sub of operator 1-norm}, it derives that
\begin{align*}
\max_{\bm{G}\in\partial\Vert\bm{A}_{\rm{s}}\Vert_1}\langle\bm{G},\bm{A}_{\rm{i}}\rangle &= \max_{\substack{t\in\mathcal{I}, \;\Vert\bm{w}_t\Vert_{\infty}\leq 1\\{\rm{supp}}(\bm{w}_t)\subseteq{\rm{supp}}^c(\bm{a}_{\rm{s}}^t)}}\Big(\big\langle{\rm{sign}}(\bm{a}_{\rm{s}}^t),\bm{a}_{\rm{i}}^t\big\rangle+\big\langle\bm{w}_t,\bm{a}_{\rm{i}}^t\big\rangle\Big)\\
&= \max_{t\in\mathcal{I}}\Big(\big\langle{\rm{sign}}(\bm{a}_{\rm{s}}^t),\bm{a}_{\rm{i}}^t\big\rangle+\sum_{k\in{\rm{supp}}^c(\bm{a}_{\rm{s}}^t)}\big|(\bm{a}_{\rm{i}}^t)_k\big|\Big),
\end{align*}
where $\bm{A}_{\rm{s}}\!=\![\bm{a}_{\rm{s}}^1,\cdots,\bm{a}_{\rm{s}}^n]$, $\bm{A}_{\rm{i}}\!=\![\bm{a}_{\rm{i}}^1,\cdots,\bm{a}_{\rm{i}}^n]$ and $\mathcal{I}\!=\!\{t\!:\!\Vert\bm{A}_{\rm{s}}\Vert_1\!=\!\Vert\bm{a}_{\rm{s}}^t\Vert_1\}$.
Besides, $(\bm{a}_{\rm{i}}^t)_k$ represents the $k$-th element of $\bm{a}_{\rm{i}}^t$.
Thus, combining the dual-valued vector 1-norm in \Cref{pro: dual vector p norm} \cref{dual vector 1-norm} and the total order of dual numbers, we obtain that 
\begin{align*}
\max_{1\leq t\leq n}\Vert\bm{a}^t\Vert_1 &= \max_{1\leq t\leq n}\bigg(\Vert\bm{a}_{\rm{s}}^t\Vert_1 + \Big(\big\langle{\rm{sign}}(\bm{a}_{\rm{s}}^t),\bm{a}_{\rm{i}}^t\big\rangle+\sum_{k\in{\rm{supp}}^c(\bm{a}_{\rm{s}}^t)}\big|(\bm{a}_{\rm{i}}^t)_k\big|\Big)\epsilon\bigg)\\
&= \max_{1\leq t\leq n}\Vert\bm{a}_{\rm{s}}^t\Vert_1 + \max_{t\in\mathcal{I}}\Big(\big\langle{\rm{sign}}(\bm{a}_{\rm{s}}^t),\bm{a}_{\rm{i}}^t\big\rangle+\sum_{k\in{\rm{supp}}^c(\bm{a}_{\rm{s}}^t)}\big|(\bm{a}_{\rm{i}}^t)_k\big|\Big)\epsilon\\
&= \Vert\bm{A}_{\rm{s}}\Vert_1 + \max_{\bm{G}\in\partial\Vert\bm{A}_{\rm{s}}\Vert_1}\langle\bm{G},\bm{A}_{\rm{i}}\rangle = \Vert\bm{A}\Vert_1.
\end{align*}

Hence, the expression \cref{eq: dual operator 1-norm} is established in conjunction with \Cref{the: extension dual matrix norm}.
\end{proof}

Similarly, the subdifferential and the dual continuation of the matrix operator $\infty$-norm can be developed, whose detailed proofs are omitted.
In fact, the expressions of the dual-valued operator 1-norm and the dual-valued operator $\infty$-norm are consistent with the results proposed by Miao and Huang in \cite{miao2023norms}.
Moreover, \Cref{the: dual matrix operator alpha beta norm} facilitates the development of other useful dual-valued operator norms.

\subsection{Other Dual-Valued Functions}\label{sec: other dual valued functions}

In addition to the norms, the trace and determinant of a dual matrix are also of significant importance in dual algebra. 

\begin{proposition}\label{pro: trace}
Given $\bm{A}=\bm{A}_{\rm{s}}+\bm{A}_{\rm{i}}\epsilon\in\mathbb{DR}^{n\times n}$.
Then the dual-valued trace ${\rm{tr}}:\mathbb{DR}^{n\times n}\rightarrow \mathbb{DR}$ of $\bm{A}$ extending from the trace ${\rm{tr}}:\mathbb{R}^{n\times n}\rightarrow \mathbb{R}$ is captured by
\begin{align}
{\rm{tr}}(\bm{A}) = {\rm{tr}}(\bm{A}_{\rm{s}}) + {\rm{tr}}(\bm{A}_{\rm{i}})\epsilon.
\end{align}
\end{proposition}

\begin{proposition}\label{pro: det}
Given $\bm{A}=\bm{A}_{\rm{s}}+\bm{A}_{\rm{i}}\epsilon\in\mathbb{DR}^{n\times n}$.
Then the dual-valued determinant ${\rm{det}}\!:\!\mathbb{DR}^{n\times n}\rightarrow \mathbb{DR}$ of $\bm{A}$ extending from ${\rm{det}}\!:\!\mathbb{R}^{n\times n}\rightarrow \mathbb{R}$ is captured by
\begin{align}
{\rm{det}}(\bm{A}) = {\rm{det}}(\bm{A}_{\rm{s}}) + \big\langle({\rm{adj}}\,\bm{A}_{\rm{s}})^\top,\bm{A}_{\rm{i}}\big\rangle\epsilon,
\end{align}
where ${\rm{adj}}\,\bm{A}_{\rm{s}}$ represents the adjoint matrix of $\bm{A}_{\rm{s}}$. 
\end{proposition}

Taking advantage of the fact that $\nabla_{\bm{X}}{\rm{tr}}(\bm{X}) = \bm{I}_{n}$ and $\nabla_{\bm{X}}{\rm{det}}(\bm{X}) = ({\rm{adj}}\,\bm{X})^\top$ for any matrix $\bm{X}\in\mathbb{R}^{n\times n}$, as detailed in \cite{horn2012matrix}, we derive the above two propositions based on \Cref{lem_dire_gra=max} and \Cref{def: dual extension}.
In addition, the fundamental properties of the dual-valued determinant are elucidated as follows.

\begin{proposition}
Given $\bm{A},\bm{B}\in\mathbb{DR}^{n\times n}$ and $c\in\mathbb{DR}$. 
Then $({\rm{i}})$ ${\rm{det}}(\bm{AB}) = {\rm{det}}(\bm{A}){\rm{det}}(\bm{B})$; 
$({\rm{ii}})$ ${\rm{det}}(\bm{A}^\top) = {\rm{det}}(\bm{A})$;
and $({\rm{iii}})$ ${\rm{det}}(c\bm{A}) = c^n{\rm{det}}(\bm{A})$.
\end{proposition}
\begin{proof}
We have known from \cite{horn2012matrix} that $({\rm{adj}}\,\bm{X})\cdot\bm{X} = \bm{X}\cdot({\rm{adj}}\,\bm{X})={\rm{det}}(\bm{X})\cdot\bm{I}_n$, ${\rm{adj}}(\bm{X}^\top) = ({\rm{adj}}\,\bm{X})^\top$, ${\rm{adj}}(c\bm{X}) = c^{n-1}({\rm{adj}}\,\bm{X})$, and ${\rm{adj}}\,(\bm{XY}) = ({\rm{adj}}\,\bm{Y})({\rm{adj}}\,\bm{X})$ for $\bm{X},\bm{Y}\in\mathbb{R}^{n\times n}$.
Thus, the three properties are demonstrated directly based on the dual-valued power function in formula \cref{eq: (a+b eps)p} and \Cref{pro: det}.
\end{proof}

As for a symmetric positive semidefinite matrix $\bm{X}\in\mathbb{R}^{n\times n}$, its trace equals its nuclear norm. 
This results from the fact that the eigenvalues of $\bm{X}$ are also its singular values and that the sum of the eigenvalues of a matrix equals its trace.
In addition, a symmetric dual matrix $\bm{A}\in\mathbb{DR}^{n\times n}$ is called positive semidefinite if for any $\bm{x}\in\mathbb{DR}^n$, $\bm{x}^\top\bm{A}\bm{x}\geq0$, as defined by Qi \textit{et al.} in \cite{qi2022low}. 
Furthermore, the following proposition claims that this conclusion is also satisfied in dual algebra.

\begin{proposition}\label{pro: trace=nuclear norm}
The trace of a symmetric positive semidefinite dual matrix equals its dual-valued nuclear norm.
\end{proposition}
\begin{proof}
Suppose that $\bm{A}=\bm{A}_{\rm{s}}+\bm{A}_{\rm{i}}\epsilon\in\mathbb{DR}^{n\times n}$ is symmetric positive semidefinite.
For the case that $\bm{A}_{\rm{s}}=\bm{O}$, it is obvious that $\bm{A}_{\rm{i}}\succcurlyeq0$ (symmetric positive semidefinite).
Thus, $\Vert\bm{A}\Vert_*=\Vert\bm{A}_{\rm{i}}\Vert_*\epsilon={\rm{tr}}(\bm{A}_{\rm{i}})\epsilon={\rm{tr}}(\bm{A})$.
Consider the other case that $\bm{A}_{\rm{s}}\neq\bm{O}$.
Then $\bm{A}_{\rm{s}}\succcurlyeq0$.
Hence, there exists an orthogonal matrix $\bm{W}=[\bm{W}_1,\bm{W}_2]\in\mathbb{R}^{n\times n}$ with $\bm{W}_1=\bm{W}(:,1:r)$, $\bm{W}_2=\bm{W}(:,r+1:n)$ and a diagonal matrix $\bm{\Lambda}={\rm{diag}}(\lambda_1,\cdots,\lambda_r,0,\cdots,0)$ satisfying $\lambda_1\geq\cdots\geq\lambda_r>0$ such that $\bm{A}_{\rm{s}}=\bm{W\Lambda W}^\top$, where $r={\rm{Rank}}(\bm{A}_{\rm{s}})$.

Moreover, as demonstrated in \cite[Theorem 4.4]{qi2022low}, the eigenvalues of the matrix $\bm{C}:=\bm{W}_2^\top\bm{A}_{\rm{i}}\bm{W}_2$, namely $s_1,\cdots,s_{n-r}$, are the infinitesimal parts of the dual matrix $\bm{A}$'s eigenvalues whose standard parts equal 0. 
In other words, $s_1\epsilon,\cdots,s_{n-r}\epsilon$ are eigenvalues of $\bm{A}$.
Besides, it yields from \cite[Theorem 4.5]{qi2022low} that a symmetric dual matrix is positive semidefinite if and only if all of its eigenvalues are non-negative.
Hence, $s_i\geq0$ for $i=1,\cdots,n-r$. 
This implies that $\bm{C}\succcurlyeq0$ and ${\rm{tr}}(\bm{C})=\Vert\bm{C}\Vert_*$.

Therefore, according to the dual-valued nuclear norm in \Cref{cor: dual nuclear norm} and the dual-valued trace in \Cref{pro: trace}, we have
\begin{align*}
{\rm{tr}}(\bm{A})&={\rm{tr}}(\bm{A}_{\rm{s}})+{\rm{tr}}(\bm{A}_{\rm{i}})\epsilon=\Vert\bm{A}_{\rm{s}}\Vert_*+{\rm{tr}}(\bm{W}^\top\bm{A}_{\rm{i}}\bm{W})\epsilon \\
&=\Vert\bm{A}_{\rm{s}}\Vert_*+\big(\langle\bm{W}_1\bm{W}_1^\top,\bm{A}_{\rm{i}}\rangle+\Vert\bm{W}_2^\top\bm{A}_{\rm{i}}\bm{W}_2\Vert_*\big)\epsilon=\Vert\bm{A}\Vert_*.
\end{align*}
\end{proof}

\section{Causal Emergence}\label{sec5: causal emergence}
Causal emergence (CE), provided by Hoel \textit{et al.} \cite{hoel2017map,hoel2013quantifying}, is defined intuitively as the phenomenon observed in dynamical systems where stronger causal effects are achieved on macro-states by the process of coarse-graining micro-states.
In other words, a macro-scale description may offer more comprehensive insight than a fully detailed micro-scale description of the system.
One of the most essential methods to measure the extent of CE, especially for complex system established by Markov chains, is based on the effective information (EI).
For a Markov chain with a discrete state space and the transitional probability matrix (TPM) $\bm{P}\in\mathbb{R}^{n\times n}$ satisfying $\mathds{1}^\top\bm{P}=\mathds{1}^\top$ and $\bm{P}\geq\bm{O}$, the function ${\rm{EI}}:\mathbb{R}^{n\times n}\rightarrow\mathbb{R}$ is defined
as
\begin{align}
    {\rm{EI}}(\bm{P})=\frac{1}{n}\left\langle\bm{P},\log_2\bm{P}-\log_2\Big(\frac{1}{n}\bm{P}\mathds{11}^\top\Big)\right\rangle,\nonumber
\end{align}
where $\log_2(\bm{P})$ satisfies $[\log_2(\bm{P})]_{ij}=\log_2(p_{ij})$ if $p_{ij}>0$ and $[\log_2(\bm{P})]_{ij}=0$ otherwise.

In order to exhibit CE, some previous work \cite{yang2024finding,yuan2024emergence} obtained coarse-graining methods by maximizing EI, but remained challenges of computational complexity and the uniqueness of the solution.
Recently, Zhang \textit{et al.} \cite{zhang2024dynamical} have measured the proximity of a Markov chain being dynamically reversible by utilizing the Schatten $p$-norm of its corresponding TPM.
Although this quantification is independent on any coarse-graining methods, it requires a significant gap in the singular values of the TPM. 
However, it is not possible to ascertain the classification number when the singular values of the TPM decline in a gradual manner without any sudden changes.

In this section, we propose the concept of the dual transitional probability matrix (DTPM) and its dual-valued effective information (${\rm{EI_d}}$) by extending the TPM and its EI to the dual algebra.
Moreover, we investigate the relationship between the ${\rm{EI_d}}$, the dual-valued Schatten $p$-norm, and the dynamical reversibility of the DTPM.

\subsection{Definitions}
The analysis commences with the definitions of the dual transitional probability matrix and its dual-valued EI. 

\begin{definition}\label{def: DTPM}
An $n$-by-$n$ dual transitional probability matrix ${\rm{(DTPM)}}$ $\bm{P}$ is a non-negative dual matrix with $\mathds{1}_n^\top\bm{P}=\mathds{1}_n^\top$, where $\mathds{1}_n$ is an all-1 vector of size $n$.
\end{definition}

Let $\bm{P}=\bm{P}_{\rm{s}}+\bm{P}_{\rm{i}}\epsilon$ be a DTPM.
Then $\bm{P}_{\rm{s}}\geq\bm{O}$, $\mathds{1}^\top\bm{P}_{\rm{s}}=\mathds{1}^\top$ and $\mathds{1}^\top\bm{P}_{\rm{i}}=\bm{0}^\top$, which implies that $\bm{P}_{\rm{s}}$ is a TPM. 
Besides, the $(j,k)$-element of $\bm{P}_{\rm{i}}$ is non-negative for any $(j,k)$ satisfying the condition that the $(j,k)$-element of $\bm{P}_{\rm{s}}$ is equal to zero.

Inspired by the equivalence between the dual-valued vector $p$-norm and the outcome yielded by the element-wise method, the dual continuation of TPM's EI has been developed similarly from the perspective of elements.

\begin{definition}\label{def: dual EI}
Given an $n$-by-$n$ ${\rm{DTPM}}$ $\bm{P}=\bm{P}_{\rm{s}}+\bm{P}_{\rm{i}}\epsilon$.
Then the ${\rm{DTPM}}$'s dual-valued effective information ${\rm{EI_d}}:\mathbb{DR}^{n\times n}\rightarrow\mathbb{DR}$ extending from ${\rm{EI}}$ is defined as
\begin{align}
{\rm{EI_d}}(\bm{P}) = \frac{1}{n}\left\langle \bm{P},\log_2\bm{P}-\log_2\Big(\frac{1}{n}\bm{P}\mathds{1}\mathds{1}^\top\Big)\right\rangle,\label{eq: DTPM's EI}
\end{align}
where $\log_2(\cdot)$ is an element-wise operator satisfying 
\begin{align}
\big[\log_2\bm{P}\big]_{jk} = \left\{\begin{array}{ll}
\log_2[\bm{P}_{\rm{s}}]_{jk}+\frac{[\bm{P}_{\rm{i}}]_{jk}}{[\bm{P}_{\rm{s}}]_{jk}\ln2}\epsilon, & [\bm{P}_{\rm{s}}]_{jk}>0, \\
\log_2[\bm{P}_{\rm{i}}]_{jk}\epsilon,  & [\bm{P}_{\rm{s}}]_{jk}=0, [\bm{P}_{\rm{i}}]_{jk}>0,\\
0, & [\bm{P}_{\rm{s}}]_{jk}=0, [\bm{P}_{\rm{i}}]_{jk}=0.
\end{array}\right.
\end{align}
\end{definition}

To facilitate the subsequent proof, we simplify the DTPM's ${\rm{EI_d}}$ defined in \cref{eq: DTPM's EI}.
Let $\bm{P}=\bm{P}_{\rm{s}}+\bm{P}_{\rm{i}}\epsilon$ be a DTPM.
We obtain
\begin{align*}
\langle\bm{P},\log_2\bm{P}\rangle &= \sum_{j,k}\big[\bm{P}_{\rm{s}}+\bm{P}_{\rm{i}}\epsilon\big]_{jk}\big[\log_2(\bm{P}_{\rm{s}}+\bm{P}_{\rm{i}}\epsilon)\big]_{jk}\\
&= \sum_{[\bm{P}_{\rm{s}}]_{jk}>0}[\bm{P}_{\rm{s}}]_{jk}\log_2[\bm{P}_{\rm{s}}]_{jk} + \sum_{[\bm{P}_{\rm{s}}]_{jk}>0}\Big(\frac{1}{\ln2}[\bm{P}_{\rm{i}}]_{jk}+[\bm{P}_{\rm{i}}]_{jk}\log_2[\bm{P}_{\rm{s}}]_{jk}\Big)\epsilon,
\end{align*}
where $[\bm{P}_{\rm{s}}]_{jk}$ represents the $(j,k)$-element of $\bm{P}_{\rm{s}}$.
Besides, 
\begin{align*}
\left\langle\bm{P},\log_2\Big(\frac{1}{n}\bm{P}\mathds{1}\mathds{1}^\top\Big)\right\rangle=\sum_{j,k}\big[\bm{P}_{\rm{s}}+\bm{P}_{\rm{i}}\epsilon\big]_{jk}\left[\log_2\Big(\frac{1}{n}\bm{P}_{\rm{s}}\mathds{1}\mathds{1}^\top+\frac{1}{n}\bm{P}_{\rm{i}}\mathds{1}\mathds{1}^\top\epsilon\Big)\right]_{jk}\\
=\sum_{[\bm{P}_{\rm{s}}]_{jk}>0}[\bm{P}_{\rm{s}}]_{jk}\log_2\Big[\frac{1}{n}\bm{P}_{\rm{s}}\mathds{1}\mathds{1}^\top\Big]_{jk} + \left(\sum_{[\bm{P}_{\rm{s}}]_{jk}>0}[\bm{P}_{\rm{s}}]_{jk}\frac{1}{\ln2}\frac{[\bm{P}_{\rm{i}}\mathds{1}\mathds{1}^\top]_{jk}}{[\bm{P}_{\rm{s}}\mathds{1}\mathds{1}^\top]_{jk}}\right.\\
\left.+\sum_{[\bm{P}_{\rm{s}}\mathds{1}\mathds{1}^\top]_{jk}>0}[\bm{P}_{\rm{i}}]_{jk}\log_2\Big[\frac{1}{n}\bm{P}_{\rm{s}}\mathds{1}\mathds{1}^\top\Big]_{jk}\right)\epsilon.
\end{align*}
Then the ${\rm{EI_d}}$ of the DTPM $\bm{P}$ shown in \cref{eq: DTPM's EI} is further expressed as
\begin{align*}
{\rm{EI_d}}(\bm{P})
:= {\rm{EI_{ds}}}(\bm{P})+{\rm{EI_{di}}}(\bm{P})\epsilon= \frac{1}{n}\sum_{[\bm{P}_{\rm{s}}]_{jk}>0}[\bm{P}_{\rm{s}}]_{jk}\log_2\frac{[\bm{P}_{\rm{s}}]_{jk}}{\Big[\frac{1}{n}\bm{P}_{\rm{s}}\mathds{1}\mathds{1}^\top\Big]_{jk}}\\
+\frac{1}{n}\left(\sum_{[\bm{P}_{\rm{s}}]_{jk}>0}\Big(\frac{[\bm{P}_{\rm{i}}]_{jk}}{\ln2}+[\bm{P}_{\rm{i}}]_{jk}\log_2[\bm{P}_{\rm{s}}]_{jk}-\frac{[\bm{P}_{\rm{s}}]_{jk}}{\ln2}\frac{[\bm{P}_{\rm{i}}\mathds{1}\mathds{1}^\top]_{jk}}{[\bm{P}_{\rm{s}}\mathds{1}\mathds{1}^\top]_{jk}}\Big)\right.\\
\left.-\sum_{[\bm{P}_{\rm{s}}\mathds{1}\mathds{1}^\top]_{jk}>0}[\bm{P}_{\rm{i}}]_{jk}\log_2\Big[\frac{1}{n}\bm{P}_{\rm{s}}\mathds{1}\mathds{1}^\top\Big]_{jk}\right)\epsilon.
\end{align*}

\subsection{Properties}

In a recent study, Zhang et al. \cite{zhang2024dynamical} demonstrated the relationship between the EI, the Schatten $p$-norm, and the dynamical reversibility of a TPM.
In this section, we investigate the similar properties in dual algebra.
For the sake of simplicity, a one-hot vector is referred to as a real vector comprising a single element equal to 1 and all other elements equal to 0.
Besides, a DTPM is called a permutation matrix if its standard part is a permutation matrix and its infinitesimal part is zero.
Moreover, we employ the abbreviation ``iff'' to signify ``if and only if''.

We begin by illustrating the maximum and minimum values of a DTPM's ${\rm{EI_d}}$.

\begin{proposition}\label{pro: dual EI}
Let $\bm{P}=\bm{P}_{\rm{s}}+\bm{P}_{\rm{i}}\epsilon$ be an $n$-by-$n$ ${\rm{DTPM}}$.
Then ${\rm{EI_d}}(\bm{P})$
\begin{enumerate}
\item[${\rm{(i)}}$] reaches its maximum value of $\log_2n$ iff $\bm{P}$ is a permutation matrix;
\item[${\rm{(ii)}}$] reaches its minimum value of $0$ iff 
$\bm{P}=\bm{u}\mathds{1}^\top + \bm{P}_{\rm{i}}\epsilon$ with $\mathds{1}^\top\bm{u}=1$.
\end{enumerate}
\end{proposition}
\begin{proof}
Based on the total order of dual numbers, the maximization (or minimization) of a dual number is equivalent to the maximization (or minimization) of its standard part initially, followed by that of its infinitesimal part.

To maximize (or minimize) ${\rm{EI_d}}$, it is first necessary to maximize (or minimize) ${\rm{EI}}_{\rm{ds}}$, which has been derived in \cite[A.1 Propositions 1 and 2]{zhang2024dynamical}.
Specifically, ${\rm{EI}}_{\rm{ds}}$ reaches its maximum value of $\log_2n$ if and only if $\bm{P}_{\rm{s}}$ is a permutation matrix, and its minimum value of 0 if and only if $\bm{P}_{\rm{s}}$ has identical columns.
Subsequently, the special $\bm{P}_{\rm{s}}$ is employed to maximize (or minimize) ${\rm{EI}}_{\rm{di}}$.

(i) Given that $\max_{\bm{P}_{\rm{s}}}{\rm{EI}}_{\rm{ds}}=\log_2n$ if and only if $\bm{P}_{\rm{s}}$ is a permutation matrix.
Besides, for the permutation matrix $\bm{P}_{\rm{s}}$, there exists a permutation $\tau:[n]\rightarrow[n]$, such that $[\bm{P}_{\rm{s}}]_{j,\tau(j)}=1$ for $j\in[n]$, where $[n]$ represents the set $\{1,2,\cdots,n\}$.
Moreover, by utilizing the properties $\bm{P}_{\rm{s}}\mathds{1}=\mathds{1}$ an $\mathds{1}^\top\bm{P}_{\rm{i}}=\bm{0}^\top$, the value of ${\rm{EI}}_{\rm{di}}$ is
\begin{align*}
{\rm{EI}}_{\rm{di}}&= \frac{1}{n}\sum_{[\bm{P}_{\rm{s}}]_{jk}>0}\Big(\frac{[\bm{P}_{\rm{i}}]_{jk}}{\ln2}+[\bm{P}_{\rm{i}}]_{jk}\log_2[\bm{P}_{\rm{s}}]_{jk}-\frac{[\bm{P}_{\rm{s}}]_{jk}}{\ln2}\frac{[\bm{P}_{\rm{i}}\mathds{1}\mathds{1}^\top]_{jk}}{[\bm{P}_{\rm{s}}\mathds{1}\mathds{1}^\top]_{jk}}\Big)\\
&\quad\quad\quad\quad\quad\quad\quad\quad\quad-\frac{1}{n}\sum_{[\bm{P}_{\rm{s}}\mathds{1}\mathds{1}^\top]_{jk}>0}[\bm{P}_{\rm{i}}]_{jk}\log_2\Big[\frac{1}{n}\bm{P}_{\rm{s}}\mathds{1}\mathds{1}^\top\Big]_{jk}\\
&= \frac{1}{n\ln2}\sum_{(j,\tau(j))}\Big([\bm{P}_{\rm{i}}]_{j,\tau(j)}-[\bm{P}_{\rm{i}}\mathds{1}\mathds{1}^\top]_{j,\tau(j)}\Big)+\frac{\log_2n}{n}\sum_{j,k}[\bm{P}_{\rm{i}}]_{jk}\\
&\overset{({\rm{I}})}{=} \frac{1}{n\ln2}\sum_{(j,\tau(j))}[\bm{P}_{\rm{i}}]_{j,\tau(j)},
\end{align*}
where (I) results from that $\sum_{(j,\tau(j))}[\bm{P}_{\rm{i}}\mathds{1}\mathds{1}^\top]_{j,\tau(j)}=\sum_{j,k}[\bm{P}_{\rm{i}}]_{jk}=\mathds{1}^\top\bm{P}_{\rm{i}}\mathds{1}=0$.

Then the objective is to maximize ${\rm{EI_{di}}}$, with a corresponding optimization problem presented below.
\begin{align*}
\begin{array}{ll}
\max\limits_{\bm{P}_{\rm{i}}\in\mathbb{R}^{n\times n}} & {\rm{EI}}_{\rm{di}}=\frac{1}{n\ln2}\sum_{(j,\tau(j))}[\bm{P}_{\rm{i}}]_{j,\tau(j)} \\
\quad s.t. & \mathds{1}^\top\bm{P}_{\rm{i}}=\bm{0}^\top, \\ &[\bm{P}_{\rm{i}}]_{jk}\geq0, \forall(j,k)\in\Theta,
\end{array}
\end{align*}
where $\Theta=\{(j,k)\in[n]\times[n]:[\bm{P}_{\rm{s}}]_{jk}=0\}$.
The combination of the equality and non-negative constraints yields the conclusion that the elements $[\bm{P}_{\rm{i}}]_{j,\tau(j)}$ must be non-positive. 
Therefore, the objective function ${\rm{EI}}_{\rm{di}}$ is always non-positive and reaches its maximum value of 0 if and only if $\bm{P}_{\rm{i}}=\bm{O}$.

Therefore, the maximum value of ${\rm{EI}}$ equals $log_2n$, which occurs precisely when the DTPM is a permutation matrix.

(ii) We have known that $\min_{\bm{P}_{\rm{s}}}{\rm{EI_{ds}}} = 0$ if and only if $\bm{P}_{\rm{s}}=\bm{u}\mathds{1}^\top$ with $\bm{u}\geq\bm{0}$ and $\mathds{1}^\top\bm{u}=1$.
Then we derive from $\bm{P}_{\rm{s}}\mathds{1}\mathds{1}^\top=n\bm{P}_{\rm{s}}$ that $[\bm{P}_{\rm{s}}]_{jk}>0\Leftrightarrow[\bm{P}_{\rm{s}}\mathds{1}\mathds{1}^\top]_{jk}>0$.
Thus, the value of ${\rm{EI}}_{\rm{di}}$ satisfies
\begin{align*}
{\rm{EI}}_{\rm{di}}&= \frac{1}{n}\sum_{[\bm{P}_{\rm{s}}]_{jk}>0}\left(\frac{[\bm{P}_{\rm{i}}]_{jk}}{\ln2}+[\bm{P}_{\rm{i}}]_{jk}\log_2[\bm{P}_{\rm{s}}]_{jk}-\frac{[\bm{P}_{\rm{s}}]_{jk}}{\ln2}\frac{[\bm{P}_{\rm{i}}\mathds{1}\mathds{1}^\top]_{jk}}{[\bm{P}_{\rm{s}}\mathds{1}\mathds{1}^\top]_{jk}}\right)\\
&\quad\quad\quad\quad\quad\quad\quad\quad\quad-\frac{1}{n}\sum_{[\bm{P}_{\rm{s}}\mathds{1}\mathds{1}^\top]_{jk}>0}[\bm{P}_{\rm{i}}]_{jk}\log_2\Big[\frac{1}{n}\bm{P}_{\rm{s}}\mathds{1}\mathds{1}^\top\Big]_{jk}\\
&= \frac{1}{n\ln2}\sum_{[\bm{P}_{\rm{s}}]_{jk}>0}\Big([\bm{P}_{\rm{i}}]_{jk}-\frac{1}{n}[\bm{P}_{\rm{i}}\mathds{1}\mathds{1}^\top]_{jk}\Big)\\
&= \frac{1}{n\ln2}\sum_{[\bm{P}_{\rm{s}}]_{j\cdot}>0}\Big([\bm{P}_{\rm{i}}\mathds{1}]_{j}-n\cdot\frac{1}{n}[\bm{P}_{\rm{i}}\mathds{1}]_{j}\Big)=0,
\end{align*}
regardless of what $\bm{P}_{\rm{i}}$ takes.
Hence, ${\rm{EI_d}}(\bm{P})$ achieves its minimum value of 0 for any $\bm{P}_{\rm{i}}$, as long as $\bm{P}_{\rm{s}}$ has identical column vectors. 

Consequently, we complete the proof.
\end{proof}

Then the dynamical reversibility of a DTPM is defined and the corresponding characteristic is elucidated.

\begin{definition}\label{def: dynamically reversible}
A ${\rm{DTPM}}$ $\bm{P}$ is called dynamically reversible iff $\bm{P}$ is invertible and $\bm{P}^{-1}$ is also an effective ${\rm{DTPM}}$.
\end{definition}

\begin{theorem}\label{the: dynamically reversible}
A ${\rm{DTPM}}$ is dynamically reversible iff it is a permutation matrix.
\end{theorem}

\begin{proof}
Let $\bm{P}=\bm{P}_{\rm{s}}+\bm{P}_{\rm{i}}\epsilon\in\mathbb{DR}^{n\times n}$ be a DTPM.
Then $\bm{P}_{\rm{s}}$ is a TPM, $\mathds{1}^\top\bm{P}_{\rm{i}}=\bm{0}$ and $\bm{P}\geq\bm{O}$.

On the one hand, suppose that $\bm{P}$ is dynamically reversible.
Definition \ref{def: dynamically reversible} implies that $\bm{P}$ is invertible and $\bm{P}^{-1} = \bm{P}_{\rm{s}}^{-1}-\bm{P}_{\rm{s}}^{-1}\bm{P}_{\rm{i}}\bm{P}_{\rm{s}}^{-1}\epsilon$ is also a DTPM.
Thus, $\bm{P}_{\rm{s}}^{-1}$ is a TPM and $\bm{P}^{-1}\geq\bm{O}$, as detailed in Definition \ref{def: DTPM}.
Hence, we derive from \cite{ding2013when} that $\bm{P}_{\rm{s}}$ is a permutation matrix.
It follows from this that there exists a permutation $\tau:[n]\rightarrow[n]$ such that $[\bm{P}_{\rm{s}}]_{j,\tau(j)=1}$ for $j\in[n]$ and all other entries of $\bm{P}_{\rm{s}}$ are equal to zero.
Denote $\Delta := \{(j,\tau(j)):j\in[n]\}$.
Then $\bm{P}\geq\bm{O}$ implies that $[\bm{P}_{\rm{i}}]_{jk}\geq0$ for any $(j,k)\in\Theta$, where $\Theta:=\{(j,k)\in[n]\times[n]:[\bm{P}_{\rm{s}}]_{jk}=0\} =[n]\times[n]\backslash\Delta$.
Besides, $\bm{P}_{\rm{s}}^{-1}=\bm{P}_{\rm{s}}^\top$, since $\bm{P}_{\rm{s}}$ is a permutation matrix.
Subsequently, only the $(\tau(j),j)$-entries $(j\in[n])$ of $\bm{P}_{\rm{s}}^\top$ are equal to one, while all other entries are equal to zero. 
It is worth mentioning that
\begin{align*}
\big[\bm{P}_{\rm{s}}^{-1}\bm{P}_{\rm{i}}\bm{P}_{\rm{s}}^{-1}\big]_{jk} &= \big[\bm{P}_{\rm{s}}^\top\bm{P}_{\rm{i}}\bm{P}_{\rm{s}}^\top\big]_{jk} = \sum_{t=1}^n\sum_{\ell=1}^n\big[\bm{P}_{\rm{s}}^\top\big]_{jt}\big[\bm{P}_{\rm{i}}\big]_{t\ell}\big[\bm{P}_{\rm{s}}^\top\big]_{\ell k}\\
&= \big[\bm{P}_{\rm{s}}^\top\big]_{j,\tau^{-1}(j)}\big[\bm{P}_{\rm{i}}\big]_{\tau^{-1}(j),\tau(k)}\big[\bm{P}_{\rm{s}}^\top\big]_{\tau(k),k}=\big[\bm{P}_{\rm{i}}\big]_{\tau^{-1}(j),\tau(k)}.
\end{align*}
Thus $[\bm{P}_{\rm{s}}^{-1}\bm{P}_{\rm{i}}\bm{P}_{\rm{s}}^{-1}]_{\tau(j),j}=[\bm{P}_{\rm{i}}]_{\tau^{-1}(\tau(j)),\tau(j)} = [\bm{P}_{\rm{i}}]_{j,\tau(j)}$.
Additionally, we infer from $\bm{P}^{-1}\geq\bm{O}$ that except for the $(\tau(j),j)$-entries of $-\bm{P}_{\rm{s}}^{-1}\bm{P}_{\rm{i}}\bm{P}_{\rm{s}}^{-1}$, all other entries are non-negative.
It can then be asserted that $-[\bm{P}_{\rm{i}}]_{jk}\geq0$ for $(j,k)\in[n]\times[n]\backslash\Delta$.
Combining the condition $[\bm{P}_{\rm{i}}]_{jk}\geq0$ for $(j,k)\in\Theta$, we obtain $[\bm{P}_{\rm{i}}]_{jk}=0$ for $(j,k)\in[n]\times[n]\backslash\Delta$.
Besides, $\mathds{1}^\top\bm{P}_{\rm{i}} = \big[[\bm{P}_{\rm{i}}]_{\tau^{-1}(1),1},\cdots,[\bm{P}_{\rm{i}}]_{\tau^{-1}(n),n}\big] = \bm{0}^\top$ implies that $\bm{P}_{\rm{i}}=\bm{O}$.
Consequently, $\bm{P}$ is a permutation matrix.

On the other hand, suppose that $\bm{P}$ is a permutation matrix.
Then $\bm{P}$ is invertible and $\bm{P}^{-1} = \bm{P}^\top$ is also a permutation matrix.
It is obvious that $\bm{P}^{-1}$ is non-negative and the sum of each column is equal to one.
That is, $\bm{P}^{-1}$ is a DTPM.
Hence, $\bm{P}$ is dynamically reversible.

Here, we complete the proof.
\end{proof}

The following part reports the maximum and minimum values of the dual-valued Schatten $p$-norm of a DTPM, including when these values are attained.

\begin{proposition}\label{pro: dual Schatten p norm and DTPM}
Let $\bm{P}=\bm{P}_{\rm{s}}+\bm{P}_{\rm{i}}\epsilon$ be an $n$-by-$n$ ${\rm{DTPM}}$. 
Then $\Vert\bm{P}\Vert_{S_p}$ satisfies
\begin{enumerate}
\item[${\rm{(i)}}$] 
$\max_{\bm{P}}\Vert\bm{P}\Vert_{S_p}=n^{\frac{1}{p}}$ for $1\leq p<2$ iff $\bm{P}$ is a permutation matrix;
\item[${\rm{(ii)}}$]
$\max_{\bm{P}}\Vert\bm{P}\Vert_{S_2}=n^{\frac{1}{2}}$ iff each column of $\bm{P}$ is a one-hot vector;
\item[${\rm{(iii)}}$]
$\max_{\bm{P}}\Vert\bm{P}\Vert_{S_p}=n^{\frac{1}{2}}$ for $2<p\leq\infty$ iff $\bm{P}$ consists of identical one-hot vectors;
\item[${\rm{(iv)}}$] 
$\min_{\bm{P}}\Vert\bm{P}\Vert_{S_1}=1$ iff $\bm{P}=\frac{1}{n}\mathds{1}\mathds{1}^\top+\bm{w}\mathds{1}^\top\epsilon$ for any $\bm{w}\in\mathbb{R}^n$ with $\mathds{1}^\top\bm{w}=0$;
\item[${\rm{(v)}}$]
$\min_{\bm{P}}\Vert\bm{P}\Vert_{S_p}=1$ for $1\!<\! p\!\leq\!\infty$ iff $\bm{P}\!=\!\frac{1}{n}\mathds{1}\mathds{1}^\top+\bm{P}_{\rm{i}}\epsilon$ for any $\bm{P}_{\rm{i}}$ with $\mathds{1}^\top\bm{P}_{\rm{i}}=\bm{0}^\top$.
\end{enumerate}
\end{proposition}
To prove this, an essential lemma is introduced.

\begin{lemma}\label{lem:properties of norms of matrices and vectors}
Given a vector $\bm{x}\in\mathbb{R}^n$. 
Then the vector $p$-norms satisfy
\begin{align}
\Vert\bm{x}\Vert_{p_2}\leq\Vert\bm{x}\Vert_{p_1}\leq n^{\big(\frac{1}{p_1}-\frac{1}{p_2}\big)}\Vert\bm{x}\Vert_{p_2},    
\end{align}
for $1\leq p_1<p_2<\infty$.
\end{lemma}
\begin{proof}
Let $\bm{x}=[x_1,\cdots,x_n]^\top$.
It yields from $\frac{p_2}{p_1}>1$ and H$\ddot{o}$lder's inequality that
\begin{align*}
\Vert\bm{x}\Vert_{p_1}^{p_1} &= \sum_{i=1}^n \vert x_i\vert^{p_1} =   \sum_{i=1}^n \vert x_i\vert^{p_1}\cdot 1\leq\left(\sum_{i=1}^n \big(\vert x_i\vert^{p_1}\big)^{\frac{p_2}{p_1}}\right)^{\frac{p_1}{p_2}}\left(\sum_{i=1}^n 1^{\frac{p_2}{p_2-p_1}}\right)^{1-\frac{p_1}{p_2}}\\
&= \Vert\bm{x}\Vert_{p_2}^{p_1}\cdot n^{\frac{p_2-p_1}{p_2}}.
\end{align*}
Hence, $\Vert\bm{x}\Vert_{p_1}\leq n^{\big(\frac{1}{p_1}-\frac{1}{p_2}\big)}\Vert\bm{x}\Vert_{p_2}$ and the equality holds if and only if $\vert x_i\vert$ are the same for any $i$.

On the other hand, we aim to prove that $\Vert\bm{x}\Vert_{p_2}\leq\Vert\bm{x}\Vert_{p_1}$ for $1\leq p_1<p_2$. 
If $\bm{x}=\bm{0}$, it is obvious that $\Vert\bm{x}\Vert_{p_2}=\Vert\bm{x}\Vert_{p_1}$.
Otherwise, let $\bm{y}=\frac{\vert\bm{x}\vert}{\Vert\bm{x}\Vert_{p_2}}:=[y_1,\cdots,y_n]^\top$.
We have $\Vert\bm{x}\Vert_{p_2}=\big(\sum_{i=1}^n|x_i|^{p_2}\big)^{\frac{1}{p_2}}\geq\big(|x_k|^{p_2}\big)^{\frac{1}{p_2}}=|x_k|$ for any $k=1,\cdots,n$.
Then $0\leq y_k\leq1$ and $y_k^{p_1}\geq y_k^{p_2}$ $(k=1,\cdots,n)$ are obtained.
Hence, $\Vert\bm{y}\Vert_{p_1}=\big(\sum_{i=1}^n y_k^{p_1}\big)^{\frac{1}{p_1}}\geq\big(\sum_{i=1}^n y_k^{p_2}\big)^{\frac{1}{p_1}}=\big(\frac{1}{\Vert\bm{x}\Vert_{p_2}^{p_2}}\sum_{i=1}^n|x_k|^{p_2}\big)^{\frac{1}{p_1}}=1$, which implies that $\Vert\bm{x}\Vert_{p_1}\geq\Vert\bm{x}\Vert_{p_2}$.
Therefore, it can be concluded that $\Vert\bm{x}\Vert_{p_1}\geq\Vert\bm{x}\Vert_{p_2}$ and the equality holds if and only if $\bm{x}$ has at least $n-1$ elements equal to zero.
\end{proof}

Subsequently, we intend to verify \Cref{pro: dual Schatten p norm and DTPM}.

\begin{proof}[Proof of \Cref{pro: dual Schatten p norm and DTPM}]

Let $\bm{P}_{\rm{s}} = \bm{U\Sigma V}^\top$ be a full SVD with $r$ non-zero singular values.
Moreover, orthogonal matrices are partitioned so that $\bm{U} =[\bm{U}_r,\widetilde{\bm{U}}_r]$ and $\bm{V} = [\bm{V}_r,\widetilde{\bm{V}}_r]$, with $\bm{U}_r$ and $\bm{V}_r$ both having $r$ columns.
Denote $\bm{P}_{\rm{s}} = [\bm{p}_1,\cdots,\bm{p}_n]$, $\bm{\sigma} = [\sigma_1,\cdots,\sigma_n]^\top$ and $\bm{\Sigma}_r = {\rm{diag}}(\sigma_1,\cdots,\sigma_r)$, where $\sigma_1\geq\cdots\geq\sigma_n\geq0$ are singular values of $\bm{P}_{\rm{s}}$.

Since $\bm{P}_{\rm{s}}$ is a TPM, we obtain the fact that $\bm{p}_i\geq\bm{0}$, $\bm{p}_i\neq\bm{0}$ and $\Vert\bm{p}_i\Vert_1 = \mathds{1}^\top\bm{p}_i=1$.
It then yields that 
\begin{align}
\Vert\bm{P}_{\rm{s}}\Vert_F &= \Vert\bm{U\Sigma V}^\top\Vert_F = \Vert\bm{\Sigma}\Vert_F = \Vert\bm{\sigma}\Vert_2\nonumber\\
&= \left(\sum_{i=1}^n\Vert\bm{p}_i\Vert_2^2\right)^{\frac{1}{2}}\leq\left(\sum_{i=1}^n\Vert\bm{p}_i\Vert_1^2\right)^{\frac{1}{2}}=\sqrt{n}.\label{eq:||Ps||_F^2<=n}
\end{align}
Note that the equality $\Vert\bm{p}_i\Vert_2=\Vert\bm{p}_i\Vert_1$ holds if and only if $\bm{p}_i$ is a one-hot vector (a real vector comprising a single element equal to one and all other elements equal to zero).
Hence, we derive that $\Vert\bm{\sigma}\Vert_2=\sqrt{n}$ if and only if $\bm{P}_{\rm{s}}$ consists of one-hot vectors.

${\rm{(i)}}$ It is demonstrated from \cite[A.2.2 Theorem 2]{zhang2024dynamical} that $\max\Vert\bm{P}_{\rm{s}}\Vert_{S_p} = n^{\frac{1}{p}}$ for $1\leq p<2$ if and only if $\bm{P}_{\rm{s}}$ is a permutation matrix.
Meanwhile, the SVD of $\bm{P}_{\rm{s}}$ satisfies $\bm{U} = \bm{P}_{\rm{s}}$, $\bm{\Sigma} = \bm{I}_n$ and $\bm{V} = \bm{I}_n$.

For the case that $p=1$, we obtain from the dual-valued nuclear norm in \Cref{cor: dual nuclear norm} that
\begin{align*}
\Vert\bm{P}\Vert_{S_1} = \Vert\bm{P}\Vert_* = \Vert \bm{P}_{\rm{s}} \Vert_* + \Big(\big\langle\bm{U}_r\bm{V}_r^\top,\bm{P}_{\rm{i}}\big\rangle + \big\Vert\widetilde{\bm{U}}_r^\top\bm{P}_{\rm{i}}\widetilde{\bm{V}}_r\big\Vert_*\Big)\epsilon= n+\langle\bm{P}_{\rm{s}},\bm{P}_{\rm{i}}\rangle\epsilon.   
\end{align*}
Note that $\langle\bm{P}_{\rm{s}},\bm{P}_{\rm{i}}\rangle\leq0$ for some permutation matrix $\bm{P}_{\rm{s}}$ and any $\bm{P}_{\rm{i}}$ satisfying $\bm{P}$ is a DTPM, which implies that $\langle\bm{P}_{\rm{s}},\bm{P}_{\rm{i}}\rangle=0$ if and only if $\bm{P}_{\rm{i}}=\bm{O}$.
Hence, the maximum value of $\Vert\bm{P}\Vert_{S_1}$ is $n$ if and only if $\bm{P}$ is a permutation matrix.

For the case that $1<p<2$, the dual-valued Schatten $p$-norm in \Cref{cor: dual schatten norm} describes that 
\begin{align*}
\Vert\bm{P}\Vert_{S_p} = \Vert\bm{P}_{\rm{s}}\Vert_{S_p}+\frac{1}{\Vert\bm{P}_{\rm{s}}\Vert_{S_p}^{p-1}}\langle\bm{U}_r\bm{\Sigma}_r^{p-1}\bm{V}_r^\top,\bm{P}_{\rm{i}}\rangle\epsilon= n^{\frac{1}{p}}+n^{\frac{1-p}{p}}\langle\bm{P}_{\rm{s}},\bm{P}_{\rm{i}}\rangle\epsilon. 
\end{align*}
Similarly, $\langle\bm{P}_{\rm{s}},\bm{P}_{\rm{i}}\rangle$ reaches its maximum value of 0 if and only if $\bm{P}_{\rm{i}}=\bm{O}$.
Thus, the maximum value of $\Vert\bm{P}\Vert_{S_p}$ is $n^{\frac{1}{p}}$ if and only if $\bm{P}$ is a permutation matrix.

Hence, $\max\Vert\bm{P}\Vert_{S_p}=n^{\frac{1}{p}}$ for $1\leq p<2$ if and only if $\bm{P}$ is a permutation matrix.

${\rm{(ii)}}$ The dual-valued Frobenius norm \Cref{cor: dual Frobenius norm} claims that
\begin{align*}
\Vert\bm{P}\Vert_{S_2} = \Vert\bm{P}\Vert_F=\Vert\bm{P}_{\rm{s}}\Vert_F+\frac{\langle\bm{P}_{\rm{s}},\bm{P}_{\rm{i}}\rangle}{\Vert\bm{P}_{\rm{s}}\Vert_F}\epsilon.
\end{align*}
It can be concluded from \cref{eq:||Ps||_F^2<=n} that $\Vert\bm{P}_{\rm{s}}\Vert_F\leq\sqrt{n}$ and the equality holds if and only if each column of $\bm{P}_{\rm{s}}$ is a one-hot vector.
Thus, $\Vert\bm{P}\Vert_{S_2} = \sqrt{n}+\frac{\langle\bm{P}_{\rm{s}},\bm{P}_{\rm{i}}\rangle}{\sqrt{n}}\epsilon$.
We still readily demonstrate that in this case, the maximum value of $\langle\bm{P}_{\rm{s}},\bm{P}_{\rm{i}}\rangle$ is zero when $\bm{P}_{\rm{i}}$ is chosen as a zero matrix.
Hence, $\Vert\bm{P}\Vert_{S_2}$ reaches its maximum value of $\sqrt{n}$ if and only if each column of $\bm{P}$ is a one-hot vector.

${\rm{(iii)}}$ Consider the case that $2<p<\infty$.
Then we derive from \Cref{cor: dual schatten norm} that the dual-valued Schatten $p$-norm of the DTPM $\bm{P}$ is expressed as 
\begin{align*}
\Vert\bm{P}\Vert_{S_p} = \Vert\bm{P}_{\rm{s}}\Vert_{S_p}+\frac{1}{\Vert\bm{P}_{\rm{s}}\Vert_{S_p}^{p-1}}\langle\bm{U}_r\bm{\Sigma}_r^{p-1}\bm{V}_r^\top,\bm{P}_{\rm{i}}\rangle\epsilon,    
\end{align*}
where $\Vert\bm{P}_{\rm{s}}\Vert_{S_p}=\Vert\bm{\sigma}\Vert_p$.
\Cref{lem:properties of norms of matrices and vectors} and the formula \cref{eq:||Ps||_F^2<=n} claim that $\Vert\bm{\sigma}\Vert_p\leq \Vert\bm{\sigma}\Vert_2\leq\sqrt{n}$.
It can be seen that the first equality holds if and only if $\bm{\sigma}$ has a single non-zero element and the second equality holds if and only if $\bm{P}_{\rm{s}}$ is made up of one-hot vectors.
Hence, the rank of $\bm{P}_{\rm{s}}$ is 1, whereby a particular row of $\bm{P}_{\rm{s}}$ must therefore be identical to $\mathds{1}^\top$ and the other rows are all equal to $\bm{0}^\top$.
Under the circumstances, $r=1$, $\bm{\sigma} = [\sqrt{n},0,\cdots,0]^\top$ and $\bm{P}_{\rm{s}} = \bm{U}_r\bm{\Sigma}_r\bm{V}_r^\top=\sqrt{n}\bm{U}_r\bm{V}_r^\top$.
Thus, $\bm{U}_r\bm{\Sigma}_r^{p-1}\bm{V}_r^\top=n^{\frac{p-1}{2}}\bm{U}_r\bm{V}_r^\top= n^{\frac{p-2}{2}}\bm{P}_{\rm{s}}$.
We conclude that $\Vert\bm{P}\Vert_{S_p} = \sqrt{n}+\frac{\langle\bm{P}_{\rm{s}},\bm{P}_{\rm{i}}\rangle}{\sqrt{n}}\epsilon$ and $\langle\bm{P}_{\rm{s}},\bm{P}_{\rm{i}}\rangle$ reaches its maximum value of zero if and only if $\bm{P}_{\rm{i}}=\bm{O}$.
Thus, the maximum value of $\Vert\bm{P}\Vert_{S_p}$ is $\sqrt{n}$ if and only $\bm{P}$ consists of identical one-hot vectors. 

Suppose that $p=\infty$.
It yields from the dual-valued spectral norm in \Cref{pro: dual matrix spectral norm} that 
\begin{align*}
\Vert\bm{P}\Vert_{S_\infty} = \Vert\bm{P}\Vert_2 = \Vert\bm{P}_{\rm{s}}\Vert_2+\lambda_{\max}(\bm{R})\epsilon,    
\end{align*}
where $\bm{R}:={\rm{sym}}\big(\bm{U}_1^\top\bm{P}_{\rm{i}}\bm{V}_1\big)$.
Besides, $\bm{U}_1\in\mathbb{R}^{m\times r_1}$ and $\bm{V}_1\in\mathbb{R}^{n\times r_1}$ are the left and right singular vectors of $\bm{P}_{\rm{s}}$ corresponding to the largest singular value with multiplicity $r_1$, respectively. 
It can be known that $\Vert\bm{P}_{\rm{s}}\Vert_2 = \Vert\bm{\sigma}\Vert_\infty\leq\Vert\bm{\sigma}\Vert_2\leq\sqrt{n}$ and $\Vert\bm{P}_{\rm{s}}\Vert_2=\sqrt{n}$ if and only if $\bm{P}_{\rm{s}}$ consists of identical one-hot column vectors.
Meanwhile, $r=1$, $\bm{\sigma} = [\sqrt{n},0,\cdots,0]^\top$ and $\bm{P}_{\rm{s}}=\sqrt{n}\bm{u}_1\bm{v}_1^\top$, where $\bm{u}_1$ and $\bm{v}_1$ represent the first columns of $\bm{U}$ and $\bm{V}$, respectively.
Thus, $\bm{R}=\frac{1}{2}(\bm{u}_1^\top\bm{P}_{\rm{i}}\bm{v}_1+\bm{v}_1^\top\bm{P}_{\rm{i}}^\top\bm{u}_1)=\bm{u}_1^\top\bm{P}_{\rm{i}}\bm{v}_1\in\mathbb{R}$ and $\lambda_{\max}(\bm{R})=\bm{u}_1^\top\bm{P}_{\rm{i}}\bm{v}_1 = \langle\bm{u}_1\bm{v}_1^\top,\bm{P}_{\rm{i}}\rangle=\frac{1}{\sqrt{n}}\langle\bm{P}_{\rm{s}},\bm{P}_{\rm{i}}\rangle$.
Therefore, $\Vert\bm{P}\Vert_{S_\infty}=\sqrt{n}+\frac{\langle\bm{P}_{\rm{s}},\bm{P}_{\rm{i}}\rangle}{\sqrt{n}}\epsilon$ and the maximum value of $\langle\bm{P}_{\rm{s}},\bm{P}_{\rm{i}}\rangle$ is zero if and only if $\bm{P}_{\rm{i}}=\bm{O}$.
In conclusion, $\Vert\bm{P}\Vert_{S_\infty}$ reaches its maximum value of $\sqrt{n}$ if and only if $\bm{P}$ has identical one-hot column vectors.

Hence, $\max\Vert\bm{P}\Vert_{S_p}=n^{\frac{1}{2}}$ for $2< p<\infty$ if and only if $\bm{P}$ consists of identical one-hot column vectors.

${\rm{(iv) (v)}}$ Given that $\min\Vert\bm{P}_{\rm{s}}\Vert_{S_p}=1$ for $1\leq p\leq\infty$ if and only if $\bm{P}_{\rm{s}}=\frac{1}{n}\mathds{1}\mathds{1}^\top$, as demonstrated in \cite[A.2.2 Proposition 5]{zhang2024dynamical}.
Meanwhile, $\bm{P}_{\rm{s}}$ has a compact SVD, $\bm{P}_{\rm{s}}=\bm{u}_1\bm{v}_1^\top$, where $\bm{u}_1=\bm{v}_1=\frac{1}{\sqrt{n}}\mathds{1}$.
In light of the total order of dual numbers, we then aim to minimize the infinitesimal part of $\Vert\bm{P}\Vert_{S_p}$.

$\textbf{Case 1:}$ $p=1$.
We obtain that 
\begin{align*}
\Vert\bm{P}\Vert_{S_1} &= \Vert\bm{P}\Vert_* = \Vert \bm{P}_{\rm{s}} \Vert_* + \big(\langle\bm{U}_r\bm{V}_r^\top,\bm{P}_{\rm{i}}\rangle + \Vert\widetilde{\bm{U}}_r^\top\bm{P}_{\rm{i}}\widetilde{\bm{V}}_r\Vert_*\big)\epsilon\\
&=1+\Big(\frac{1}{n}\langle\mathds{1}\mathds{1}^\top,\bm{P}_{\rm{i}}\rangle+ \Vert\widetilde{\bm{U}}_r^\top\bm{P}_{\rm{i}}\widetilde{\bm{V}}_r\Vert_*\Big)\epsilon=1+\Vert\widetilde{\bm{U}}_r^\top\bm{P}_{\rm{i}}\widetilde{\bm{V}}_r\Vert_*\epsilon,    
\end{align*}
where $\langle\mathds{1}\mathds{1}^\top,\bm{P}_{\rm{i}}\rangle=0$ is due to the fact that $\mathds{1}^\top\bm{P}_{\rm{i}}=\bm{0}^\top$.
Then the optimization problem is as follows,
\begin{align*}
\begin{array}{ll}
\min\limits_{\bm{P}_{\rm{i}}} & \Vert\widetilde{\bm{U}}_r^\top\bm{P}_{\rm{i}}\widetilde{\bm{V}}_r\Vert_* \\
s.t. & \mathds{1}^\top\bm{P}_{\rm{i}}=\bm{0}^\top,
\end{array}
\end{align*}
where $\big[\frac{1}{\sqrt{n}}\mathds{1},\widetilde{\bm{U}}_r\big]$ and $\big[\frac{1}{\sqrt{n}}\mathds{1},\widetilde{\bm{V}}_r\big]$ are orthogonal.
It is worth noting that $\Vert\widetilde{\bm{U}}_r^\top\bm{P}_{\rm{i}}\widetilde{\bm{V}}_r\Vert_*\geq0$ and $\Vert\widetilde{\bm{U}}_r^\top\bm{P}_{\rm{i}}\widetilde{\bm{V}}_r\Vert_*=0$ if and only if $\widetilde{\bm{U}}_r^\top\bm{P}_{\rm{i}}\widetilde{\bm{V}}_r=\bm{O}$.
On the one hand, 
\begin{align*}
\bm{O}=\widetilde{\bm{U}}_r\widetilde{\bm{U}}_r^\top\bm{P}_{\rm{i}}\widetilde{\bm{V}}_r=\Big(\bm{I}-\frac{1}{n}\mathds{1}\mathds{1}^\top\Big)\bm{P}_{\rm{i}}\widetilde{\bm{V}}_r=\bm{P}_{\rm{i}}\widetilde{\bm{V}}_r,    
\end{align*}
which derives that $\bm{P}_{\rm{i}}=\bm{w}\mathds{1}^\top$ for any $\bm{w}$ satisfying $\mathds{1}^\top\bm{w}=0$.
Conversely, such $\bm{P}_{\rm{i}}$ leads to $\widetilde{\bm{U}}_r^\top\bm{P}_{\rm{i}}\widetilde{\bm{V}}_r=\bm{O}$.
Consequently, the objective function $\Vert\widetilde{\bm{U}}_r^\top\bm{P}_{\rm{i}}\widetilde{\bm{V}}_r\Vert_*$ attains its minimum value of 0 if and only if $\bm{P}_{\rm{i}}$ has identical column vectors whose sum is zero.

In conclusion, the minimum value of $\Vert\bm{P}\Vert_{S_1}$ is 1 if and only if $\bm{P}=\frac{1}{n}\mathds{1}\mathds{1}^\top + \bm{w}\mathds{1}^\top\epsilon$ for any $\bm{w}\in\mathbb{R}^n$ satisfying $\mathds{1}^\top\bm{w}=0$.

$\textbf{Case 2:}$ $1<p<\infty$. 
It yields that 
\begin{align*}
\Vert\bm{P}\Vert_{S_p} = \Vert\bm{P}_{\rm{s}}\Vert_{S_p}+\frac{1}{\Vert\bm{P}_{\rm{s}}\Vert_{S_p}^{p-1}}\langle\bm{U}_r\bm{\Sigma}_r^{p-1}\bm{V}_r^\top,\bm{P}_{\rm{i}}\rangle\epsilon=1+\frac{1}{n}\langle\mathds{1}\mathds{1}^\top,\bm{P}_{\rm{i}}\rangle\epsilon=1.  
\end{align*}
Hence, the minimum value of $\Vert\bm{P}\Vert_{S_p}$ $(1<p<\infty)$ is 1 if and only if $\bm{P}=\frac{1}{n}\mathds{1}\mathds{1}^\top + \bm{P}_{\rm{i}}\epsilon$ for any $\bm{P}_{\rm{i}}$ satisfying $\mathds{1}^\top\bm{P}_{\rm{i}}=\bm{0}^\top$.

$\textbf{Case 3:}$ $p=\infty$.
It can be demonstrated that
\begin{align*}
\Vert\bm{P}\Vert_{S_\infty} &= \Vert\bm{P}\Vert_2 = \Vert\bm{P}_{\rm{s}}\Vert_2+\lambda_{\max}\Big(\frac{1}{2}(\bm{U}_1^\top\bm{P}_{\rm{i}}\bm{V}_1+\bm{V}_1^\top\bm{P}_{\rm{i}}^\top\bm{U}_1)\Big)\epsilon\\
&=1+\lambda_{\max}\Big(\frac{1}{n}\mathds{1}^\top\bm{P}_{\rm{i}}\mathds{1}\Big)\epsilon=1.    
\end{align*}
Thus, the minimum value of $\Vert\bm{P}\Vert_{S_\infty}$ is 1 if and only if $\bm{P}=\frac{1}{n}\mathds{1}\mathds{1}^\top + \bm{P}_{\rm{i}}\epsilon$ for any $\bm{P}_{\rm{i}}$ satisfying $\mathds{1}^\top\bm{P}_{\rm{i}}=\bm{0}^\top$.

In conclusion, we complete the proof.
\end{proof}

Accordingly, \Cref{pro: dual EI,pro: dual Schatten p norm and DTPM,the: dynamically reversible} have revealed that a DTPM $\bm{P}$ is dynamically reversible iff both $\Vert\bm{P}\Vert_{S_p}$ $(1\!\leq\! p\!<\!2)$ and ${\rm{EI_d}}(\bm{P})$ reach their maximum values.
Thus, it follows that $\Vert\bm{P}\Vert_{S_p}$ $(1\!\leq\! p\!<\!2)$ characterizes the approximate dynamical reversibility of $\bm{P}$ and ${\rm{EI_d}}(\bm{P})$ quantifies the extent of the causal effect of the corresponding dual Markov chain, defined by Qi and Cui in \cite{qi2024markov}.
Moreover, the validity of the indicator is verified in the following numerical experiments.

\subsection{Numerical Experiments}
In this section, the classical dumbbell Markov chain is employed to elucidate the pivotal role of the infinitesimal part of the dual-valued Ky Fan $p$-$k$-norm in investigating the optimal classification number of the system for the occurrence of CE.
We wrote our codes in Matlab R2024b on a personal computer with Intel(R) Core(TM) i7-10510U CPU @ 1.80GHz 2.30 GHz.
The codes are available at \url{https://github.com/TongWeiviolet/Causal-emergence-analysis.git}.

The prevailing approach to identifying CE involves determining an optimal classification number $k$, whereby the EI of the reduced TPM of size $k$-by-$k$ is significantly larger than that of the original TPM.
As evidenced in \cite{zhang2024dynamical}, this $k$ is selected as the clear cut-off in the singular values of the original TPM.
However, when the singular values decrease steadily, the visual determination of the cut-off becomes unreliable.

The highlight of our findings is the ability to select $k$ in a manner that allows the infinitesimal part of $\Vert\bm{P}\Vert_{(k,p)}$ for the DTPM $\bm{P}$ to reach a discernible maximum value by adjusting $p$ in $(1,2]$. 
The detailed experimental steps are described below.

The initial step is to construct a classical ``dumbbell'' Markov chain containing 25 states in each further ``weight'', 15 in each closer ``weight'', and 5 in the ``bar'', with a total of 85 nodes.
Then the random transition probabilities are specified within each weight and the bar, as well as between weights and the bar.
The dumbbell Markov chain is illustrated in \Cref{fig: dumbbell Markov chain}a, and its corresponding TPM, expressed as $\bm{M}\!\in\!\mathbb{R}^{85\times 85}$, is shown in the first graph of \Cref{fig: dumbbell Markov chain}b.
Besides, $\bm{M}\!\geq\!\bm{O}$ and $\mathds{1}^\top\bm{M}\!=\!\mathds{1}^\top$.

\begin{figure}[htpb]
\centering
\includegraphics[width=0.99\linewidth]{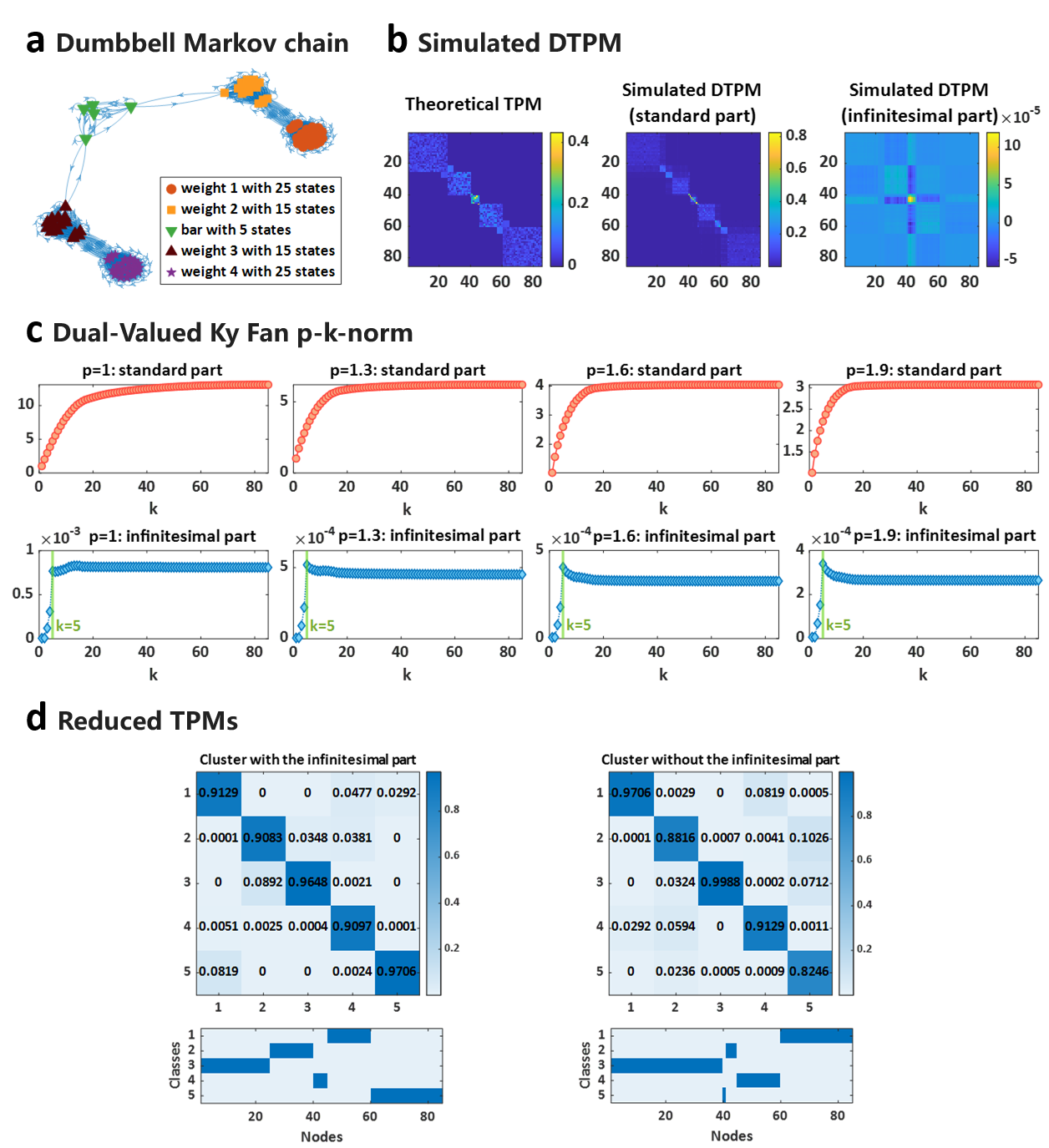}
\caption{Dumbbel Markov Chain}
\label{fig: dumbbell Markov chain}
\end{figure}

Subsequently, we intend to develop a DTPM and begin by constructing time-series data formulated by $\bm{x}_{t+1}\!=\!\bm{Mx}_{t}$, where the initial point of size 85-by-1 is randomly selected satisfying $\mathds{1}^\top\bm{x}_1\!=\!1$.
Thus, the time-series data $\bm{x}$ is obtained and utilized to yield two dual matrices $\bm{X}\!=\!\bm{X}_{\rm{s}}+\bm{X}_{\rm{i}}\epsilon\!\in\!\mathbb{DR}^{85 \times T}$ and $\bm{Y}\!=\!\bm{Y}_{\rm{s}}+\bm{Y}_{\rm{i}}\epsilon\!\in\!\mathbb{DR}^{85 \times T}$, where $\bm{X}_{\rm{s}}\!=\!\bm{x}(1\!:\!T)$, $\bm{X}_{\rm{i}}\!=\!\bm{x}(2\!:\!T+1)-\bm{x}(1\!:\!T-1)$, $\bm{Y}_{\rm{s}}\!=\!\bm{x}(2\!:\!T+1)$ and $\bm{Y}_{\rm{i}}\!=\!\bm{x}(3\!:\!T+2)-\bm{x}(2\!:\!T+1)$. 
Our following objective is to derive a DTPM $\bm{P}\!=\!\bm{P}_{\rm{s}}+\bm{P}_{\rm{i}}\epsilon\in\mathbb{DR}^{85\times85}$ such that $\bm{Y}\!=\!\bm{PX}$.
This entails solving two optimization problems 
in succession according to the total order of dual numbers, the former is
\begin{align*}
\begin{array}{cl}
\min\limits_{\bm{P}_{\rm{s}}\in\mathbb{R}^{85\times 85}} & \Vert\bm{Y}_{\rm{s}}-\bm{P}_{\rm{s}}\bm{X}_{\rm{s}}\Vert_F, \\
s.t. & \mathds{1}^\top\bm{P}_{\rm{s}}=\mathds{1}^\top,\;\bm{P}_{\rm{s}}\geq\bm{O}.
\end{array}
\end{align*}
Then the above optimal solution, designated as $\bm{P}_{\rm{s}}$, is applied to the latter one
\begin{align*}
\begin{array}{cl}
\min\limits_{\bm{P}_{\rm{i}}\in\mathbb{R}^{85\times 85}} & \Vert\bm{Y}_{\rm{i}}-\bm{P}_{\rm{s}}\bm{X}_{\rm{i}}-\bm{P}_{\rm{i}}\bm{X}_{\rm{s}}\Vert_F, \\
s.t. & \mathds{1}^\top\bm{P}_{\rm{i}}=\bm{0}^\top,\;[\bm{P}_{\rm{i}}]_{jk}\geq0,\;\forall (j,k)\in\Theta,
\end{array}  
\end{align*}
where $\Theta:=\{(j,k):[\bm{P}_{\rm{s}}]_{jk}<10^{-13}\}$.
Moreover, the Matlab package ``$\mathtt{cvx}$'' \cite{cvx,gb08} is utilized to solve these two optimization problems.
Hence, these two optimal solutions generate a simulated DTPM $\bm{P}=\bm{P}_{\rm{s}}+\bm{P}_{\rm{i}}\epsilon$ satisfying $\mathds{1}^\top\bm{P}=\mathds{1}^\top$ and $\bm{P}\geq\bm{O}$ from the original dumbbell Markov chain, as presented in the last two graphs of \Cref{fig: dumbbell Markov chain}b.

The further intention is to adjust the parameters $(k,p)$ of $\Vert\bm{P}\Vert_{(k,p)}$ to ascertain the optimal classification number of the system for the occurrence of CE. 
\Cref{fig: dumbbell Markov chain}c shows how $\Vert\bm{P}\Vert_{(k,p)}$ varies with $k$ when $p=1,1.3,1.6,1.9$.
The rationale behind selecting these specific values of $p$ is rooted in the fact that the dual-valued Schatten $p$-norm of a DTPM for $1\leq p<2$ characterizes the approximate dynamical reversibility of the DTPM, supported by the previous section.
Concerning the quantification proposed by Zhang \textit{et al.} in \cite{zhang2024dynamical}, the degree of vague CE for the TPM $\bm{P}_{\rm{s}}$ is expressed as
\begin{align*}
    \Delta\Gamma_p(k) = \frac{1}{k}\Vert\bm{P}_{\rm{s}}\Vert_{(p,k)}^p - \frac{1}{n}\Vert\bm{P}_{\rm{s}}\Vert_{S_p}^p,
\end{align*}
where $n$ equals the size of $\bm{P}_{\rm{s}}$ and $k$ is selected according to the relatively clear cut-off in the spectrum of singular values.
Nevertheless, the red lines in the first row of \Cref{fig: dumbbell Markov chain}c, representing $\Vert\bm{P}_{\rm{s}}\Vert_{(k,p)}$, exhibit a smooth increase as $k$ rises, which is difficult to find a suitable threshold $k$.
Conversely, in the second row of \Cref{fig: dumbbell Markov chain}c, an obvious maximum point is observed at $k=5$ for $p=1.3,1.6,1.9$, consistent with the number of categories of the original dumbbell Markov chain. 
Consequently, we adjust $p$ in the interval $[1,2)$ to identify $k$ corresponding to the maximum value of the infinitesimal part of $\Vert\bm{P}\Vert_{(k,p)}$, which is the optimal classification number.

Ultimately, we intend to develop a coarse-grained description of the original system and obtain a corresponding reduced TPM of size 5-by-5.
One such method is clustering all 85 nodes into 5 distinct groups with the infinitesimal part. 
Based on the CDSVD \cite{wei2024singular} of the DTPM $\bm{P}$, expressed as $\bm{P}\!=\!\bm{U\Sigma V}^\top$ with $\bm{U}\!=\!\bm{U}_{\rm{s}}+\bm{U}_{\rm{i}}\epsilon$, the k-means algorithm \cite{hartigan1979k} is employed to cluster all the columns of the concatenation $\bm{Q}_1\!=\!\big[\big(\bm{U}_{\rm{s}}(:,1\!:\!5)\big)\!^\top\bm{P}_{\rm{s}};\big(\bm{U}_{\rm{s}}(:,1\!:\!5)\big)\!^\top\bm{P}_{\rm{i}}+\big(\bm{U}_{\rm{i}}(:,1\!:\!5)\big)^\top\bm{P}_{\rm{s}}\big]\!\in\!\mathbb{R}^{10\times85}$ into 5 classes, thereby producing a projection matrix $\bm{\Phi}_1\!\in\!\mathbb{R}^{85\times5}$.
Besides, $\bm{\Phi}_1(\alpha,\beta)$ equals 1 if the $\alpha$-th column of $\bm{Q}_1$ belongs to the $\beta$-th class and 0 otherwise.
Thus, the reduced TPM $\bm{\Upsilon}_1\! =\! (\bm{\Phi}_1^\top\bm{P}_{\rm{s}}\bm{\Phi}_1)./{\rm{sum}}(\bm{\Phi}_1^\top\bm{P}_{\rm{s}}\bm{\Phi}_1)$ is developed.
The detailed presentation of $\bm{\Upsilon}_1$ and the corresponding concrete classifications are illustrated in the first column of \Cref{fig: dumbbell Markov chain}d.
Another traditional method is to cluster directly without the infinitesimal part.
Similarly, take advantage of the k-means algorithm to cluster the columns of $\bm{Q}_2\!=\!\big(\bm{U}_{\rm{s}}(:,1\!:\!5)\big)\!^\top\bm{P}_{\rm{s}}$ into 5 groups and to provide the projection matrix $\bm{\Phi}_2$.
Then the reduce TPM $\bm{\Upsilon}_2\!=\!(\bm{\Phi}_2^\top\bm{P}_{\rm{s}}\bm{\Phi}_2)./{\rm{sum}}(\bm{\Phi}_2^\top\bm{P}_{\rm{s}}\bm{\Phi}_2)$ is determined and displayed in the second column of \Cref{fig: dumbbell Markov chain}d, accompanied with the corresponding coarse-graining method from micro-states to macro-states.
A comparative analysis of the two methods has revealed that clustering with the infinitesimal part can augment extra available information, thereby enhancing the precision of the classification.

\section{Conclusions}\label{sec7:conclusions}
This paper provides an innovative method for constructing the dual continuation of real-valued functions concerning real matrices based on the G\^ateaux derivative.
Moreover, the general forms of dual-valued vector and matrix norms are derived and several characteristics are elucidated.
Besides, an intuitive overview of dual-valued functions proposed in this paper is outlined in \Cref{table: dual vector norms,table: dual matrix norms}.

Another core finding of our study is the identification of the optimal classification number for a system that occurs causal emergence based on the dual-valued Ky Fan $p$-$k$-norm.
We first construct a DTPM $\bm{P}$ by solving two optimization problems involving time-series data generated from the original system.
This follows by identifying what value of $k$ achieves the maximum value of the infinitesimal part of $\Vert\bm{P}\Vert_{(k,p)}$ after adjusting the value of $p$ in the interval $[1,2)$.
This $k$ is then regarded as the optimal number of classifications.
Finally, the infinitesimal part is concatenated and the k-means algorithm is employed to cluster the original nodes into $k$ groups, which provides the coarse-graining method and occurs causal emergence.

Although our proposed dual continuation of real-valued functions possesses a variety of prominent properties theoretically and practically, it still leaves room for improvement. 
For example, we can explore the dual continuation of real-valued and complex-valued functions for complex matrices. 
Moreover, our method, based on the dual-valued Ky Fan $p$-$k$-norm, can be applied to other time-series data and further ascertain the optimal classification number of the system.
This will be beneficial for scholars engaged in related research.

\begin{table*}[htpb]
\renewcommand{\arraystretch}{2}
\centering
\caption{Dual-Valued Vector Norms}
\vspace{-10pt}
\vspace{3mm}
\label{table: dual vector norms}
\resizebox{\textwidth}{!}{
\begin{tabular}{|c|c|r@{=}l|}
\hline
\textbf{Dual-Valued Vector Norms} & \textbf{Theorems} & \multicolumn{2}{c|}{\textbf{Formulas}} \\
\hline
Dual-valued vector 1-norm & \Cref{pro: dual vector p norm} \cref{dual vector 1-norm} & $\Vert\bm{x}\Vert_{1}$ & $\Vert\bm{x}_{\rm{s}}\Vert_1+ \Big(\langle {\rm{sign}}(\bm{x}_{\rm{s}}),\bm{x}_{\rm{i}}\rangle+\sum_{k\in{\rm{supp}}^c(\bm{x}_{\rm{s}})}\big|x_{\rm{i}}^k\big|\Big)\epsilon.$\\
\hline
\makecell[c]{Dual-valued vector $p$-norm\\$(1<p<\infty)$} & \Cref{pro: dual vector p norm} \cref{eq: dual vector p-norm} & $\Vert\bm{x}\Vert_{p}$ & $\left\{\begin{array}{ll}
\Vert\bm{x}_{\rm{s}}\Vert_p+ \frac{\langle |\bm{x}_{\rm{s}}|^{p-2}\odot\bm{x}_{\rm{s}},\bm{x}_{\rm{i}}\rangle}{\Vert\bm{x}_{\rm{s}}\Vert_p^{p-1}}\epsilon,  &  \bm{x}_{\rm{s}}\neq\bm{0},\\
\Vert\bm{x}_{\rm{i}}\Vert_p\epsilon, & \bm{x}_{\rm{s}}=\bm{0}.
\end{array}\right.$\\
\hline
Dual-valued vector $\infty$-norm & \Cref{pro: dual vector p norm} \cref{dual vector infty-norm} & $\Vert\bm{x}\Vert_{\infty}$ & $\left\{\begin{array}{ll}
\Vert\bm{x}_{\rm{s}}\Vert_\infty+ \max\limits_{k\in\mathcal{I}}\big({\rm{sign}}(x_{\rm{s}}^k)x_{\rm{i}}^k\big)\epsilon,  & \bm{x}_{\rm{s}}\neq\bm{0}, \\
\Vert\bm{x}_{\rm{i}}\Vert_\infty\epsilon, &  \bm{x}_{\rm{s}}=\bm{0}.
\end{array}\right.$\\
\hline
\end{tabular}}
\end{table*}

\begin{table*}[htpb]
\caption{Dual-Valued Functions of Dual Matrices}
\vspace{-10pt}
\renewcommand{\arraystretch}{2}
\centering
\vspace{3mm}
\label{table: dual matrix norms}
\resizebox{\textwidth}{!}{
\begin{tabular}{|c|c|r@{=}l|}
\hline
\textbf{Dual-valued matrix norms} & \textbf{Theorems} & \multicolumn{2}{c|}{\textbf{Formulas}} \\
\hline
\makecell[c]{Dual-valued Ky Fan $p$-$k$-norm\\$(1<p<\infty)$} & \Cref{pro: dual ky fan p-k-norm} & $\Vert\bm{A}\Vert_{(k,p)}$ & $\left\{\begin{array}{ll}
\Vert\bm{A}_{\rm{s}}\Vert_{(k,p)}+\frac{\langle\bm{U}_1\bm{\Sigma}_1^{p-1}\bm{V}_1^\top,\bm{A}_{\rm{i}}\rangle+\sigma_k^{p-1}\sum_{\ell=1}^t\lambda_\ell(\bm{M})}{\Vert\bm{A}_{\rm{s}}\Vert_{(k,p)}^{p-1}}\epsilon, & \bm{A}_{\rm{s}}\neq\bm{O}, \\
\Vert\bm{A}_{\rm{i}}\Vert_{(k,p)}\epsilon, & \bm{A}_{\rm{s}}=\bm{O}.
\end{array}\right.$\\
\hline
Dual-valued Ky Fan $k$-norm & \Cref{pro: dual ky fan} & $\Vert\bm{A}\Vert_{(k)}$ & $\left\{\begin{array}{ll}
\Vert\bm{A}_{\rm{s}}\Vert_{(k)}+\left(\langle\bm{U}_1\bm{V}_1^\top,\bm{A}_{\rm{i}}\rangle+\sum_{\ell=1}^t\lambda_\ell(\bm{M})\right)\epsilon, & \bm{A}_{\rm{s}}\neq\bm{O},\sigma_k\neq0,\\
\Vert\bm{A}_{\rm{s}}\Vert_{(k)}+\left(\langle\bm{U}_1\bm{V}_1^\top,\bm{A}_{\rm{i}}\rangle+\sum_{\ell=1}^t\sigma_\ell(\bm{N})\right)\epsilon, & \bm{A}_{\rm{s}}\neq\bm{O},\sigma_k=0,\\
\Vert\bm{A}_{\rm{i}}\Vert_{(k)}\epsilon, & \bm{A}_{\rm{s}}=\bm{O}.
\end{array}\right.$\\
\hline
Dual-valued spectral norm & \Cref{pro: dual matrix spectral norm}  & $\Vert\bm{A}\Vert_2$ & $ \left\{\begin{array}{ll}
\Vert \bm{A}_{\rm{s}} \Vert_2 + \lambda_{\rm{max}}(\bm{R})\epsilon, & \bm{A}_{\rm{s}}\neq \bm{O}, \\
\Vert \bm{A}_{\rm{i}} \Vert_2\epsilon, & \bm{A}_{\rm{s}}=\bm{O}.
\end{array}\right.$\\
\hline
\makecell[c]{Dual-valued Schatten $p$-norm\\$(1<p<\infty)$} & \Cref{cor: dual schatten norm} & $\Vert\bm{A}\Vert_{S_p}$ & $\left\{\begin{array}{ll}
\Vert \bm{A}_{\rm{s}} \Vert_{S_p}+\frac{1}{\Vert\bm{A}_{\rm{s}}\Vert_{S_p}^{p-1}}\langle \bm{U}_r\bm{\Sigma}_r^{p-1}\bm{V}_r^\top,\bm{A}_{\rm{i}}\rangle\epsilon, & \bm{A}_{\rm{s}}\neq\bm{O},\\
\Vert\bm{A}_{\rm{i}}\Vert_{S_p}\epsilon, & \bm{A}_{\rm{s}}=\bm{O}.
\end{array}\right.$\\
\hline
Dual-valued nuclear norm & \Cref{cor: dual nuclear norm} & $\Vert\bm{A}\Vert_*$ & $ \left\{\begin{array}{ll}
\Vert \bm{A}_{\rm{s}} \Vert_* + \Big(\big\langle\bm{U}_r\bm{V}_r^\top,\bm{A}_{\rm{i}}\big\rangle + \big\Vert\widetilde{\bm{U}}_r^\top\bm{A}_{\rm{i}}\widetilde{\bm{V}}_r\big\Vert_*\Big)\epsilon, &  \bm{A}_{\rm{s}}\neq\bm{O},\\
\Vert \bm{A}_{\rm{i}} \Vert_*\epsilon, &  \bm{A}_{\rm{s}}=\bm{O}.
\end{array}\right.$\\
\hline
Dual-valued Frobenius norm & \Cref{cor: dual Frobenius norm} & $\Vert\bm{A}\Vert_F$ & $  \left\{\begin{array}{ll}
\Vert \bm{A}_{\rm{s}} \Vert_F+\frac{\langle \bm{A}_{\rm{s}},\bm{A}_{\rm{i}}\rangle}{\Vert\bm{A}_{\rm{s}}\Vert_F}\epsilon, &  \bm{A}_{\rm{s}}\neq\bm{O},\\
\Vert\bm{A}_{\rm{i}}\Vert_F\epsilon, & \bm{A}_{\rm{s}}=\bm{O}.
\end{array}\right.$\\
\hline
Dual-valued operator 1-norm & \Cref{pro: dual operator 1-norm} & $\Vert\bm{A}\Vert_1$ & $\left\{\begin{array}{ll}
\max_{1\leq t\leq n}\Vert\bm{a}_t\Vert_1, & \bm{A}_{\rm{s}}\neq\bm{O},\\
\Vert\bm{A}_{\rm{i}}\Vert_1, & \bm{A}_{\rm{s}}=\bm{O}.
\end{array}\right.$\\
\hline
Dual-valued operator $\infty$-norm & \Cref{pro: dual operator infty norm} & $\Vert\bm{A}\Vert_\infty$ & $\left\{\begin{array}{ll}
\max_{1\leq t\leq m}\Vert\bm{b}_t\Vert_1, & \bm{A}_{\rm{s}}\neq\bm{O},\\
\Vert\bm{A}_{\rm{i}}\Vert_\infty, & \bm{A}_{\rm{s}}=\bm{O}.
\end{array}\right.$\\
\hline
Dual-valued trace & \Cref{pro: trace} & ${\rm{tr}}(\bm{A})$ & ${\rm{tr}}(\bm{A}_{\rm{s}})+{\rm{tr}}(\bm{A}_{\rm{i}})\epsilon$.\\
\hline
Dual-valued determinant & $\Cref{pro: det}$ & ${\rm{det}}(\bm{A})$ & ${\rm{det}}(\bm{A}_{\rm{s}}) + \big\langle({\rm{adj}}\,\bm{A}_{\rm{s}})^\top,\bm{A}_{\rm{i}}\big\rangle\epsilon$.\\
\hline
\end{tabular}}
\end{table*}

\bibliographystyle{siamplain}
\bibliography{reference}
\end{document}

%% file: ex_shared.tex
\usepackage{amsmath}
\usepackage{amsfonts}
\usepackage{lipsum}
\usepackage{graphicx}
\usepackage{epstopdf}
\usepackage{algorithmic}
\usepackage{braket,amsopn} 

\usepackage{array} 
\usepackage[caption=false]{subfig}
\usepackage{graphicx,epstopdf} 
\usepackage[caption=false]{subfig}

\usepackage{algorithmic} 
\Crefname{ALC@unique}{Line}{Lines}

\usepackage{bm}

\usepackage{enumerate}
\usepackage{diagbox}

\usepackage{extarrows}
\usepackage{stmaryrd}
\usepackage{dsfont}
\usepackage{cases}
\usepackage{amssymb}
\usepackage{makecell}

\ifpdf
  \DeclareGraphicsExtensions{.eps,.pdf,.png,.jpg}
\else
  \DeclareGraphicsExtensions{.eps}
\fi

\newsiamremark{remark}{Remark}
\newsiamremark{hypothesis}{Hypothesis}
\crefname{hypothesis}{Hypothesis}{Hypotheses}
\newsiamthm{claim}{Claim}
\newsiamremark{example}{Example}

\headers{Dual-Valued Functions}{Tong Wei, Weiyang Ding, and Yimin Wei}

\title{Dual-Valued Functions of Dual Matrices with Applications in Causal Emergence
\thanks{Received by the editors November 12, 2024.
\funding{This work is partially supported by the Science and Technology Commission of Shanghai Municipality (No. 23ZR1403000) and the National Natural Science Foundation of China (No. 12471481). Y. Wei is supported by the Science and Technology Commission of Shanghai Municipality under grant 23JC1400501.}
}
}

\author{Tong Wei\thanks{Institute of Science and Technology for Brain-Inspired Intelligence, Fudan University, Shanghai, P.R. China. (\email{20110850008@fudan.edu.cn, weitong0802kun@163.com}).}
\and Weiyang Ding\thanks{Corresponding author, Institute of Science and Technology for Brain-Inspired Intelligence, Fudan University, Shanghai, P.R. China. (\email{dingwy@fudan.edu.cn}).}
\and Yimin Wei\thanks{School of Mathematical Sciences and Key Laboratory of Mathematics for Nonlinear Sciences, Fudan University, Shanghai, P.R. China. 
(\email{ymwei@fudan.edu.cn}).}
}

\usepackage{amsopn}
